\documentclass[final]{siamltex}
\usepackage{}
\usepackage{amsfonts}
\usepackage{bbm}
\usepackage{amsmath}
\usepackage{mathrsfs}
\usepackage{graphicx}
\usepackage{epstopdf}
\usepackage{caption}
\usepackage{siunitx}
\usepackage{booktabs}
\usepackage{algorithm}
\usepackage{algpseudocode}
\usepackage{caption}
\usepackage{subcaption}
\usepackage{geometry}
\usepackage{cases}

\graphicspath{{FIG/}}
\usepackage{amsmath,amssymb,enumerate,color}

\usepackage{comment,enumerate,multicol,xspace}

\title{Deep forward-backward dynamic programming schemes for high-dimensional semilinear nonlocal PDEs and FBSDE with jumps\thanks{This work was supported by grants from the National Natural Science
Foundation of China (Grant Nos. 12271367, 92470119).}}

\author{Wansheng Wang\thanks{Corresponding author. Department of Mathematics, Shanghai Normal University, Shanghai, 200234, China ({\tt w.s.wang@163.com}).}
\and Jiangtao Pan\thanks{Department of Mathematics, Shanghai Normal University, Shanghai, 200234, China.}
\and Jie Wang\thanks{Department of Mathematics, Shanghai Normal University, Shanghai, 200234, China.}
\and Zaijun Ye\thanks{Department of Mathematics, Shanghai Normal University, Shanghai, 200234, China.}
}

\begin{document}

\maketitle

\begin{abstract}
We propose a new deep learning algorithm for solving high-dimensional parabolic integro- differential equations (PIDEs) and forward-backward stochastic differential equations with jumps (FBSDEJs). This novel algorithm can be viewed as an extension and generalization of the DBDP2 scheme and a dynamic programming version of the forward-backward algorithm proposed recently for high-dimensional {semilinear PDEs and semilinear PIDEs}, respectively. Different from the DBDP2 scheme for  {semilinear PDEs}, our algorithm approximate simultaneously the solution and the integral kernel by deep neural networks, while the gradient of the solution is approximated by numerical differential techniques. The related error estimates for the integral kernel approximation play key roles in deriving error estimates for the novel algorithm. Numerical experiments confirm our theoretical results and verify the effectiveness of the proposed methods.
\end{abstract}

\begin{keywords}
parabolic integro- differential equations, forward-backward stochastic differential equations with jumps, deep learning, DFBDP scheme, error estimates
\end{keywords}

{\noindent \textbf{MSC codes.} 60H35, 65C20, 65M15, 65C30, 60H10, 65M75}

\pagestyle{myheadings} \thispagestyle{plain} \markboth{W. S. WANG, J. T. PAN, J. Wang, Z. J. Ye}{DFBDP schemes for PIDEs and FBSDEJs}

%%%%%%%%%%%%%%%%%%%%%%%%%%%%%%%%%%%%% 第1章 介绍 %%%%%%%%%%%%%%%%%%%%%%%%%%%%%%%%%%%%%%%%

\section{Introduction}  In this paper we propose a new deep learning algorithm for solving high-dimensional partial integro-differential equations (PIDEs)
\begin{eqnarray}\label{eq1.1}
\left\{
\begin{aligned}
\partial_{t}u+\mathcal L u+f\left(t,x,u,\sigma^\mathsf{T}(t,x)\nabla_{x}u,B[u]\right) &=0, \quad \;(t,x) \in [0,T)\times \mathbb R^d,\\
u(T,x)&=g, \quad \;x\in \mathbb R^d,
\end{aligned}
\right.
\end{eqnarray}
where $d\ge 1$ and $T>0$, $g$: $\mathbb R^d \rightarrow \mathbb R$ is terminal condition, the second-order nonlocal operator $\mathcal L$ is defined as follows:
\begin{eqnarray*}
\begin{aligned}
\mathcal Lu=&\frac{1}{2}{\rm Tr}\left(\sigma\sigma^\mathsf{T}(t,x)\partial_{x}^{2}u\right)+\left\langle b(t,x),\nabla_{x} u\right\rangle\\
&+\int_{E}(u(t,x+\beta(t,x,e))-u(t,x)-\left\langle\nabla_{x} u,\beta(t,x,e)\right \rangle)\lambda(\mathrm{d}e),
\end{aligned}
\end{eqnarray*}
and $B$ is an integral operator
\begin{eqnarray*}
B[u]=\int_{E}(u(t,x+\beta(t,x,e))-u(t,x))\gamma(e)\lambda(\mathrm{d}e).
\end{eqnarray*}
Here $b(t,x) $: $ [0,T] \times \mathbb R^d  \to \mathbb R^d$, $\sigma(t,x)$: $[0,T] \times \mathbb R^d \to \mathbb R^{d \times d} $, $ \beta(t,x,e) $: $[0,T] \times \mathbb R^d \times E \to \mathbb R^d $, and $f$: $[0,T] \times \mathbb R^d \times \mathbb R \times \mathbb R^{d} \times \mathbb R \rightarrow \mathbb R$ are deterministic and Lipschitz continuous functions of linear growth which are additionally supposed to satisfy some monotonicity conditions, $A^{\mathsf{T}}$ denotes the transpose of a vector or matrix $A$, $ E \triangleq \mathbb R^{d} \backslash \{0\} $ is equipped with its Borel field $\mathcal E$ for some measurable functions $\gamma:  E\rightarrow \mathbb R $ satisfying
\begin{eqnarray}\label{eq1.2}
\sup\limits_{e \in E}|\gamma(e)| \leq K_{\gamma},
\end{eqnarray}
and $\lambda(\mathrm{d}e)$ is assumed to be a $\sigma $-finite measure on $ (E,\mathcal E ) $ satisfying
\begin{eqnarray*}
\int_{E} (1 \land |e|^2) \lambda(\mathrm{d}e) < \infty.
\end{eqnarray*}
It is well-known that the solution to PIDEs (\ref{eq1.1}) can be represented by high-dimensional forward-backward stochastic differential equations with jumps (FBSDEJs), because of the generalized nonlinear Feynman-Kac formula \cite{GRE}. As two ubiquitous mathematical models, PIDEs and FBSDEJs have been widely used in various practical science and engineering applications, for example, anomalous diffusion \cite{RJ}, contaminant migration \cite{JWGLL} in groundwater and plasma physics, random control problems \cite{JZ,SX}, option pricing \cite{NSM}, risk measure \cite{S,MQA}, portfolio hedge \cite{AE}, and index utility maximization \cite{DB} in finance. Jump behavior in many financial problems are also included in the FBSDEJs, such as buyer (seller) default, operation failure, insurance events \cite{AE,LD,EN}, and so on.

Due to the complex solution structure, however, explicit solutions of PIDEs and FBSDEJs can seldom be found. Consequently, one usually resorts to numerical methods to solve the two kinds of equations, and a volume of work has been performed on their numerical solutions.
Forward-backward stochastic differential equation (FBSDEs) without jump term has been fully studied \cite{EN,CBJ,BN,CJ1,DK,YS,JZDJJ,JZZ,WLS,WWL} by different numerical methods, include four-step method \cite{JPJ}, Runge-Kutta methods (see, e.g., \cite{JFD}), high-order multi-step methods (see, e.g., \cite{WYT}). However, there are few numerical methods for FBSDEJs. Picard iteration, Multistep and prediction-correction schemes have been used to solve low-dimensional FBSDEJs (see, for example, \cite{AES,BR,WFT,WZG,YJW}). In 1997, Barles, Buckdahn, and Pardoux \cite{GRE} connected the FBSDEJs to a class of nonlinear PIDEs through the nonlinear Feynman-Kac theory. Therefore, FBSDEJs becomes a powerful probabilistic technique that is used to study the numerical solutions of PIDEs, and usually the numerical algorithms for the design of one class of equations in these two classes of equations can solve another class of equations. To solve low-dimensional PIDEs, IMEX time discretizations combined with finite difference method, finite element method, or spectral method, have been studied in the literature (see, for example, \cite{Achdou05,Pindza14,Kadalbajoo17,Wang19,WMZ2021,Mao22}).

As the dimensionality increases, traditional grid-based numerical methods are no longer applicable for higher-dimensional problems, where their computational complexity grows exponentially and leads to the so-called ``curse of dimensionality" \cite{RB}. Therefore, the solution of high-dimensional problems has always been a challenge for scientists. Recently, based on the Feyman-Kac representation of the PDEs, branch diffusion process method  and Monte Carlo method  have been studied; see, for example, \cite{Gobet05,HL2012,PNT2016,Warina18}.

Over the past few years, machine learning and deep learning algorithms have attracted more and more scholars' attention, since these approximation methods provide a new idea for numerical approximation of high-dimensional functions and overcome the problem of ``curse of dimensionality" (see, for example, \cite{Chan19,Jentzen23}). With multilevel techniques and automatic differentiation, multi-layer Picard iterative methods have been developed for handling some high-dimensional PDEs with nonlinearity (see, for example, \cite{WMAT19,MATTP18,Hutzenthler20}). Using machine learning representation of the solution, the so-called Deep Galerkin method has proposed to solve PDEs on a finite domain in \cite{JK2018}. On basis of the backward stochastic differential equation (BSDE) approach first developed in \cite{PP1990}, by viewing the BSDE as a stochastic control problem with the gradient of the solution being the policy function, a neutral network method was proposed to solve high-dimensional PDEs in the pioneering papers \cite{WEM18,WJA}. {Given the universal approximation capability of neural networks, a posteriori error estimation of the solution is provided and it is proved that the error converges to zero for the deep BSDE method (see \cite{HL2020})}. A new deep learning algorithm was proposed to solve fully nonlinear PDEs and nonlinear second-order backward stochastic differential equations (2BSDE) by exploiting a connection between PDEs and 2BSDEs \cite{Beck19}. For linear and semilinear PIDEs, the deep learning approximations of the numerical solution were proposed in \cite{Gonon21a,Gonon21b} and \cite{Castro21}. {An extension of the deep solver of \cite{WEM18} to the case of FBSDEJs and corresponding PIDEs was proposed in \cite{AGPP25}, and the error arising in the numerical approximation of FBSDEs and corresponding PIDEs by means of a deep learning-based method was studied in \cite{GOP25}}. {Moreover, a Deep BSDE scheme for nonlinear integro-PDEs with unbounded nonlocal operators was introduced and analyzed in \cite{JM24}}. {Quite recently, a deep learning algorithm for solving high-dimensional PIDEs and high-dimensional FBSDEJs was proposed in \cite{Wang2025}, and the error estimates for this deep learning algorithm were derived in this work.} {Various machine learning solvers for coupled forward-backward systems of stochastic differential equations (FBSDEs) driven by a Brownian motion and a Poisson random measure were investigated in \cite{ABDW24}}. It is worth noting that deep learning backward dynamic programming (DBDP) methods, including DBDP1 scheme and DBDP2 scheme, in which some machine learning techniques are used to estimate simultaneously the solution and its gradient by minimizing a loss function on each time step, were proposed in \cite{CHX}. The DBDP1 algorithm has then been extended to the case of semilinear parabolic nonlocal integro-differential equations in \cite{Castro21}. In this paper, we apply the DBDP2 algorithm to solve semilinear PIDEs and decoupled FBSDEJs and carry out the corresponding convergence analysis for this deep forward-backward dynamic programming (DFBDP) scheme. The main difference between the two algorithms is in the gradient estimation of the solution: in the DBDP1 algorithm it is directly estimated by the neural network and in the DBDP2 scheme it is estimated by using numerical differential techniques.

The outline of the paper is organized as follows. In Section 2, we start by introducing the deep learning-based DFBDP algorithm for FBSDEJs and related PIDEs. Section 3 is devoted to the error estimates for the DFBDP scheme. Before we do this, some assumptions on the data are also given in this section. Two numerical examples are presented in Section 4 and verify our theoretical results. In Section 5 we finally conclude with some remarks.

%%%%%%%%%%%%%%%%%%%%%%%%%%%%%%%%%%%%%%第2章 SBSDEJs,DFBDP算法%%%%%%%%%%%%%%%%%%%%%%%%%%%%%%%%%%%%%%

\section{Deep learning-based schemes for semilinear PIDEs and decoupled FBSDEJs} In this section, we introduce the details about deep learning-based schemes for solving decoupled FBSDEJs and the associated semilinear PIDEs.

\subsection{Semilinear PIDEs and FBSDEJs} Let $|\cdot|$ denote the Euclidean norm in the Euclidean space, and $C^{l,k}$ denote the set of functions $f(t,x)$ with continuous partial derivatives up to $l$ with respect to $t$ and up to $k$ with respect to $x$. Let $(\Omega,\mathcal{F}, \mathbb F,P)$, $\mathbb F=(\mathcal{F}_{t})_{0\leq t<T}$, be a stochastic basis such that $ \mathcal{F}_{0} $ contain all zero $P$-measure sets, and $ \mathcal{F}_{t^+} \triangleq \bigcap_{\epsilon > 0}\mathcal{F}_{t+\epsilon}=\mathcal{F}_{t} $. The filtration $\mathbb F$ is generated by a $d$-dimension Brownian motion (BM) $ \{ W_{t} \} _{0\leq t<T} $ and a Poisson random measure $\mu$  on $\mathbb R_+\times E$, independent of $W$.

Let $ u(t,x) \in C^{1,2}([0,T] \times \mathbb R^d) $ be the unique viscosity solution of (\ref{eq1.1}). Then by the It\^o formula, one can show that the solution $u$ admits a probabilistic representation, i.e., we have (see \cite{RJ}), %and $ \mathcal L $ be the generator associated to a jump-diffusion process:
%\begin{eqnarray}\label{eq2.2}
% X_t=\xi+\int_{0}^{t}b(s,X_{s})ds+\int_{0}^{t}\sigma(s,X_{s})dW_{s}+\int_{0}^{t}\int_{E}\beta(X_{s^-},e)\tilde{\mu}(de,ds),
%\end{eqnarray}
%where $\xi:=X_0$, $\{\tilde{\mu}(A,[0,t])\}_{t\ge 0}$ is a martingale for all $A\in \mathcal E$ satisfying $\lambda(A)<\infty$ with $\tilde{\mu}(de,dt) =\mu(de,dt)-\nu(de,dt)$. As usual, $X_{s^-}$ denotes the a.e. limit of $X_t$ as $t\uparrow s$. For the viscosity solution $u(t,x)$, we can consider the stochastic process
\begin{eqnarray}\label{eq2.1}
u(t,\mathcal{X}_{t})=Y_{t}.
\end{eqnarray}
Furthermore, the quadruplet $(\mathcal{X}_{t},Y_{t},Z_{t},\Gamma_{t}) $ with following relationship
%which can be expressed in the following form by It{\^{o}} formula
%\begin{eqnarray}\label{eq2.3}
%\begin{aligned}
%u(t,X_{t})=&u(t_0,X_{t_0})+\int^t_{t_0}\frac{\partial u}{\partial  t}(t,X_{s}) ds +\int^t_{t_0}\mathcal Lu(s,X_s)ds\\
%&+\sum\limits_{i,j=1}^d\int^t_{t_0}\frac{\partial u}{\partial  x^i}\sigma^{i,j}(s,X_{s})dW_s^j\\ &+\int^t_{t_0}\int_{E}(u(s,X_{s^{-}}+\beta(s,X_{s^{-}},e))-u(s,X_{s^{-}}))\tilde{\mu}(de,ds).
%\end{aligned}
%\end{eqnarray}
%Now, with introducing the stochastic process
\begin{eqnarray}\label{eq2.2}
\left\{
\begin{aligned}
&Z_{t}=\sigma^T(t,\mathcal{X}_{t})\nabla_{x}u(t,\mathcal{X}_{t}),\\
&U_{t}=u\left(t,\mathcal{X}_{t^{-}}+\beta(t,\mathcal{X}_{t^{-}},e)\right)-u(t,\mathcal{X}_{t^{-}}),\\
&\Gamma_t=B[u(t,\mathcal{X}_t)],
\end{aligned}
\right.
\end{eqnarray}
 is the solution of the FBSDEJs
 \begin{eqnarray}\label{eq2.3}
 \left\{
\begin{aligned}
 \mathcal{X}_t &=\xi+\int_{0}^{t}b(s,\mathcal{X}_{s})\mathrm{d}s+\int_{0}^{t}\sigma(s,\mathcal{X}_{s})\mathrm{d}W_{s}+\int_{0}^{t}\int_{E}\beta(s,\mathcal{X}_{s^-},e)\widetilde{\mu}(\mathrm{d}e,\mathrm{d}s),\\
 Y_{t}&=g(\mathcal{X}_{T})+\int_{t}^{T}f(s,\mathcal{X}_s,Y_{s},Z_{s},\Gamma_{s})\mathrm{d}s-\int_{t}^{T}Z_{s}^{T}\mathrm{d}W_{s}-\int_{t}^{T}\int_{E}U_{s}(e)\widetilde{\mu}(\mathrm{d}e,\mathrm{d}s),
 \end{aligned}
 \right.
  \end{eqnarray}
where compensator $\tilde\mu(\mathrm{d}e,\mathrm{d}t)=\nu(\mathrm{d}e,\mathrm{d}t)-\lambda(\mathrm{d}e) \mathrm{d}t$, $B[u]$ is an integral operator. The quadruplet $(\mathcal{X}_{t},Y_{t},Z_{t},\Gamma_{t}) $ are called the ``forward part'', the ``backward part'', the ``control part'' and the ``jump part'', respectively. The presence of the control part $Z_t$ is crucial to find a nonanticipative solution.

{By nonlinear Feynman-Kac formulas, the solution to the PIDE (\ref{eq1.1}) can be represented as a conditional expectation:
\begin{equation}\label{eq2.3b}
u(t,x)=\mathbb{E}\left[g(\mathcal{X}_{T})+\int_{t}^{T}f(s,\mathcal{X}_{s},Y_{s},Z_{s},\Gamma_{s})\mathrm{d}s\bigg|\mathcal{F}_{t}\right].
\end{equation}}
The above formulas (\ref{eq2.3b}) are the nonlinear Feynman-Kac formulas, and such formulas indicate an interesting relationship between solutions of FBSDEJs and PIDEs. Using nonlinear Feynman-Kac formulas (\ref{eq2.3b}) and the decoupled FBSDEJs, the deep learning algorithm DBDP1 proposed in \cite{CHX} has been extended to the semilinear PIDEs in \cite{Castro21}. Note that the FBSDEJs (\ref{eq2.3}) is decoupled since the coefficients $b$, $\sigma$, and $\beta$ are independent of $Y_t$ and can be solved in sequence. When the coefficients $b$, $\sigma$, and $\beta$ depend on $Y_t$, the FBSDEJs is called coupled and its deep learning-based approximation has been investigated in \cite{Wang2025}.

\subsection{Deep learning (DL) and function approximation} In this subsection, we give a brief introduction about Deep learning (DL). After the authors of \cite{Warreb43,Rosenblatt58} introduced the concept of a neural networks (NN) in 1943 and 1958, respectively, the use of NNs  as a way to approximate functions started to gain importance for its well performance in applications. A rigorous theoretical justification of the approximation property that under suitable conditions on the approximated functions, NNs have a very good performance measured in some mathematical terms, further promotes their application in a large variety of disciplines in science (see, e.g., \cite{Wang17,Md19}). In 2006, the Hinton lab solved the training problem for Deep neural networks (DNNs) \cite{Hinton061,Hinton062}. Since then, a variety of DL algorithms are increasingly emerging. Due to their composition, many layers of DNNs are more capable of representing highly varying nonlinear functions compared to shallow learning approaches. Moreover, DNNs are more efficient for learning because of the combination of feature extraction and classification layers \cite{Md19}. Deep feedforward neutral network is the simplest neural network, but it is sufficient for most PDE problems. Since it is a class of universal neural network, we consider Deep feedforward neural network in this paper.

 Let $m_{\ell}(\ell=0,\ldots,L)$ be the number of neurous in the $\ell-$th layers, $L$ is layer of neural network. The first layer is the input layer, the last layer is the output layer, and another layers are the hidden layers. A feedforward neural network can be defined as the composition
 \begin{eqnarray}\label{eq2.4}
 x \in \mathbb R^{d}\to \mathcal{N}_{L}   \circ \mathcal{N}_{L-1}\circ... \circ\mathcal{N}_{1}(x)\in \mathbb R^{d_{1}},
 \end{eqnarray}
where $d$ is the dimension of $x$ and the output dimension $ d_{1}=k, k \in \mathbb R_+$. We fix $d_{1}=1$ in this paper, and
\begin{eqnarray}\label{eq2.5}
\left\{
\begin{aligned}
&\mathcal{N}_{0}(x)=x\in \mathbb R^{d},\\
&\mathcal{N}_{\ell}(x)=\varrho(\mathbf{w}_{\ell}\mathcal{N}_{\ell-1}(x)+\mathbf{b}_{\ell})\in \mathbb R^{m_{\ell}}, \quad {\hbox{for}} \quad 1\leq \ell \leq L-1,\\
&\mathcal{N}_{L}(x)=\mathbf{w}_{L}\mathcal{N}_{L-1}(x)+\mathbf{b}_L \in \mathbb R,
\end{aligned}
\right.
\end{eqnarray}
where $ \mathbf{w}_{\ell}\in \mathbb R^{m_{\ell}\times m_{\ell-1}}$ and $\mathbf{b}_{\ell} \in \mathbb R^{m_{\ell}}$ denote the weight matrix and bias vector, respectively, $\varrho$ is a nonlinear activation function such as the logistic sigmoid function, the hyperbolic tangent ($\tanh$) function, the rectified linear unit (relu) function and other similar functions. We use the (tanh) function for the hidden layer in this paper. The final layers $\mathcal{N}_{L}(x) $ is typically linear.

 Let $\theta$ denote the parameters of the neural network:
 \begin{eqnarray*}
 	\theta:=\{\mathbf{w}^{\ell}, \mathbf{b}^{\ell}\}, \qquad{\ell=1,...,L}.
 \end{eqnarray*}
 The DNN is trained by optimizing over the parameters $\theta$ by (\ref{eq2.4}).

 Considering that our work is to apply the DBDP2 algorithm to solve semilinear PIDEs and related FBSDEJs, we should consider neural networks with one hidden layer, $m$ neurons with total variation smaller than $\gamma_m$, a $C^3$ activation function $\varrho$ with linear growth condition and bounded derivative, e.g., a sigmoid activation function, or a tanh function (see \cite{CHX}). This class of neural networks can be represented by the parametric set of functions
 \begin{equation}
 \begin{aligned}
 \mathcal{N}\mathcal{N}^{\varrho}_{d,1,2,m}(\Theta^{\gamma}_{m}):=\left\{x\in\mathbb{R}^d\mapsto\mathcal{U}(x;\theta)=\sum_{k=1}^{m}c_{k}\varrho(a_{k}x+b_{k})+b_{0},\theta=(a_{k},b_{k},c_{k},b_{0})^m_{k=1}\in\Theta^{\gamma}_{m}\right\},\nonumber
 \end{aligned}
 \end{equation}
 with
  \begin{equation}
 \begin{aligned}
 \Theta^{\gamma}_{m}:=\left\{\theta=(a_{k},b_{k},c_{k},b_{0})^m_{k=1}:\mathop{max}_{k=1,...m}|a_{k}|\leq\gamma_{m},\sum_{k=1}^{m}|c_{k}|\leq\gamma_{m}\right\},\nonumber
 \end{aligned}
 \end{equation}
for some sequence $\{\gamma_{m}\}_{m}$ convering to $\infty$, as $m$ goes to infinity, and such that (see \cite{CHX})
\begin{equation}\label{eq2.6}
\begin{aligned}
\lim_{m,N \to \infty}\frac{\gamma_{m}^{6}}{N}=0.
\end{aligned}
\end{equation}
{It should be emphasized that the sequence $\{ \gamma_m \}_{m}$ depends on the number of neurons $m$ in the hidden layer of the neural network $\mathcal{U}\in \mathcal{N}\mathcal{N}^{\varrho}_{d,1,2,m}(\Theta^{\gamma}_{m})$, is used to constrain on the maximum norm of the input weights and the $L^{1}-$ norm of the output weights, and $N$ is the number of subdivisions of the interval $t \in [0,T]$.}

Notice that the neural networks in $\mathcal{N}\mathcal{N}^{\varrho}_{d,1,2,m}(\Theta^{\gamma}_{m})$ have their first, second, and third derivatives uniformly bounded w.r.t. the state variable $x$, there exists some constant $C$ depending only on $d$ and the derivatives of $\varrho$ such that for any $\mathcal{U}\in\mathcal{N}\mathcal{N}^{\varrho}_{d,1,2,m}(\Theta^{\gamma}_{m})$ (see \cite{CHX}),
\begin{equation}\label{eq2.7}
\begin{aligned}
 \sup_{x\in\mathbb{R},\theta\in\Theta^{\gamma}_{m}}|D_{x}\mathcal{U}(x;\theta)|\leq C\gamma^2_{m}, \quad \sup_{x\in\mathbb{R},\theta\in\Theta^{\gamma}_{m}}|D^2_{x}\mathcal{U}(x;\theta)|\leq C\gamma^3_{m},\quad\sup_{x\in\mathbb{R},\theta\in\Theta^{\gamma}_{m}}|D^3_{x}\mathcal{U}(x;\theta)|\leq C\gamma^4_{m}.
\end{aligned}
\end{equation}

{ We should note that in this subsection, $\mathcal{U}$ denotes a neural network in $\mathcal{N}\mathcal{N}^{\varrho}_{d,1,2,m}(\Theta^{\gamma}_{m})$, while in the following subsections and throughout the remainder of the text, it is used to represent the neural network approximating the solution $u(t,x)$ to the PIDE.}

\subsection{The DFBDP scheme}  In this subsection we propose the DFBDP algorithm for approximating the decoupled FBSDEJs (\ref{eq2.3}). To do this, for a nonuniform partition of the interval $[0,T]$, let us define $\pi=\left\{t_i\right\}_{i=0}^N$ , $\Delta t_i=t_{i+1}-t_i$ , $|\pi|=\max_{i=0,\cdots,N-1} {\Delta t_i}$, and set $\Delta W_{t_i}=W_{t_{i+1}}-W_{t_i}$. Then it is well-known that an Euler scheme for the first equation or the forward equation in  (\ref{eq2.3}) has the form
\begin{equation}\label{eq2.8}
\begin{aligned}
  X^\pi_{t_{i+1}}=&\xi,\\
  X^\pi_{t_{i+1}}=&X^\pi_{t_{i}}+b(t_i,X^\pi_{t_i})\Delta t_i+\sigma(t_i,X^\pi_{t_i})\Delta W_{t_i}+\int_{t_{i}}^{t_{i+1}}\int_{E}\beta(t_i, X_{t_i},e)\widetilde{\mu}(\mathrm{d}e,\mathrm{d}s).
 \end{aligned}
 \end{equation}
Note that we already formulate the PIDEs equivalently as FBSDEJs by nonlinear Feyman-Kac formula. Then to approximation the second equation in (\ref{eq2.3}), we define
$$
\begin{gathered}
F(t, x, y, z, U):=y-f\left(t, x, y, z, \int_{E} U(e) \lambda(d e)\right) \Delta t_{i}+z^{\mathrm{T}} \Delta W_{t_{i}}+\int_{t_{i}}^{t_{i+1}}\int_{E} U(e) \widetilde{\mu}\left(\mathrm{d}e,\mathrm{d}t\right),
\end{gathered}
$$
and consider the following approximation
\begin{equation}\label{eq2.9}
\begin{aligned}
&u\left(t_{i+1}, X_{t_{i+1}}^{\pi}\right)\\
&\approx F\left(t_{i}, X_{t_{i}}^{\pi}, u\left(t_{i}, X_{t_{i}}^{\pi}\right), \sigma^{T}(t_i,X_{t_i}^{\pi}) D_{x} u\left(t_{i}, X_{t_{i}}^{\pi}\right),u\left(t_{i}, X_{t_{i}}^{\pi}+\beta\left(t_{i}, X_{t_{i}}^{\pi}, e\right)\right)-u\left(t_{i}, X_{t_{i}}^{\pi}\right)\right).
\end{aligned}
\end{equation}

Now adapting the argument of \cite{CHX} to the non-local case, we propose the following DFBDP algorithm:
\begin{enumerate}[(i)]
\item Initialize $\widehat{\mathcal{U}}_{{N}}\left(X_{t_{N}}^{\pi}\right)=g\left(X_{t_{N}}^{\pi}\right)$.
\item Approximate $u\left(t_{i}, X_{t_{i}}^{\pi}\right)$ and $U_{t_{i}}\left(X_{t_{i}}^{\pi}, e\right)=u\left(t_{i}, X_{t_{i}}^{\pi}+\beta\left(t_{i}, X_{t_{i}}^{\pi}, e\right)\right)-u\left(t_{i}, X_{t_{i}}^{\pi}\right)$ by a deep learning algorithm. More specificity, let $\rho \in \mathbb{N}$ be the size of the hidden layer and $\theta \in \mathbb{R}^{\rho}$, and use $\mathcal{U}_{t_i}^{\theta,\pi}=\mathcal{U}_{i}(\cdot;\theta_{i}): \mathbb{R}^{d} \rightarrow \mathbb{R}$ and $\mathcal{G}_{t_i}^{\theta,\pi}=\mathcal{G}_{i}(\cdot,e;\theta_{i}): \mathbb{R}^{d}\times \mathbb{R} \rightarrow \mathbb{R}, i \in\{0,1, \cdots, N-1\}$, $e \in E$, to approximate the following policy functions:
$$
\begin{gathered}
\mathcal{U}_{i}\left(X_{t_{i}}^{\pi};\theta_{i}\right) \approx u\left(t_{i}, X_{t_{i}}^{\pi}\right), \qquad
\mathcal{G}_{i}\left(X_{t_{i}}^{\pi}, e;\theta_{i}\right) \approx U_{t_{i}}\left(X_{t_{i}}^{\pi}, e\right) .
\end{gathered}
$$

\item We then minimize the expected loss function through stochastic gradient descent (SGD) type algorithms
\begin{equation}\label{eq2.10}
\begin{aligned}
&\widehat{L}_{i}(\theta)\\&\quad=\mathbb{E}\left|\widehat{\mathcal{U}}_{i+1}\left(X_{t_{i+1}}^{\pi}\right)-F\left(t_{i}, X_{t_{i}}, \mathcal{U}_{i}\left(X_{t_{i}}^{\pi};\theta_i\right), \sigma^{\mathrm{T}}(t_i,X_{t_i}^{\pi}) {D}_{x} \mathcal{U}_{{i}}\left(X_{t_{i}}^{\pi};\theta_i\right), \mathcal{G}_{{i}}\left(X_{t_{i}}^{\pi},e;\theta_i\right)\right)\right|^{2},
\end{aligned}
\end{equation}
where ${D}_{x} \mathcal{U}_{t_{i}}^{\theta, \pi}\left(X_{t_{i}}^{\pi}\right)$ is the numerical differentiation of $\mathcal{U}_{t_{i}}^{\theta, \pi}\left(X_{t_{i}}^{\pi}\right)$. Through the results calculated above $\theta_{i}^{\star} \rightarrow \inf _{\theta} \widehat{L}_{i}(\theta)$, update $\widehat{\mathcal{U}}_{i}\left(X_{t_{i}}^{\pi}\right)=\mathcal{U}_{i}\left(X_{t_{i}}^{\pi};\theta_i^{\star}\right)$, $\widehat{\mathcal{Z}}_{i}\left(X_{t_{i}}^{\pi}\right)=\sigma^{\mathrm{T}}(t_i,X_{t_i}^{\pi}) D_{x} \mathcal{U}_{i}\left(X_{t_{i}}^{\pi};\theta_{i}^{\star}\right)$, and { $\widehat{\mathcal{G}}_{i}(X_{t_i}^{\pi},e)=\mathcal{G}_{i}(X_{t_i}^{\pi},e;\theta_{i}^{\star})$}, $\widehat{\mathcal{T}}_{i}\left(X_{t_{i}}^{\pi}\right)=\mathcal{T}_{i}(X_{t_i}^{\pi};\theta_{i}^{\star})$, where $${\mathcal{T}}_{{i}}\left(X_{t_{i}}^{\pi};\theta_i\right)=\int_{E} \mathcal{G}_{{i}}\left(X_{t_{i}}^{\pi}, e;\theta_i\right) \lambda(\mathrm{d}e),$$
\end{enumerate}
and then repeat the above steps (i), (ii) and (iii).

We end this section with several remarks on our DFBDP algorithm. We first note that our DFBDP algorithm can be viewed as an extension of the DBDP2 scheme proposed in \cite{CHX} for semilinear PDEs. Since the function $F$ depends on the time interval $(t_i,t_{i+1}]$ in terms of the integrated measure $\tilde\mu(\mathrm{d}e,(t_i,t_{i+1}])$, however, we do not approximate the nonlocal term at time $t_i$ in the PIDEs but consider how the measure $\tilde \mu$ behaves on the time interval $(t_i,t_{i+1}]$.

Unlike the deep learning scheme presented in \cite{Castro21} for PIDEs in which the gradient $\sigma^T\nabla u(t_i,\cdot) $ is also approximated by neural network, we approximate it by numerical differentiation. The deep learning scheme of \cite{Castro21} can be viewed as an extension of the DBDP1 scheme proposed in \cite{CHX}.

The pseudo-code for the DFBDP Algorithm is presented in Algorithm 1.

\begin{algorithm}
\caption{DFBDP algorithm of FBSDEJs.}
\label{alg:fbsdej}
{\begin{algorithmic}[1]

\State \textbf{Input:} Given PIDE (2.1), class of deep Neural Networks $\mathcal{U}_{t_i}^{\theta,\pi} : \mathbb{R}^d \rightarrow \mathbb{R}$ and $\mathcal{G}_{t_i}^{\theta,\pi} : \mathbb{R}^d \times \mathbb{R} \rightarrow \mathbb{R}$, $i \in \{0, 1, \cdots, N-1\}$
\State \textbf{Output:} Suitable approximation of $u(t_i, X_{t_i}^{\pi})$ and $u(t_i, X_{t_i}^{\pi} + \beta(t_i, X_{t_i}^{\pi}, e)) - u(t_i, X_{t_i}^{\pi})$
\State Use the nonlinear Feynman-Kac formula and classical Euler scheme to transform PIDE into the FBSDEJ scheme
\State Generate initial state $\xi$ ;
\For{\textit{each time step $i=N-1\rightarrow 0$}}
    \If{$i<N-1$}
        \State initialize the weights and bias of the $i-th$ neural network to the weights and bias of the $(i+1)-th$ neural network
    \EndIf
    \For{\textit{iter} $<$ \textit{max iteration}}
        \For{\textit{each time step $j=0\rightarrow i$}}
            \State Simulate the Brownian increment $\Delta W_{t_j}$ and random jump size $e$
            \State Generate the forward stochastic process $X_{t_i}^{\pi}$ by (\ref{eq2.8})
        \EndFor

        \State Approximate $u(t_i, X_{t_i}^{\pi})$, $u(t_i, X_{t_i}^{\pi} + \beta(t_i, X_{t_i}^{\pi}, e)) - u(t_i, X_{t_i}^{\pi})$ by $\mathcal{U}_{t_i}^{\theta,\pi}$, $\mathcal{G}_{t_i}^{\theta,\pi}$
        \State Approximate $\sigma^{\top}(t_i, X_{t_i}^{\pi}) \nabla_x u(t_i, X_{t_i}^{\pi})$ by numerical differentiation techniques

        \If{$i = N-1$}
            \State Calculate $u(t_{N},X_{t_{N}}^{\pi})$ by terminal condition $g(\cdot)$
        \Else
            \State Approximate $u(t_{i+1},X_{t_{i+1}}^{\pi})$ by neural network $\mathcal{U}_{t_{i+1}}^{\theta^{\star},\pi}$
        \EndIf

        \State Calculate the loss function $\widehat{L}_{i}(\theta)$ by (\ref{eq2.10}) and optimize the neural network parameters $\theta_{i}$ through stochastic gradient descent
    \EndFor

\EndFor

\end{algorithmic}}
\end{algorithm}

%%%%%%%%%%%%%%%%%%%%%%%%%%%%%%%%%%%%%%%%%%第3章 误差分析%%%%%%%%%%%%%%%%%%%%%%%%%%%%%%%%%%%%%%%%%%%%%%

\section{Convergence analysis for DFBDP scheme}\label{sec:1}
The main goal of this section is to show the convergence of the DFBDP scheme towards the solution $(X,Y,Z,\Gamma)$ to the FBSDEJs (\ref{eq2.3}).  {Let us firstly define the following process space for $t \in [0,T]$ and $p \geq 2$ (see \cite{BR})}.\\
{Let $\mathcal{S}^{p}(\mathbb{R})$ denote the space of $\mathcal{F}_{t}-$adapted c$\grave{a}$dl$\grave{a}$g  processes $Y:\Omega\times[0,T]\rightarrow \mathbb{R}$, which satisfy}
\begin{eqnarray*}
  ||Y||_{\mathcal{S}^{p}} := \mathbb{E}\left[\sup_{t \in [0,T]}|Y_{t}|^{p}\right]^{1/p}<\infty.
\end{eqnarray*}
{ Let $L^{p}_{W}(\mathbb{R}^{d})$ denote the space of  $\mathcal{F}_{t}-$progressively measurable  processes $Z:\Omega \times [0,T]\rightarrow \mathbb{R}^{d}$, which satisfies}
\begin{eqnarray*}
||Z||_{L^{p}_{W}} := \mathbb{E}\left[\left(\int_{0}^{T}|Z_{t}|^2\mathrm{d}t\right)^{p/2}\right]^{1/p}<\infty.
\end{eqnarray*}
{Let $L^{p}_{\lambda}(\mathbb{R})$ denote the space of mappings $U:\Omega\times[0,T]\times E\to \mathbb{R}$ which are $\mathcal{P}\bigotimes \varepsilon$ measurable and such that}
\begin{eqnarray*}
||U||_{L^{p}_{\lambda}} := \mathbb{E}\left[\int_{0}^{T}\int_{E}|U_{t}(e)|^{p}\lambda(\mathrm{d}e)\mathrm{d}t\right]^{1/p}<\infty.
\end{eqnarray*}
We set $\mathcal{B}^{p}=\mathcal{S}^p(\mathbb{R}) \times L^p_{W}(\mathbb{R}^{d})\times L^{p}_{\lambda}(\mathbb{R}).$\\

\subsection{Assumptions} We first make some assumptions on the coefficients of the equations considered here.

\textbf{Assumption (H1)} (i) { Regularity of the following functions:} { $g:\mathbb{R}^{d}\rightarrow \mathbb{R}$, $b:\mathbb{R}^{d}\rightarrow \mathbb{R}^{d}$ and $\sigma:\mathbb{R}^{d}\rightarrow \mathbb{R}^{d \times d}$} all satisfy K-Lipschitz continuous for $x \in \mathbb{R}^{d}$\\
(ii) { Boundedness and Uniformly Lipschitz of measurable functions $\beta$:} for all $e \in E$, there exists $C>0$, $\beta$ satisfy
\begin{equation*}
\sup _{e \in E}|\beta(0, 0, e)| \leq C,
\end{equation*}
and for all $x_{1}, x_{2} \in \mathbb{R}^{d}$, there exists $K>0$, $\beta$ satisfy
\begin{equation*}
\sup _{e \in E}\left|\beta\left(t, x_{1}, e\right)-\beta\left(t, x_{2}, e\right)\right| \leq K\left|x_{1}-x_{2}\right| .
\end{equation*}
(iii) { H\"older continuity of function $f(\cdot)$:} there is a constant $f_{L}$ such that $f$ satisfies:
\begin{eqnarray}
\begin{array}{lll}
&|f(t_{2},x_{2},y_{2},z_{2},\Gamma_{2})-f(t_{1}, x_{1}, y_{1}, z_{1},\Gamma_{1})| \\
&\qquad\le f_{L}\left(|t_{2}-t_{1}|^{1/2}+|x_{2}-x_{1}|+|y_{2}-y_{1}|+|z_{2}-z_{1}|+|\Gamma_{2}-\Gamma_{1}|\right),\nonumber
\end{array}
\end{eqnarray}
for all $(t_{1},x_{1},y_{1},z_{1},\Gamma_{1})$ and $(t_{2},x_{2},y_{2},z_{2},\Gamma_{2})\in[0,T]\times \mathbb{R}^{d}\times \mathbb{R}\times \mathbb{R}^{d}$, therefore:
\begin{eqnarray}
	\sup\limits_{0\leq t \leq T}|f(t,0,0,0,0)|<\infty\nonumber.
\end{eqnarray}
(iv) { Bounded derivatives with K-Lipschitz continuity of coefficients:} $b$, $\sigma$, $\beta(\cdot,e)$, $f$ and $g$ have bounded derivative function with K-Lipchitz derivatives.\\

\textbf{Assumption(H2)} (i)The functions $x\to b(t,\cdot), \sigma(t,\cdot), \beta(t,\cdot,e)$ are $C^{1}$ with bounded derivatives uniformly w.r.t. $(t,x)\in[0,T]\times \mathbb{R}^d$.\\
(ii){ and the function $x\to f(t,\cdot)$ is $C^1$ with bounded derivatives uniformly w.r.t. $(t,x,y,z,\Gamma)$ in $[0,T]\times \mathbb{R}^{d} \times \mathbb{R} \times \mathbb{R}^{d} \times \mathbb{R}$.}\\

Recall that Assumption \textbf{(H1)} ensures the existence and uniqueness of an adapted solution \\ $(\mathcal{X},Y,Z,U) \in \mathcal{S}^{2} \times \mathcal{B}^{2}$ to (\ref{eq2.3}) satisfying (see \cite{BR}, Remark 2.1)
\begin{equation}\label{eq3.0}
\begin{aligned}
    &\mathbb{E}\left[\sup_{t \in [0,T]}|\mathcal{X}_{t}-\xi|^{p}\right] \leq C(1+|\xi|^{p}),
\end{aligned}
\end{equation}

\begin{equation*}
\begin{aligned}
&||(\mathcal{X},Y,Z,U)||^{p}_{\mathcal{S}^{p}\times \mathcal{B}^{p}}\\
&\quad=\mathbb{E}\left[\sup\limits_{0\leq t \leq T}|\mathcal{X}_{t}|^{p}+\sup\limits_{0\leq t \leq T}|Y_{t}|^{p}+\left(\int_{0}^{T}|Z_{t}|^{2}dt \right)^{p/2} +\int_{0}^{T}\int_{E}|U_{t}(e)|^p\lambda(\mathrm{d}e)\mathrm{d}t\right]\leq C(1+|\xi|^{p}),
\end{aligned}
\end{equation*}
and
{\begin{equation*}
\begin{aligned}
&\mathbb{E}\left [\sup_{t_i \leq t\leq t_{i+1}}|Y_{t}-Y_{t_i}|^{p}\right]\leq C\left [\left(1+|\xi|^{p}\right)|\Delta t_i|^{p}+||Z||^{p}_{L^{p}_{W,[t_i,t_{i+1}]}}+||U||^{p}_{L^{p}_{\lambda,[t_i,t_{i+1}]}}\right ]\\
\end{aligned}
\end{equation*}}
where $\mathcal{X}_{t_0}=\xi$, $p \geq 2$ and $C$ is a constant for the above three inequalities.\\

%\begin{assumption}
%The function $g$ satisfies the linear growth condition, third order derivative and derivative function are bounded.
%\end{assumption}

\subsection{Main convergence results}In this section, we present our main convergence results for the DFBDP scheme. To do this, we first introduce several important properties of $(\mathcal{X}, Y, Z, \Gamma)$.

{We firstly note that the error of the Euler scheme (\ref{eq2.8}) for the forward equation (\ref{eq2.3}) satisfies (see \cite{LD} and \cite{BR})
\begin{equation}\label{eq3.1}
    \max_{i=0,\cdots,N-1} \mathbb{E}\left[\sup_{t\in[t_i,t_{i+1}]} |\mathcal{X}_{t_{}}-X_{t_{i}}|^{2}\right]=O(|\pi|).
\end{equation}}

{ In addition, we have the $L^{2}$-regularity of $Y$ (see \cite{BR}, Proposition 2.1)
\begin{eqnarray}\label{eq3.2}
\sum_{i=0}^{N-1} \mathbb{E}\left[\int_{t_i}^{t_{i+1}}|Y_{t}-Y_{t_i}|^{2}\mathrm{d}t\right]=O(|\pi|).
\end{eqnarray}}

We also have the $L^{2}$-regularity of $Z$ and $\Gamma$
\begin{eqnarray}\label{eq3.3}
\varepsilon^Z(\pi):={ \mathbb{E}}\left[\sum_{i=0}^{N-1}\int_{t_{i}}^{t_{i+1}}|Z_{t}-\bar{Z}_{t_{i}}|^2dt\right] \quad with \quad  \bar{Z}_{t_{i}}:=\frac{1}{\Delta t_{i}}{ \mathbb{E}_{i}}\left[\int_{t_{i}}^{t_{i+1}}Z_{t}dt\right],
\end{eqnarray}
and
\begin{eqnarray}\label{eq3.4}
\varepsilon^{\Gamma}(\pi):={\mathbb{E}}\left[\sum_{i=0}^{N-1}\int_{t_{i}}^{t_{i+1}}|\Gamma_{t}-\bar{\Gamma}_{t_{i}}|^2dt\right] \quad with \quad  \bar{\Gamma}_{t_{i}}:=\frac{1}{\Delta t_{i}}{\mathbb{E}_{i}}\left[\int_{t_{i}}^{t_{i+1}}\Gamma_{t}dt\right],
\end{eqnarray}
{where $\mathbb{E}_{i}$ denotes the conditional expectation given $\mathcal{F}_{t_i}$}. Since $\overline{Z}$ is an $L^2-$projection of $Z$ and $\overline{\Gamma}$ is an $L^2-$projection of $\Gamma$, { we know that under Conditions (i) to (iv) of (\textbf{H1}), the following properties related to $Z$ and $\Gamma$ hold: (see \cite{Castro21})
\begin{equation}\label{eq3.5}
\varepsilon^{Z}(\pi)=O(|\pi|),\qquad \varepsilon^{\Gamma}(\pi)=O(|\pi|).
\end{equation}}

Next, we introduce several important lemmas for stochastic calculus.
\begin{lemma}[\cite{R2005}]\label{le2.1}
	Let $Z^1, Z^2\in L^{2}_{W}(\mathbb{R}^{d})$ and $U^1, U^2\in L^{2}_{\mu}(\mathbb{R})$, we have
	\begin{eqnarray*}
	\mathbb{E}\left[ \int_{t_{i}}^{t_{i+1}}Z_{s}^1\mathrm{d}W_{s}\int_{t_{i}}^{t_{i+1}}Z_{s}^2\mathrm{d}W_{s}\right]=\mathbb{E}\left[\int_{t_{i}}^{t_{i+1}}Z_{s}^1Z_{s}^2\mathrm{d}s\right],
    \end{eqnarray*}	
	\begin{eqnarray*}
		\mathbb{E}\left[ \int_{t_{i}}^{t_{i+1}}\int_{E}U_{s}^1(e)\tilde{\mu}(\mathrm{d}e,\mathrm{d}s)\int_{t_{i}}^{t_{i+1}}Z_{s}^2\mathrm{d}W_{s}\right]=0,
		    \end{eqnarray*}	
and
	\begin{eqnarray*}
	\mathbb{E}\left[\int_{t_{i}}^{t_{i+1}}\int_{E}U_{s}^1(e)\tilde{\mu}(\mathrm{d}e,\mathrm{d}s)\int_{t_{i}}^{t_{n+1}}\int_{E}U_{s}^2(e)\tilde{\mu}(\mathrm{d}e,\mathrm{d}s)\right]=\mathbb{E}\left[\int_{t_{i}}^{t_{i+1}}\int_{E}U_{s}^{1}(e)U_{s}^{2}(e)\lambda(\mathrm{d}e)\mathrm{d}s\right].
	\end{eqnarray*}	
\end{lemma}

\begin{lemma}[Martingale representation theorem \cite{R2005}]\label{le2.2}
	Let m(t), $0\leq t\leq T$, be a martingale with respect to this filtration $\mathbb F$, then there is an adapted process $(Z, U)\in L^{2}_{W}(\mathbb{R}^{d})\times L^{2}_{\mu}(\mathbb{R})$, such that
	\begin{eqnarray*}
	m(t)=m(0)+ \int_{0}^{t}Z_{s}^{T}\mathrm{d}W_{s}+\int_{0}^{t}\int_{E}U_{s}(e)\widetilde{\mu}(\mathrm{d}e,\mathrm{d}s).
	\end{eqnarray*}	
\end{lemma}

Now let us investigate the convergence of the DFBDP scheme. We define (implicity)
\begin{eqnarray}\label{eq3.6}
\left\{
\begin{aligned}
\widehat{\mathcal{V}}_{t_{i}}:&=\mathbb{E}_{i}\left[\widehat{\mathcal{U}}_{i+1}(X_{t_{i+1}})\right]+f(t_{i}, X_{t_{i}},\widehat{\mathcal{V}}_{t_{i}},\bar{\widehat{Z}}_{t_{i}},\bar{\widehat{\Gamma}}_{t_{i}})\Delta t_{i},\\
\overline{\widehat{Z}}_{t_{i}}:&=\frac{1}{\Delta t_{i}}\mathbb{E}_{i}\left[{\widehat{\mathcal{U}}_{i+1}}(X_{t_{i+1}})\Delta W_{t_{i}}\right],\\
\overline{\widehat{\Gamma}}_{t_{i}}:&=\frac{1}{\Delta t_{i}}\mathbb{E}_{i}\left[{\widehat{\mathcal{U}}_{i+1}}(X_{t_{i+1}})\int_{t_i}^{t_{i+1}}\int_{E}\gamma(e)\widetilde{\mu}(\mathrm{d}e,\mathrm{d}t)\right],
\end{aligned}
\right.
\end{eqnarray}
for $i= 0,\cdots,N-1$. Notice that $\hat{\mathcal{V}}_{t_{i}}$ is well-defined for $|\pi|$ small enough (recall that $f$ is Lipschitz) by a fixed point argument. By the Markov property of the discretized forward process $(X_{t_{i}}), i=0,\cdots,N$, there exists some deterministic functions
$\widehat{v}_{{i}}(\cdot)$,$\overline{\widehat{z}}_{i}(\cdot)$, and $\overline{\widehat{\mathcal{T}}}_{{i}}(\cdot)$ such that
{\begin{eqnarray}
\begin{aligned}
\widehat{\mathcal{V}}_{t_{i}}= \widehat{v}_{i}(X_{t_i}),\qquad
\overline{\widehat{Z}}_{{i}}=\overline{
\widehat{z}}_{i}(X_{t_i}),\qquad
\overline{\widehat{\Gamma}}_{t_{i}}=\overline{\widehat{\mathcal{T}}}_{i}(X_{t_i}),
\nonumber
 \end{aligned}
\end{eqnarray}}
where
{ \begin{equation*}
\overline{\widehat{\mathcal{T}}}_{i}(X_{t_i})=\int_{E}\overline{\widehat{U}}_{i}(X_{t_i},e)\gamma(e)\lambda(\mathrm{d}e).
\end{equation*}}

Moreover, there exists an $\mathbb{R}^d-$valued
square integrable process {$\{\widehat{Z}_t\}_{t\geq 0}$ and $\{\widehat{U}_{t}(e)\}_{t\geq 0}$} such that
\begin{eqnarray}\label{eq3.7}
\begin{aligned}
\widehat{\mathcal{U}}_{i+1}(X_{t_{i+1}})&=\widehat{\mathcal{V}}_{t_{i}}-f\left(t_{i}, X_{t_{i}},\widehat{\mathcal{V}}_{t_{i}},\overline{\widehat{Z}}_{t_{i}},\overline{\widehat{\Gamma}}_{t_{i}}\right)\Delta t_{i}\\
&\quad+\int_{t_{i}}^{t_{i+1}}\widehat{Z}^{\mathrm{T}}\mathrm{d}W_{s}+\int_{t_{i}}^{t_{i+1}}\int_{E}\widehat{U}_{s}(e)\widetilde{\mu}(\mathrm{d}e,\mathrm{d}s),
\end{aligned}
\end{eqnarray}
and by It$\hat{o}$ isometry, we have
\begin{equation*}
\begin{aligned}
&\overline{\widehat{Z}}_{t_{i}}= \frac{1}{\Delta t_{i}}\mathbb{E}_{i}\left[\int_{t_{i}}^{t_{i+1}}\widehat{Z}_{s}ds\right],\quad\overline{\widehat{\Gamma}}_{t_{i}}{=\frac{1}{\Delta t_{i}}\mathbb{E}_{i}\left[\int_{t_{i}}^{t_{i+1}}\int_{E}\widehat{U}_{s}(e)\gamma(e)\lambda(\mathrm{d}e)\mathrm{d}s\right].}
\end{aligned}
\end{equation*}

{Let us now define a measure of the (squared) error for the DFBDP scheme by
\begin{equation*}
\begin{aligned}
\mathcal{E}[( \widehat{\mathcal{U}}, \widehat{\mathcal{Z}},\widehat{\mathcal{T}} ), (Y, Z,\Gamma)] :&= \max_{i=0,\cdots,N-1}\mathbb{E}\left|Y_{t_i}-\widehat{\mathcal{U}}_{i}(X_{t_{i}})\right|^{2}
\\&\quad+\mathbb{E}\left[\sum_{i=0}^{N-1} \int_{t_{i}}^{t_{i+1}}\left|Z_{t}-\widehat{\mathcal{Z}}_{i}\left(X_{t_{i}}\right)\right|^{2} d t\right]+\mathbb{E}\left[\sum_{i=0}^{N-1} \int_{t_{i}}^{t_{i+1}}\left|\Gamma_{t}-\widehat{\mathcal{T}}_{i}\left(X_{t_{i}}\right)\right|^{2} d t\right] ,
\end{aligned}
\end{equation*}}
{where $\widehat{\mathcal{T}}_{i}(X_{t_i})=\int_{E}\mathcal{\mathcal{G}}(X_{t_i},e;\theta^{\star})\gamma(e)\lambda(\mathrm{d}e)$}, and two measures of the approximation error of the neural network by
{\begin{equation}
\begin{aligned}
\varepsilon^{\mathcal{N},v}_{i}:=\inf_{{\xi}}&\left\{ \mathbb{E}|\widehat{v}_{t_{i}}(X_{t_i})-\widehat{\mathcal{U}}_{i}(X_{t_{i}})|^2+\right. \left.\Delta t_{i}\mathbb{E}\left|\sigma^{T}(t_i,X_{t_{i}})(D_{x}\widehat{v}_{i}(X_{t_{i}})-D_{x}\mathcal{U}_{i}(X_{t_{i}};\xi))\right|^2 \right\},\nonumber
\end{aligned}
\end{equation}
and
\begin{equation}
\begin{aligned}
\varepsilon^{\mathcal{N},\Gamma}_{i}:=\inf _{{\eta}}&\left\{ \mathbb{E}\left[\int _{E}\left(\overline{\widehat{U}}_{t_{i}}(e)-\mathcal{G}_{i}(X_{t_{i}},e;\eta)\right)^{2}\lambda(\mathrm{d}e)\right] \right\}.\nonumber
\end{aligned}
\end{equation}}
Then we know that $\varepsilon^{\mathcal{N},v}_{i}$ and $\varepsilon^{\mathcal{N},\Gamma}_{i}$ converge to zero as $m$ goes to infinity in view of the universal approximation theorem. As a consequence, we have the following error estimates.

\begin{theorem}[Convergence of DFBDP scheme] Under (\textbf{H1})-(\textbf{H2}), there exists a constant $C>0$, independent of $\pi$, such that
\begin{equation}\label{eq3.8}
\begin{aligned}
&{\mathcal{E}\left[\left( \widehat{\mathcal{U}}, \widehat{\mathcal{Z}},\widehat{\mathcal{T}} \right), (Y, Z,\Gamma)\right]} \\
&\quad \leq C\left[\mathbb{E}\left|g\left(\mathcal{X}_{T}\right)-g\left(X_{T}\right)\right|^{2}+\frac{\gamma_{m}^{6}}{N}+\varepsilon^{Z}(\pi)+\varepsilon^{\Gamma}(\pi)+ {\sum_{i=0}^{N-1}\left(N \varepsilon_{i}^{\mathcal{N}, v}+\varepsilon_{i}^{\mathcal{N}, \Gamma}\right)}\right].
\end{aligned}
\end{equation}
\end{theorem}

\begin{proof} In the following, we show this theorem by the following five steps in which $C$ will denote a positive generic constant independent of $|\pi|$ and may take different values from line to line.
%%%%%%%%%%%%%%%%%%%%%%%%%%%%% Step 1 %%%%%%%%%%%%%%%%%%%%%%%%%%%%%%%%

{\bf Step 1}: Fixing $ i\in{0,...,N-1}$, subtracting (\ref{eq2.3}) and  (\ref{eq3.7}) leads to
\begin{eqnarray*}
Y_{t_{i}}-\widehat{\mathcal{V}}_{t_{i}}=\mathbb{E}_{i}\left[Y_{t_{i+1}}-\widehat{\mathcal{U}}_{i+1}(X_{t+1})\right]+\mathbb{E}_{i}\left[\int_{t_{i}}^{t_{i+1}}\left(f(t,{ \mathcal{X}_{t}}, Y_{t}, Z_{t}, \Gamma_{t})-f\left(t_{i}, X_{t_{i}},\widehat{\mathcal{V}}_{t_{i}},\overline{\widehat{Z}}_{t_{i}},\overline{\widehat{\Gamma}}_{t_{i}}\right)\right)dt\right].
\end{eqnarray*}
By using Young inequality: $(a+b)^2\leq (1+r \Delta t_{i})a^2+(1+\frac{1}{r \Delta t_{i}})b^2$ for some $r >0$ to be choose later, Cauchy-Schwarz inequality, the Lipschitz condition on $f$, and {the error bound of the Euler scheme of the forward process (\ref{eq3.1})}, we have
\begin{eqnarray}\label{eq3.9}
\begin{array}{lll}
	&\mathbb{E}|Y_{t_{i}}-\widehat{\mathcal{V}}_{t_{i}}|^2\\
 &\quad\leq(1+r \Delta t_{i})\mathbb{E}
	\left| \mathbb{E}_{i}\left[Y_{t_{i+1}}-\widehat{\mathcal{U}}_{i+1}(X_{t_i+1})\right]\right|^2\\
 &\qquad+5{f_{L}^2}\Delta t_{i}(1+\frac{1}{r \Delta t_{i}})\left\{|\Delta t_{i}|^2+{\mathbb{E}\left[\int_{t_i}^{t_{i+1}}\left|\mathcal{X}_{t}-X_{t_i}\right|^{2}\mathrm{d}t\right]}+\mathbb{E}\left[\int_{t_{i}}^{t_{i+1}}\left|Y_{t}-\widehat{\mathcal{V}}_{t_{i}}\right|^2\mathrm{d}t\right]\right. \\
 &\qquad\left. +\mathbb{E}\left[\int_{t_{i}}^{t_{i+1}}\left|Z_{t}-\overline{\widehat{Z}}_{t_{i}}\right|^2\mathrm{d}t\right]
	+\mathbb{E}\left[\int_{t_{i}}^{t_{i+1}}\left|\Gamma_{t}-\overline{\widehat{\Gamma}}_{t_{i}}\right|^2\mathrm{d}t\right]\right\}\\
&\quad\leq (1+r \Delta t_{i})\mathbb{E}
	\left| \mathbb{E}_{i}\left[Y_{t_{i+1}}-\widehat{\mathcal{U}}_{i+1}(X_{t+1})\right]\right|^2\\
&\qquad+\frac{5f_{L}^2}{r}(1+r \Delta t_{i})\left\{{C}|\pi|^2+{2}\mathbb{E}\left[\int_{t_{i}}^{t_{i+1}}\left|Y_{t}-Y_{t_{i}}\right|^2\mathrm{d}t\right]\right. \\
 &\qquad+\left. 2\Delta t_{i}{\mathbb{E}}\left|Y_{t}-\widehat{\mathcal{V}}_{t_{i}}\right|^2+\mathbb{E}\left[\int_{t_{i}}^{t_{i+1}}\left|Z_{t}-\overline{\widehat{Z}}_{t_{i}}\right|^2\mathrm{d}t\right]
	+\mathbb{E}\left[\int_{t_{i}}^{t_{i+1}}\left|\Gamma_{t}-\overline{\widehat{\Gamma}}_{t_{i}}\right|^2\mathrm{d}t\right]\right\}.
\end{array}
\end{eqnarray}
Recalling the definition of $\overline{Z}_{t_i}$ (\ref{eq3.3}) and $\overline{\Gamma}_{t_i}$ (\ref{eq3.4}), we observe that
\begin{eqnarray}\label{c}
\begin{aligned}
	\mathbb{E}\left[\int_{t_{i}}^{t_{i+1}}\left|Z_{t}-\overline{\widehat{Z}}_{t_{i}}\right|^2dt \right]=\mathbb{E}\left[\int_{t_{i}}^{t_{i+1}}\left|Z_{t}-{\overline{{Z}}_{t_{i}}}\right|^2dt\ \right]+ \Delta t_{i}\mathbb{E}\left|{\overline{Z}_{t_i}}-\overline{\widehat{Z}}_{t_{i}}\right|^2,
\end{aligned}
\end{eqnarray}
and
\begin{eqnarray}\label{h}
\begin{aligned}
	\mathbb{E}\left[\int_{t_{i}}^{t_{i+1}}\left|\Gamma_{t}-\overline{\widehat{\Gamma}}_{t_{i}}\right|^2dt \right]=\mathbb{E}\left[\int_{t_{i}}^{t_{i+1}}\left|\Gamma_{t}-{\overline{{\Gamma}}_{t_{i}}}\right|^2dt\ \right]+ \Delta t_{i}\mathbb{E}\left|{\overline{\Gamma}_{t_i}}-\overline{\widehat{\Gamma}}_{t_{i}}\right|^2.
\end{aligned}
\end{eqnarray}
Multiplying equation (\ref{eq2.3}) between $t_{i}$ and $t_{i+1}$ by $\Delta W_{t_{i}}$, according to  Lemma 3.1 and (\ref{eq3.6}), we have together with
\begin{eqnarray*}
	\begin{aligned}
	\Delta t_{i}(\overline{Z}_{t_i}-\overline{\widehat{Z}}_{t_{i}})=&\mathbb{E}_{i}\left[\Delta W_{t_{i}}\left(Y_{t_{i+1}}-\widehat{\mathcal{U}}_{i+1}(X_{t_{i+1}})\right)\right]+\mathbb{E}_{i}\left[\Delta W_{t_{i}}\int_{t_{i}}^{t_{i+1}}f(t,\mathcal{X}_{t}, Y_{t}, Z_{t}, \Gamma_{t})\mathrm{d}t\right]\\
	=&\mathbb{E}_{i}\left[\Delta W_{t_{i}}\left(Y_{t_{i+1}}-\widehat{\mathcal{U}}_{i+1}(X_{t_{i+1}})-\mathbb{E}_{i}\left[Y_{t_{i+1}}-\widehat{\mathcal{U}}_{i+1}(X_{t_{i+1}})\right]\right)\right]\\
&+\mathbb{E}_{i}\left[\Delta W_{t_{i}}\int_{t_{i}}^{t_{i+1}}f(t,\mathcal{X}_{t}, Y_{t}, Z_{t}, \Gamma_{t})dt\right].	
    \end{aligned}
\end{eqnarray*}
Similarly, multiplying equation (\ref{eq2.3}) between $t_{i}$ and $t_{i+1}$ by $\int_{t_{i}}^{t_{i+1}}\int_{E}\gamma(e)\tilde{\mu}(\mathrm{d}e,\mathrm{d}t)$, according to Lemma 3.1 and (\ref{eq3.6}), we have
\begin{eqnarray*}
\begin{aligned}
\Delta t_{i}(\overline{\Gamma}_{t_i}-\overline{\widehat{\Gamma}}_{t_{i}})=&\mathbb{E}_{i}\left[\int_{t_{i}}^{t_{i+1}}\int_{E}\gamma(e)\widetilde{\mu}(\mathrm{d}e,\mathrm{d}t)\left(Y_{t_{i+1}}-\widehat{\mathcal{U}}_{i+1}(X_{t_{i+1}})\right)\right]\\
&+\mathbb{E}_{i}\left[\int_{t_{i}}^{t_{i+1}}\int_{E}\gamma(e)\widetilde{\mu}(\mathrm{d}e,\mathrm{d}t)\int_{t_{i}}^{t_{i+1}}f(t,\mathcal{X}_{t}, Y_{t}, Z_{t}, \Gamma_{t})\mathrm{d}t\right]\\
=&\mathbb{E}_{i}\left[\int_{t_{i}}^{t_{i+1}}\int_{E}\gamma(e)\widetilde{\mu}(\mathrm{d}e,\mathrm{d}t)\left(Y_{t_{i+1}}-\widehat{\mathcal{U}}_{i+1}(X_{t_{i+1}})-\mathbb{E}_{i}\left[Y_{t_{i+1}}-\widehat{\mathcal{U}}_{i+1}(X_{t_{i+1}})\right]\right)\right]\\
&+\mathbb{E}_{i}\left[\int_{t_{i}}^{t_{i+1}}\int_{E}\gamma(e)\widetilde{\mu}(\mathrm{d}e,\mathrm{d}t)\int_{t_{i}}^{t_{i+1}}f(t,\mathcal{X}_{t}, Y_{t}, Z_{t}, \Gamma_{t})\mathrm{d}t\right].	
\end{aligned}
\end{eqnarray*}
{Applying the expectation form of the Cauchy–Schwarz inequality, the law of iterated conditional expectations and Lemma 3.1}, these imply for $\mathbb{E}|\overline{Z}_{t_i}-\overline{\widehat{Z}}_{t_i}|^{2}$:
{\begin{eqnarray}\label{a}
\begin{aligned}
&\Delta t_{i}\mathbb{E}|\overline{Z}_{t_i}-\overline{\widehat{Z}}_{t_{i}}|^2 \\
&\quad\leq  2 \Delta t_{i} \mathbb{E} \left|\mathbb{E}_{i}\left[\frac{\Delta W_{t_{i}}}{\Delta t_{i}}\left(Y_{t_{i+1}}-\widehat{\mathcal{U}}_{i+1}(X_{t_{i+1}})-\mathbb{E}_{i}[Y_{t_{i+1}}-\widehat{\mathcal{U}}_{i+1}(X_{t_{i+1}})]\right)\right]\right|^{2}	\\
&\qquad+2 \Delta t_{i} \mathbb{E}\left|\mathbb{E}_{i}\left[\frac{\Delta W_{t_i}}{\Delta t_{i}}\int _{t_{i}}^{t_{i+1}}f(t,\mathcal{X}_{t},Y_{t},Z_{t},\Gamma_{t})\mathrm{d}t\right]\right|^{2}  \\
& \quad\leq 2\Delta t_{i}\mathbb{E}\left[\mathbb{E}_{i}\left|\frac{\Delta W_{t_{i}}}{\Delta_{t_i}}\right|^{2}\mathbb{E}_{i}\left|Y_{t_{i+1}}-\widehat{\mathcal{U}}_{i+1}(X_{t_{i+1}})-\mathbb{E}_{i}\left[Y_{t_{i+1}}-\widehat{\mathcal{U}}_{i+1}(X_{t_{i+1}})\right]\right|^{2}\right]\\
& \qquad +2 (\Delta t_{i})^{2} \mathbb{E}\left[\mathbb{E}_{i}\left|\frac{\Delta W_{t_{i}}}{\Delta_{t_i}}\right|^{2}\mathbb{E}_{i}\left[\int _{t_i}^{t_{i+1}}\left|f(t,\mathcal{X}_{t},Y_{t},Z_{t},\Gamma_{t})\right|^{2}\mathrm{d}t \right]\right]\\
&\quad =2d\left(\mathbb{E}\left|Y_{t_{i+1}}-\widehat{\mathcal{U}}_{i+1}(X_{t_{i+1}})\right|^{2}-\mathbb{E}\left|\mathbb{E}_{i}\left[Y_{t_{i+1}}-\widehat{\mathcal{U}}_{i+1}(X_{t_{i+1}})\right]\right|^{2}  \right)\\
& \qquad +2d \Delta t_{i}\mathbb{E}\left[\int_{t_i}^{t_{i+1}}|f(t,\mathcal{X}_{t},Y_{t},Z_{t},\Gamma_{t})|^{2}\mathrm{d}t\right],
\end{aligned}
\end{eqnarray}}
and for $\mathbb{E}|\overline{\Gamma}_{t_i}-\overline{\widehat{\Gamma}}_{t_i}|^{2}$: {
\begin{eqnarray}\label{b}
\begin{aligned}
	&\Delta t_{i}\mathbb{E}|\overline{\Gamma}_{t_i}-\overline{\widehat{\Gamma}}_{t_{i}}|^2\\
& \leq  2 \Delta t_{i} \mathbb{E} \left|\mathbb{E}_{i}\left[\frac{\int_{t_i}^{t_{i+1}}\int_{E}\gamma(e)\widetilde{\mu}(\mathrm{d}e,\mathrm{d}t)}{\Delta t_{i}}\right.\right.\\
&\left.\left.\quad\times\left(Y_{t_{i+1}}-\widehat{\mathcal{U}}_{i+1}(X_{t_{i+1}})-\mathbb{E}_{i}[Y_{t_{i+1}}-\widehat{\mathcal{U}}_{i+1}(X_{t_{i+1}})]\right)\right]\right|^{2}	\\
&\quad+2 \Delta t_{i} \mathbb{E}\left|\mathbb{E}_{i}\left[\frac{\int_{t_i}^{t_{i+1}}\int_{E}\gamma(e)\widetilde{\mu}(\mathrm{d}e,\mathrm{d}t)}{\Delta t_{i}}\int _{t_{i}}^{t_{i+1}}f(t,\mathcal{X}_{t},Y_{t},Z_{t},\Gamma_{t})\mathrm{d}t\right]\right|^{2}  \\
& \leq 2\Delta t_{i}\mathbb{E}\left[\mathbb{E}_{i}\left|\frac{\int_{t_i}^{t_{i+1}}\int_{E}\gamma(e)\widetilde{\mu}(\mathrm{d}e,\mathrm{d}t)}{\Delta_{t_i}}\right|^{2}\right.\\
&\left.\quad\times\mathbb{E}_{i}\left|Y_{t_{i+1}}-\widehat{\mathcal{U}}_{i+1}(X_{t_{i+1}})-\mathbb{E}_{i}\left[Y_{t_{i+1}}-\widehat{\mathcal{U}}_{i+1}(X_{t_{i+1}})\right]\right|^{2}\right]\\
& \quad +2 (\Delta t_{i})^{2} \mathbb{E}\left[\mathbb{E}_{i}\left|\frac{\int_{t_i}^{t_{i+1}}\int_{E}\gamma(e)\widetilde{\mu}(\mathrm{d}e,\mathrm{d}t)}{\Delta_{t_i}}\right|^{2}\mathbb{E}_{i}\left[\int _{t_i}^{t_{i+1}}\left|f(t,\mathcal{X}_{t},Y_{t},Z_{t},\Gamma_{t})\right|^{2}\mathrm{d}t \right]\right]\\
&=2 \int_{E}\gamma^{2}(e)\lambda(\mathrm{d}e) \left(\mathbb{E}\left|Y_{t_{i+1}}-\widehat{\mathcal{U}}_{i+1}(X_{t_{i+1}})\right|^{2}-\mathbb{E}\left|\mathbb{E}_{i}\left[Y_{t_{i+1}}-\widehat{\mathcal{U}}_{i+1}(X_{t_{i+1}})\right]\right|^{2}  \right)\\
& \quad +2 \Delta t_i \int_{E}\gamma^{2}(e)\lambda(\mathrm{d}e)\mathbb{E}\left[\int_{t_i}^{t_{i+1}}|f(t,\mathcal{X}_{t},Y_{t},Z_{t},\Gamma_{t})|^{2}\mathrm{d}t\right].
\end{aligned}
\end{eqnarray}}
Substituting (\ref{c}), (\ref{h}), (\ref{a}) and (\ref{b}) into (\ref{eq3.9}), and choosing $r=10(d+\int_{E}\gamma^{2}(e)\lambda(de))f^{2}_{L}$, yield
\begin{eqnarray*}\label{eq3.19}
\begin{aligned}
\mathbb{E}\left|{Y_{t_i}}-\widehat{\mathcal{V}}_{t_{i}}\right|^2\leq &C\Delta t_{i}\mathbb{E}\left|Y_{t_i}-\widehat{\mathcal{V}}_{t_{i}}\right|^2+(1+r \Delta t_{i})\mathbb{E}\left|Y_{t_{i+1}}-\widehat{\mathcal{U}}_{i+1}(X_{t_{i+1}})\right|^2+C|\pi|^2\\&+{C}\mathbb{E}\left[\int_{t_{i}}^{t_{i+1}}\left|Z_{t}-{\overline{Z}_{t_{i}}}\right|^2dt\right]
+{C}\mathbb{E}\left[\int_{t_{i}}^{t_{i+1}}\left|\Gamma_{t}-{\overline{\Gamma}_{t_{i}}}\right|^2dt\right]\\&+{C}\mathbb{E}\left[\int_{t_{i}}^{t_{i+1}}\left|Y_{t}-Y_{t_{i}}\right|^2dt\right]+C\Delta t_{i}\mathbb{E}\left[\int_{t_{i}}^{t_{i+1}}{|f(t,X_{t}, Y_{t}, Z_{t}. \Gamma_{t})|^{2}}dt\right],
\end{aligned}
\end{eqnarray*}
Thus for $|\pi|$ small enough, we have
\begin{eqnarray}\label{eq3.14}
\begin{aligned}
\mathbb{E}\left|Y_{t_i}-\widehat{\mathcal{V}}_{t_{i}}\right|^2\leq &(1+C|\pi|)\mathbb{E}\left|Y_{t_{i+1}}-\widehat{\mathcal{U}}_{i+1}(X_{t_{i+1}})\right|^2+C|\pi|^2\\&+C\mathbb{E}\left[\int_{t_{i}}^{t_{i+1}}\left|Z_{t}-{\overline{{Z}}_{t_{i}}}\right|^2dt\right]
+C\mathbb{E}\left[\int_{t_{i}}^{t_{i+1}}\left|\Gamma_{t}-{\overline{{\Gamma}}_{t_{i}}}\right|^2dt\right]\\&+C\mathbb{E}\left[\int_{t_{i}}^{t_{i+1}}\left|Y_{t}-Y_{t_{i}}\right|^2dt\right]+C|\pi|\mathbb{E}\left[\int_{t_{i}}^{t_{i+1}}{|f(t,X_{t}, Y_{t}, Z_{t}, \Gamma_{t})|^{2}}dt\right].
\end{aligned}
\end{eqnarray} 
%%%%%%%%%%%%%%%%%%%%%%%%%%%%% Step 2 %%%%%%%%%%%%%%%%%%%%%%%%%%

{\bf Step 2}: By using Young inequality in the form $(a+b)^2\geq (1-|\pi|)a^2+(1-\frac{1}{|\pi|})b^2\geq (1-|\pi|)a^2-\frac{1}{|\pi|}b^2$, we get
\begin{eqnarray}\label{eq3.15}
\begin{aligned}
\mathbb{E}\left|Y_{t_i}-\widehat{\mathcal{V}}_{t_{i}}\right|^2=&\mathbb{E}\left|Y_{t_{i}}-\widehat{\mathcal{U}}_{i}(X_{t_{i}})+\widehat{\mathcal{U}}_{i}(X_{t_{i}})-\widehat{\mathcal{V}}_{t_{i}}\right|^2\\
\geq&(1-|\pi|)\mathbb{E}\left|Y_{t_{i}}-\widehat{\mathcal{U}}_{i}(X_{t_{i}})\right|^2-\frac{1}{|\pi|}\mathbb{E}\left|\widehat{\mathcal{U}}_{i}(X_{t_{i}})-\widehat{\mathcal{V}}_{t_{i}}\right|^2.
\end{aligned}
\end{eqnarray}
Combining (\ref{eq3.14}) and (\ref{eq3.15}) leads to, for $|\pi|$ small enough,
\begin{eqnarray*}
\begin{aligned}
\mathbb{E}\left|Y_{t_{i}}-\widehat{\mathcal{U}}_{i}(X_{t_{i}})\right|^2\leq &(1+C|\pi|)\mathbb{E}\left|Y_{t_{i+1}}-\widehat{\mathcal{U}}_{i+1}(X_{t_{i+1}})\right|^2+C|\pi|^2\\&+C\mathbb{E}\left[\int_{t_{i}}^{t_{i+1}}\left|Z_{t}-{\overline{Z}_{t_{i}}}\right|^2dt\right]
+C\mathbb{E}\left[\int_{t_{i}}^{t_{i+1}}\left|\Gamma_{t}-{\overline{\Gamma}_{t_{i}}}\right|^2dt\right]\\&+C\mathbb{E}\left[\int_{t_{i}}^{t_{i+1}}\left|Y_{t}-Y_{t_{i}}\right|^2dt\right]+CN\mathbb{E}\left|\widehat{\mathcal{V}}_{t_{i}}-\widehat{\mathcal{U}}_{i}(X_{t_{i}})\right|^2\\&+C|\pi|\mathbb{E}\left[\int_{t_{i}}^{t_{i+1}}{|f(t,X_{t}, Y_{t}, Z_{t}, \Gamma_{t})|^{2}}dt\right].
\end{aligned}
\end{eqnarray*}
Then from discrete Gronwall's lemma, and recalling the terminal condition $Y_{t_{N}}=g(X_{T})$, $\widehat{U}(X_{t_{N}})=g(X_{t_{N}})$, we get
\begin{eqnarray}\label{eq3.16}
\begin{aligned}
\max\limits_{i=0,...,N-1}\mathbb{E}\left|Y_{t_{i}}-\widehat{\mathcal{U}}_{i}(X_{t_{i}})\right|^2\leq &{C}\mathbb{E}\left|g(\mathcal{X}_{T})-g(X_{t_{N}})\right|^2+C|\pi|+C\varepsilon^Z(\pi)\\&+C{\varepsilon^{\Gamma}(\pi)}+CN\sum_{i=0}^{N-1}\mathbb{E}\left|\widehat{\mathcal{V}}_{t_{i}}-\widehat{\mathcal{U}}_{i}(X_{t_{i}})\right|^2.
\end{aligned}
\end{eqnarray}

%%%%%%%%%%%%%%%%%%%%%%%%%%%%% Step 3 %%%%%%%%%%%%%%%%%%%%%%%%%%%%%%%%

{\bf Step 3}: Fix $ i \in {0,...,N-1} $. Using the relation in the expression of the expected quadratic loss function, and recalling the definition of $\bar{\hat{Z}}_{t_{i}}$ as an $L^2$-projection of $\hat{Z}_{t}$, we have for all parameters $\theta=(\xi, \eta)$ of the neural networks $\mathcal{U}_{i}(.;\xi)$ and $\mathcal{G}_{i}(.,e;\eta)$  with $i=0,\cdots,N-1$
  \begin{equation*}
  \begin{aligned}
  \widehat{L}_{i}(\theta)=\widetilde{L}_{i}(\theta)+\mathbb{E}\left[\int_{t_{i}}^{t_{i+1}}\left|{\widehat{Z}_{t}}-\overline{\widehat{Z}}_{t_{i}}\right|^2\mathrm{d}t\right]+\mathbb{E}\left[\int_{t_{i}}^{t_{i+1}}\int_{E}\left|{\widehat{U}_{t}}(e)-\overline{\widehat{U}}_{t_{i}}(e)\right|^2\lambda(\mathrm{d}e)\mathrm{d}t\right],
  \end{aligned}
  \end{equation*}
  where
   \begin{equation}\label{eq3.17}
  \begin{aligned}
  \widetilde{L}_{i}(\theta):=&\mathbb{E}\left|\widehat{\mathcal{V}}_{t_{i}}-\right. \left.\mathcal{U}_{i}(X_{t_{i}},\xi)\right. \\
  & \left.+\Big[f\Big(t_{i}, X_{t_{i}},\mathcal{U}_{i}(X_{t_{i}}, \xi),{\sigma^{T}(t_{i},X_{t_{i}})D_{x}}\mathcal{U}_{i}(X_{t_{i}},\xi),{\int_{E}\mathcal{G}_{i}(X_{t_{i}},e;\eta)\gamma(e)\lambda(\mathrm{de})}\Big) \right.\\ & \left. -f\left(t_{i}, X_{t_{i}},\widehat{\mathcal{V}}_{t_{i}},\overline{\widehat{Z}}_{t_{i}},\overline{\widehat{\Gamma}}_{t_{i}}\right)\Big]\Delta t_{i}\right|^2+\Delta t_{i}\mathbb{E}\left|\overline{\widehat{Z}}_{t_{i}}-{\sigma^{T}(t_i,X_{t_{i}})D_{x}\mathcal{U}_{i}(X_{t_{i}},\xi)}\right|^2\\
  &+\Delta t_{i}\mathbb{E}\left[{\int_{E}\left(\overline{\widehat{U}}_{t_{i}}(e)-\mathcal{G}_{i}(X_{t_{i}},e;\eta)\right)^2\lambda(\mathrm{d}e)}\right].
  \end{aligned}
  \end{equation}
{By the boundedness of $\gamma(e)$ (see (\ref{eq1.2})) and using Cauchy–Schwarz inequality, we can obtain
\begin{equation}\label{eq3.17b}
\begin{aligned}
&\mathbb{E}\left|\int_{E}\left(\mathcal{G}_{i}(X_{t_{i}},e;\eta)-\overline{\widehat{U}}_{i}(X_{t_i},e)\right)\gamma(e)\lambda(\mathrm{d}e)\right|^2 \\
&\quad \leq  \int_{E}|\gamma(e)|^{2}\lambda(\mathrm{d}e) \cdot \mathbb{E}\left|\int_{E}\left(\mathcal{G}_{i}(X_{t_{i}},e;\eta)-\overline{\widehat{U}}_{i}(X_{t_i},e)\right)^{2}\lambda(\mathrm{d}e)\right|\\
&\quad \leq \int_{E}K_{\gamma}^{2}\lambda(\mathrm{d}e)\mathbb{E}\left|\int_{E}\left(\mathcal{G}_{i}(X_{t_{i}},e;\eta)-\overline{\widehat{U}}_{i}(X_{t_i},e)\right)^{2}\lambda(\mathrm{d}e)\right|.
\end{aligned}
\end{equation}}
And employing Young inequality again in the form $(a+b)^2\leq (1+r\Delta t_{i})a^2+(1+\frac{1}{r\Delta t_{i}})b^2$, together with the Lipschitz condition on $f$ in Assumption \textbf{(H1)}, from (\ref{eq3.17}) and (\ref{eq3.17b}) we obtain
\begin{equation}\label{eq3.18}
\begin{aligned}
  \widetilde{L}_{i}(\theta)\leq&(1+r\Delta t_{i})\mathbb{E}\left|\widehat{\mathcal{V}}_{t_{i}}-\mathcal{U}_{i}(X_{t_{i}},\xi)\right|^2 +3f^2_{L}\Delta t_{i}\left(1+\frac{1}{r\Delta t_{i}}\right)\left(\mathbb{E}\left|\widehat{\mathcal{V}}_{t_{i}}-\mathcal{U}_{i}(X_{t_{i}},\xi)\right|^2 \right.\\
  & \left.+\mathbb{E}\left|{\sigma^{T}(t_i,X_{t_{i}})D_x\mathcal{U}_{i}(X_{t_{i}},\xi)}-\overline{\widehat{Z}}_{t_{i}}\right|^2+{\mathbb{E}\left|\int_{E}\left(\mathcal{G}_{i}(X_{t_{i}},e;\eta)-\overline{\widehat{U}}_{i}(X_{t_i},e)\right)\gamma(e)\lambda(\mathrm{d}e)\right|^2}\right)
  \\&+\Delta t_{i}\mathbb{E}\left|\overline{\widehat{Z}}_{t_{i}}-{\sigma^{T}(t_i,X_{t_{i}})D_x\mathcal{U}_{i}(X_{t_{i}};\xi)}\right|^2+\Delta t_{i}{\mathbb{E}\left[\int_{E}\left(\overline{\widehat{U}}_{i}(X_{t_i},e)-\mathcal{G}_{i}(X_{t_{i}},e;\eta)\right)^2\lambda(\mathrm{d}e)\right]}  \\
\leq&(1+C\Delta t_{i})\mathbb{E}\left|\widehat{\mathcal{V}}_{t_{i}}-\mathcal{U}_{i}(X_{t_{i}},\xi)\right|^2+C\Delta t_{i}\mathbb{E}\left|\overline{\widehat{Z}}_{t_{i}}-{\sigma^{T}(t_{i},X_{t_{i}})D_x\mathcal{U}_{i}(X_{t_{i}},\xi)}\right|^2\\&+C\Delta t_{i}{\mathbb{E}\left[\int_{E}\left(\overline{\widehat{U}}_{i}(X_{t_i},e)-\mathcal{G}_{i}(X_{t_{i}},e;\eta)\right)^2\lambda(\mathrm{d}e)\right]}.
\end{aligned}
\end{equation}
On the other hand, using Young inequality in the form {$(a+b)^2\geq (1-r\Delta t_{i})a^2+\left(1-\frac{1}{r\Delta t_{i}}\right)b^{2} \geq(1-r\Delta t_{i})a^2-\frac{1}{r\Delta t_{i}}b^{2}$}, together with the Lipschitz condition on $f$, and choosing { $r=6 f^{2}_{L}\max \left \{1,\int_{E}K_{\gamma}^{2} \lambda(\mathrm{d}e) \right \} $}, from (\ref{eq3.17}) and (\ref{eq3.17b}), we have
\begin{equation}\label{eq3.19(1)}
\begin{aligned}
  \widetilde{L}_{i}(\theta)\geq&(1-r\Delta t_{i})\mathbb{E}\left|\widehat{\mathcal{V}}_{t_{i}}-\mathcal{U}_{i}(X_{t_{i}};\xi)\right|^2 -\frac{3f^2_{L}\Delta t_{i}}{r}\left(\mathbb{E}\left|\widehat{\mathcal{V}}_{t_{i}}-\mathcal{U}_{i}(X_{t_{i}};\xi)\right|^2 \right.\\
& \left.+\mathbb{E}\left|{\sigma^{T}(t_i,X_{t_{i}})D_x\mathcal{U}_{i}(X_{t_{i}};\xi)}-\overline{\widehat{Z}}_{t_{i}}\right|^2+{\mathbb{E}\left|\int_{E}\left(\mathcal{G}_{i}(X_{t_{i}},e;\eta)-\overline{\widehat{U}}_{i}(X_{t_i},e)\right)\gamma(e)\lambda(\mathrm{d}e)\right|^2}\right) \\&+\Delta t_{i}\mathbb{E}\left|\overline{\widehat{Z}}_{t_{i}}-{\sigma^{T}(t_i,X_{t_{i}})D_{x}\mathcal{U}_{i}(X_{t_{i}};\xi)}\right|^2+\Delta t_{i}{\mathbb{E}\left[\int_{E}\left(\overline{\widehat{U}}_{i}(X_{t_i},e)-\mathcal{G}_{i}(X_{t_{i}},e;\eta)\right)^2\lambda(\mathrm{d}e)\right]}  \\
\geq&(1-C\Delta t_{i})\mathbb{E}\left|\widehat{\mathcal{V}}_{t_{i}}-\mathcal{U}_{i}(X_{t_{i}},\xi)\right|^2+\frac{\Delta t_{i}}{2}\mathbb{E}\left|\overline{\widehat{Z}}_{t_{i}}-{\sigma^{T}(t_i,X_{t_{i}})D_{x}\mathcal{U}_{i}(X_{t_{i}};\xi)}\right|^2\\&+\frac{\Delta t_{i}}{2}{\mathbb{E}\left[\int_{E}\left(\overline{\widehat{U}}_{i}(X_{t_i},e)-\mathcal{G}_{i}(X_{t_{i}},e;\eta)\right)^2\lambda(\mathrm{d}e)\right]}.
\end{aligned}
\end{equation}\\

%%%%%%%%%%%%%%%%%%%%%%%%%%%%%% Step 4 %%%%%%%%%%%%%%%%%%%%%%%%%%%%%%

{\bf Step 4}: For simplicity of notation, we assume $d=1$, from (\ref{eq3.6}) and the Euler scheme (\ref{eq2.8}), by setting $X_{t_i}=x$, we have
\begin{align*}
    \widehat{v}_i(x) &= \widetilde{v}_i(x) + f\left(t_i, x, \widehat{v}_i(x), \overline{\widehat{z}}_i(x), \overline{\widehat{\mathcal{T}}}_i(x)\right) \Delta t_i, \quad
    \widetilde{v}_i(x) := \mathbb{E}\left[\widehat{\mathcal{U}}_{i+1}(X_{t_{i+1}}^{x})\right], \\
    \overline{\widehat{z}}_i(x) &= \frac{1}{\Delta t_i} \mathbb{E}\left[\widehat{\mathcal{U}}_{i+1}(X_{t_{i+1}}^x) \Delta W_{t_i}\right], \quad
    \overline{\widehat{\mathcal{T}}}_i(x) = \frac{1}{\Delta t_i} \mathbb{E}\left[\widehat{\mathcal{U}}_{i+1}(X_{t_{i+1}}^x) \int_{t_i}^{t_{i+1}} \int_E \gamma(e) \widetilde{\mu}(\mathrm{d}e, \mathrm{d}t)\right], \\
    X_{t_{i+1}}^x &= x + {b(t_i, x)} \Delta t_i + {\sigma(t_i, x)} \Delta W_{t_i} + \int_{t_i}^{t_{i+1}} \int_E {\beta(t_i, x, e)} \widetilde{\mu}(\mathrm{d}e, \mathrm{d}t).
\end{align*}
{ Applying the Malliavin integration-by-parts formula to $\overline{\widehat{z}}_{i}(x)$ , we have}
\begin{eqnarray*}
\begin{aligned}
\overline{\widehat{z}}_{i}(x)=&\sigma(t_{i},x)\mathbb{E}\left[D_{x}\widehat{\mathcal{U}}_{i+1}(X_{t_{i+1}}^{x})\right],
\end{aligned}	
\end{eqnarray*}
{and applying the Malliavin integration-by-parts formula for Poisson random measures to $\overline{\widehat{\mathcal{T}}}_{i}(x)$, we obtain
\begin{equation*}
\overline{\widehat{\mathcal{T}}}_{i}(x)=\frac{1}{\Delta_{t_i}}\mathbb{E}\left[\int_{t_i}^{t_{i+1}}\int_{E}\left(\widehat{\mathcal{U}}_{i+1}\left(X_{t_{i+1}}^{x}+\beta(t_i,x,e)\right)-\widehat{\mathcal{U}}_{i+1}\left(X_{t_{i+1}}^x\right)\right)\gamma(e)\lambda(\mathrm{d}e)\mathrm{d}t\right].
\end{equation*}}
Under Assumption \textbf{(H2)}, recalling that $\widehat{\mathcal{U}}_{i+1}(\cdot)=\mathcal{U}_{i+1}(\cdot,\theta^{\star})$ is $C^2$ with bounded derivatives, { by applying the Malliavin integration-by-parts formula, we can obtain that $\widetilde{v}_{i}(x)$ is $C^{1}$ with}
\begin{align}\label{eq3.21}
\begin{aligned}
D_{x}\widetilde{{v}}_{i}(x)
&= \mathbb{E}\left[
\left(1 + b_{x}(t_{i},x)\Delta t_{i} + \sigma_{x}(t_{i},x)\Delta W_{t_{i}}
+ \int_{t_{i}}^{t_{i+1}}\!\!\int_{E}\beta_{x}(t_i,x,e)\widetilde{\mu}(\mathrm{d}e,\mathrm{d}t)\right)
D_{x}\widehat{\mathcal{U}}_{i+1}(X_{t_{i+1}}^{x})
\right] \\
&= \mathbb{E}\left[D_{x}\widehat{\mathcal{U}}_{i+1}(X_{t_{i+1}}^{x})\right]
+ b_{x}(t_{i},x)\mathbb{E}\left[D_{x}\widehat{\mathcal{U}}_{i+1}(X_{t_{i+1}}^{x})\right]\Delta t_{i} \\
&\quad + \sigma(t_{i},x)\sigma_{x}(t_{i},x)\mathbb{E}\left[D^2_{x}\widehat{\mathcal{U}}_{i+1}(X_{t_{i+1}}^{x})\right]\Delta t_{i} \\
&\quad +{ \mathbb{E}\left[\int_{t_i}^{t_{i+1}}\int_{E}\left(D_{x}{\widehat{\mathcal{U}}}_{i}(X_{t_{i+1}}^{x}+\beta(t_i,x,e))-D_{x}{\widehat{\mathcal{U}}}(X_{t_{i+1}}^{x})\right)\beta_{x}(t_i,x,e)\lambda(\mathrm{d}e)\mathrm{d}t \right]}.
\end{aligned}
\end{align}

{ Next, we focus on $D_{x}\overline{\widehat{z}}_{i}(x)$ and $D_{x}\overline{\widehat{\mathcal{T}}}_{i}(x)$. Firstly, we take the derivative of $\overline{\widehat{z}}_{i}(x)$ with respect to $x$}
\begin{equation*}
\begin{aligned}
&D_x \overline{\widehat{z}}_{i}(x)\\
&= \sigma_x(t_i,x) \mathbb{E}\bigl[ D_x \widehat{\mathcal{U}}_{i+1}(X_{t_{i+1}}^x) \bigr]
+ \sigma(t_i,x) \mathbb{E}\biggl[ \Bigl( 1 + b_{x}(t_i,x)\Delta t_i
+ \sigma_{x}(t_i,x)\Delta W_{t_i} \\
&\quad + \int_{t_i}^{t_{i+1}}\int_E \beta_x(t_i,x,e) \widetilde{\mu}(\mathrm{d}e,\mathrm{d}t) \Bigr) D_x^2 \widehat{\mathcal{U}}_{i+1}(X_{t_{i+1}}^x) \biggr] \\
&= \sigma_x(t_i,x) \mathbb{E}\left[ D_x \widehat{\mathcal{U}}_{i+1}(X_{t_{i+1}}^x) \right]
+ \sigma(t_i,x) \mathbb{E}\left[ D_x^2 \widehat{\mathcal{U}}_{i+1}(X_{t_{i+1}}^x) \right] \\
&\quad + \sigma(t_i,x) b_{x}(t_i,x) \mathbb{E}\left[ D_x^2 \widehat{\mathcal{U}}_{i+1}(X_{t_{i+1}}^x) \right] \Delta t_i +\sigma^2(t_i,x) \sigma_{x}(t_i,x) \mathbb{E}\left[ D_x^3 \widehat{\mathcal{U}}_{i+1}(X_{t_{i+1}}^x) \right] \Delta t_i \\
&\quad + {\sigma(t_i,x) \mathbb{E}\left[\int_{t_i}^{t_{i+1}}\int_E \left(D_{x}^{2}\widehat{\mathcal{U}}_{i+1}(X_{t_{i+1}}^{x}+\beta(t_i,x,e))-D_{x}^{2}\widehat{\mathcal{U}}_{i+1}(X_{t_{i+1}}^{x})\right) \beta_x(t_i,x,e) \lambda(\mathrm{d}e) \mathrm{d}t \right]} .
\end{aligned}
\end{equation*}
{Applying the Cauchy–Schwarz inequality to the last term of the above equation
\begin{equation*}
\begin{aligned}
&\left|\sigma(t_i,x) \mathbb{E}\left[\int_{t_i}^{t_{i+1}}\int_E \left(D_{x}^{2}\widehat{\mathcal{U}}_{i+1}(X_{t_{i+1}}^{x}+\beta(t_i,x,e))-D_{x}^{2}\widehat{\mathcal{U}}_{i+1}(X_{t_{i+1}}^{x})\right) \beta_x(t_i,x,e) \lambda(\mathrm{d}e) \mathrm{d}t \right]\right|^{2}\\
&\leq \Delta {t_i}|\sigma(t_i,x)|^{2} \mathbb{E}\left[\int_{t_i}^{t_{i+1}}\int_E \left|D_{x}^{2}\widehat{\mathcal{U}}_{i+1}(X_{t_{i+1}}^{x}+\beta(t_i,x,e))-D_{x}^{2}\widehat{\mathcal{U}}_{i+1}(X_{t_{i+1}}^{x})\right|^{2} |\beta_x(t_i,x,e)|^{2} \lambda(\mathrm{d}e) \mathrm{d}t \right],
\end{aligned}
\end{equation*}}
and under Assumption \textbf{(H2)}(i), by the linear growth condition of $\sigma$ and the bounds on the derivatives of the neural networks in $\mathcal{N}\mathcal{N}^{\varrho}_{d,1,2,m}(\Theta_{m}^{\gamma})$ in (\ref{eq2.7}), we obtain
{\begin{equation}\label{eq3.22a}
\left|D_{x}\overline{\widehat{z}}_{i}(x)\right|^{2}\leq  C\left(\gamma_{m}^{6}+\gamma_{m}^{8}|\pi|^{2}\right) .
\end{equation}}

Then, taking the derivative of $\overline{\widehat{\mathcal{T}}}_{i}(x)$ with respect to $x$, we have
{
\begin{equation*}
\begin{aligned}
&D_x \overline{\widehat{\mathcal{T}}}_{i}(x)\\
&=\frac{1}{\Delta t_i} \mathbb{E}\biggl[\int_{t_i}^{t_{i+1}}\int_{E}\biggl\{ \left(D_{x}\widehat{\mathcal{U}}_{i+1}(X_{t_{i+1}}^{x}+\beta(t_i,x,e))-D_{x}\widehat{\mathcal{U}}_{i+1}(X_{t_{i+1}}^{x})\right)\\
&\quad\cdot\Bigl( 1 + b_{x}(t_i,x)\Delta t_i
+ \sigma_{x}(t_i,x)\Delta W_{t_i} +\int_{t_i}^{t_{i+1}}\int_E \beta_x(t_i,x,e) \widetilde{\mu}(\mathrm{d}e,\mathrm{d}t) \Bigr)\\
&\quad+D_{x}\widehat{\mathcal{U}}_{i+1}(X_{t_{i+1}}^{x}+\beta(t_i,x,e))\beta_{x}(t_i,x,e)\bigg\} \gamma(e)\lambda(\mathrm{d}e)\mathrm{d}t \biggr] \\
&=\frac{1}{\Delta {t_i}}\mathbb{E}\bigg[\int_{t_i}^{t_{i+1}}\int_{E}\Big(D_{x}\widehat{\mathcal{U}}_{i+1}(X_{t_{i+1}}^{x}+\beta(t_i,x,e))-D_{x}\widehat{\mathcal{U}}_{i+1}(X_{t_{i+1}}^{x}) \\
&\quad + D_{x}\widehat{\mathcal{U}}_{i+1}(X_{t_{i+1}}^{x}+\beta(t_i,x,e))\beta_{x}(t_i,x,e)\Big)\gamma(e)\lambda(\mathrm{d}e)\mathrm{d}t \bigg]\\
&\quad +\mathbb{E}\left[\int_{t_i}^{t_{i+1}}\int_{E}\left(D_{x}\widehat{\mathcal{U}}_{i+1}(X_{t_{i+1}}^{x}+\beta(t_i,x,e))-D_{x}\widehat{\mathcal{U}}_{i+1}(X_{t_{i+1}}^{x})\right)b_{x}(t_i,x)\gamma(e)\lambda(\mathrm{d}e)\mathrm{d}t \right]\\
&\quad +\mathbb{E}\left[\int_{t_i}^{t_{i+1}}\int_{E}\left(D_{x}^{2}\widehat{\mathcal{U}}_{i+1}(X_{t_{i+1}}^{x}+\beta(t_i,x,e))-D_{x}^{2}\widehat{\mathcal{U}}_{i+1}(X_{t_{i+1}}^{x})\right)\sigma(t_i,x)\sigma_{x}(t_i,x)\gamma(e)\lambda(\mathrm{d}e)\mathrm{d}t \right]\\
&\quad +\mathbb{E}\bigg[\int_{E}\beta_{x}(t_i,x,e)\lambda(\mathrm{d}e)\int_{t_i}^{t_{i+1}}\int_{E}\Big(D_{x}\widehat{\mathcal{U}}_{i+1}(X_{t_{i+1}}^{x}+2\beta(t_i,x,e))\\
&\quad-2D_{x}\widehat{\mathcal{U}}_{i+1}(X_{t_{i+1}}^{x}+\beta(t_i,x,e))+D_{x}\widehat{\mathcal{U}}_{i+1}(X_{t_{i+1}}^{x})\Big)\gamma(e)\lambda(\mathrm{d}e)\mathrm{d}t\bigg].
\end{aligned}
\end{equation*}}
{ Similarly to the approach to obtain (\ref{eq3.22a}), applying the Cauchy–Schwarz inequality to the terms of the above equation on $D_x \overline{\widehat{\mathcal{T}}}_{i}(x)$, and then under Assumption \textbf{(H2)}(i), by the linear growth condition of $\sigma$, the boundedness of $\gamma(e)$  (see \ref{eq1.2}), and the bounds on the derivatives of the neural networks in $\mathcal{N}\mathcal{N}^{\varrho}_{d,1,2,m}(\Theta_{m}^{\gamma})$ (see \ref{eq2.7}), we obtain
 \begin{equation}\label{eq3.22b}
\left|D_{x}\overline{\widehat{\mathcal{T}}}_{i}(x)\right|^{2}\leq C\left(\gamma_{m}^{4}+\gamma_{m}^{6}|\pi|^{2}\right).
\end{equation}}

We set $ \widehat{f}_{i}(x)=f\left(t_{i}, x, \widehat{v}_{i}(x),\overline{\widehat{z}}_{i}(x),\overline{\widehat{\mathcal{T}}}_{i}(x)\right)	$. Then
{ it follows by the implicit function theorem, for $|\pi|$ small enough, that $\widehat{v}_{i}(x)$ is $C^{1}$ with derivative given by}
\begin{eqnarray*}
\begin{aligned}
D_{x}\widehat{v}_{i}(x)&=D_{x}\widetilde{{v}}_{i}(x)+\Delta t_{i}\left(D_{x}\widehat{f}_{i}(x)+D_{y}\widehat{f}_{i}(x)D_{x}\widehat{v}_{i}(x)+D_{z}\widehat{f}_{i}(x)D_{x}\overline{\widehat{z}}_{i}(x)+D_{\Gamma}\widehat{f}_{i}(x)D_{x}\overline{\widehat{\mathcal{T}}}_{i}(x)\right).
\end{aligned}
\end{eqnarray*}
{ By (\ref{eq3.21}) and multiplying $\sigma(t_{i},x)$ on both sides (see \cite{CHX}), we have}
\begin{eqnarray*}
\begin{aligned}
&\left(1-\Delta t_{i}D_{y}\hat{f}_{i}(x)\right)\sigma(t_i,x)D_{x}\widehat{v}_{i}(x)\\
&\quad=\overline{\widehat{z}}_{i}(x)+\Delta t_{i}\sigma(t_i,x)\left(D_{x}\widehat{f}_{i}(x)+D_{z}\widehat{f}_{i}(x)D_{x}\overline{\widehat{z}}_{i}(x)+D_{\Gamma}\widehat{f}_{i}(x)D_{x}\overline{\widehat{\mathcal{T}}}_{i}(x)\right)\\
&\qquad+\Delta{t_i}\sigma(t_i,x)\left(b_{x}(t_i,x)\mathbb{E}\left[D_{x}\widehat{\mathcal{U}}_{i+1}(X_{t_{i+1}}^{x})\right]+\sigma(t_i,x)\sigma_{x}(t_i,x)\mathbb{E}\left[D_{x}^{2}\widehat{\mathcal{U}}_{i+1}(X_{t_{i+1}}^{x})\right]\right)\\
&\qquad +{\sigma(t_i,x)\mathbb{E}\left[\int_{t_i}^{t_{i+1}}\int_{E}\left(D_{x}\widehat{\mathcal{\mathcal{U}}}_{i+1}(X_{t_{i+1}}^{x}+\beta(t_i,x,e))-D_{x}\widehat{\mathcal{\mathcal{U}}}_{i+1}(X_{t_{i+1}}^{x})\right)\beta_{x}(t_i,x,e)\lambda(\mathrm{d}e)\mathrm{d}t \right]}.
\end{aligned}
\end{eqnarray*}
{ Under Assumption \textbf{(H2)}(ii), by the linear growth condition of $\sigma$ and (\ref{eq3.21}),(\ref{eq3.22a}) and (\ref{eq3.22b}), we have}
{\begin{eqnarray*}
\begin{aligned}
	\mathbb{E}\left|\sigma(t_i,X_{t_i})D_{x}\widehat{v}_{{i}}(X_{t_i})-\overline{\widehat{Z}}_{t_{i}}\right|^2=C\left(\gamma_{m}^{4}+\gamma^6_{m}+\gamma^8_{m}|\pi|^2\right)|\pi|^2.
\end{aligned}
\end{eqnarray*}}
We should note that the above arguments can be extended to the $d-$ dimensional $(d>1)$ setting.

Fix $i\in{0,...N-1}$, and take $\theta^{\star}_{i}=(\xi^{\star}_{i},\eta^{\star}_{i})\in argmin_{\theta}\widehat{L}_{i}(\theta)$, so that $\widehat{\mathcal{U}}_{i}(X_{t_i})=\mathcal{U}_{i}(X_{t_{i}}; \xi^{\star}_{i}),\widehat{Z}_{i}(X_{t_i})=\sigma^{T}(t_i,X_{t_{i}})D_{x}\mathcal{U}_{i}(X_{t_{i}},\xi^{\star}_{i})$ and { $\widehat{\mathcal{G}}_{i}(X_{t_i},e)=\mathcal{G}(X_{t_{i}},e;\eta^{\star}_{i})$}. We then have for all $\theta=(\xi,\eta)$,
\begin{equation*}
\begin{aligned}
 (1-C\Delta t_{i})&E|\widehat{\mathcal{V}}_{t_{i}}-\widehat{\mathcal{U}}_{i}(X_{t_{i}})|^2+\frac{\Delta t_{i}}{2}E\left|\overline{\widehat{Z}}_{t_{i}}-\widehat{\mathcal{Z}}_{i}(X_{t_{i}})\right|^2\\
 &+\frac{\Delta t_{i}}{2}{\mathbb{E}\left[\int_{E}\left(\overline{\widehat{U}}_{i}(X_{t_i},e)-\widehat{\mathcal{G}}_{i}(X_{t_i},e)\right)^2\lambda(\mathrm{d}e)\right]}\leq \widetilde{L}_{i}(\theta^{\star}_{i})\leq \widetilde{L}_{i}(\theta)\\
 \leq&(1+C\Delta t_{i})\mathbb{E}|\widehat{\mathcal{V}}_{t_{i}}-\mathcal{U}_{i}(X_{t_{i}},\xi)|^2+C\Delta t_{i}\mathbb{E}\left|\overline{\widehat{Z}}_{t_{i}}-\sigma^{T}(t_i,X_{t_{i}})D_{x}\mathcal{U}_{i}(X_{t_{i}},\xi)\right|^2\\&+C\Delta t_{i}{\mathbb{E}\left[\int_{E}\left(\overline{\widehat{U}}_{i}(X_{t_i},e)-\mathcal{G}_{i}(X_{t_{i}},e;\eta)\right)^2\lambda(\mathrm{d}e)\right]}.
  \end{aligned}
\end{equation*}
It is straightforward to obtain
\begin{equation*}
\begin{aligned}
&\Delta t_i \mathbb{E}\left|\overline{\widehat{Z}}_{t_i}-\sigma^{T}(t_i,X_{t_i})D_{x}\widehat{\mathcal{U}}_{i}(X_{t_i};\xi)\right|^{2} \\
&\qquad \leq 2\Delta t_i \mathbb{E}\left|\overline{\widehat{Z}}_{t_i}-\sigma^{T}(t_i,X_{t_i})D_{x}\widehat{v}_{i}(X_{t_i})\right|^{2}+2\Delta t_{i} \mathbb{E}\left|\sigma^{T}(t_i,X_{t_i})D_{x}\widehat{\mathcal{U}}_{i}(X_{t_i},\xi)-\sigma^{T}(t_i,X_{t_i})D_{x}\widehat{\mathcal{V}}_{t_i}\right|^{2}.
\end{aligned}
\end{equation*}
As a consequence, we have, for $|\pi|$ small enough,
\begin{equation}\label{eq3.23}
\begin{aligned}
&\mathbb{E}|\widehat{\mathcal{V}}_{t_{i}}-\widehat{\mathcal{U}}_{i}(X_{t_{i}})|^2+\Delta t_{i}\mathbb{E}\left|\overline{\widehat{Z}}_{t_{i}}-\widehat{\mathcal{Z}}_{i}(X_{t_{i}})\right|^2+\Delta t_{i}\mathbb{E}\left[\int_{E}\left(\overline{\widehat{U}}_{i}(X_{t_i},e)-{\widehat{\mathcal{G}}_{i}(X_{t_i},e)}\right)^2\lambda(\mathrm{d}e)\right]\\
&\leq C \left[ \varepsilon_{i}^{\mathcal{N},v}+|\pi|\varepsilon_{i}^{\mathcal{N},\Gamma}+{(\gamma_{m}^{4}+\gamma^6_{m}+\gamma^8_{m} |\pi|^2)}|\pi|^3\right].
\end{aligned}
\end{equation}
Then, by substituting (\ref{eq3.23}) into (\ref{eq3.16}), for all $\theta=(\xi,\eta)$, we have
\begin{equation}\label{eq3.24}
\begin{aligned}
&\max\limits_{i=0,...,N-1}\mathbb{E}\left|Y_{t_{i}}-\widehat{\mathcal{U}}_{i}(X_{t_{i}})\right|^2\\
&\qquad\leq C\left(E|g(\mathcal{X}_{T})-g(X_{T})|^2+{ \frac{\gamma_{m}^{4}+\gamma_{m}^{6}+\gamma_{m}^{8}|\pi|^{2}}{N}}+\varepsilon^{Z}(\pi)+\varepsilon^{\Gamma}(\pi)+
{\sum_{i=0}^{N-1}\left(N\varepsilon^{\mathcal{N},v}_{i}+\varepsilon^{\mathcal{N},\Gamma}_{i}\right)}\right).
\end{aligned}
\end{equation} \\

%%%%%%%%%%%%%%%%%%%%%%%%%%%%%%%% Step 5 %%%%%%%%%%%%%%%%%%%%%%%%%%%%%%%%

{\bf Step 5}: Now we prove the convergence of $Z$ and $\Gamma$. According to the $L^2$ regularity of $Z$ and $\Gamma$, we have the following two inequalities for $Z$ and $\Gamma$
\begin{eqnarray}\label{eq3.25}
\begin{aligned}
&\mathbb{E}\left[\int_{t_{i}}^{t_{i+1}}\left|Z_{t}-\overline{\widehat{Z}}_{t_{i}}\right|^2dt\right]\\
\quad\leq&  2d|\pi|\mathbb{E}\left[\int_{t_{i}}^{t_{i+1}}{|f(t,X_{t}, Y_{t}, Z_{t}, \Gamma_{t})|^{2}}dt\right]+ 2\mathbb{E}\left[\int_{t_{i}}^{t_{i+1}}\left|Z_{t}-\overline{{Z}}_{t_{i}}\right|^2dt\right]\\
&+ 2d\left( \mathbb{E}\left|Y_{t_{i+1}}-\widehat{\mathcal{U}}_{i+1}(X_{t_{i+1}})\right|^2-\mathbb{E}\left|\mathbb{E}_{i}\left[Y_{t_{i+1}}-\widehat{\mathcal{U}}_{i+1}(X_{t_{i+1}})\right]\right|^2\right),
\end{aligned}
\end{eqnarray}
and
\begin{eqnarray}\label{eq3.26}
\begin{aligned}
&\mathbb{E}\left[\int_{t_{i}}^{t_{i+1}}\left|\Gamma_{t}-\overline{\widehat{\Gamma}}_{t_{i}}\right|^2dt \right] \\
\qquad\leq& 2|\pi| \int_{E}\gamma^2(e)\lambda(de)\mathbb{E}\left[\int_{t_{i}}^{t_{i+1}}{|f(t,X_{t}, Y_{t}, Z_{t}, \Gamma_{t})|^{2}}dt\right]+2\mathbb{E}\left[\int_{t_{i}}^{t_{i+1}}\left|\Gamma_{t}-\overline{\Gamma}_{t_{i}}\right|^2dt \right]\\
&+2\int_{E}\gamma^2(e)\lambda(de)\left( \mathbb{E}\left|Y_{t_{i+1}}-\widehat{\mathcal{U}}_{i+1}(X_{t_{i+1}})\right|^2-\mathbb{E}\left|\mathbb{E}_{i}\left[Y_{t_{i+1}}-\widehat{\mathcal{U}}_{i+1}(X_{t_{i+1}})\right]\right|^2\right).
\end{aligned}
\end{eqnarray}
{ Summing from $0$ to $N-1$ over $i$ for equations (\ref{eq3.25}) and (\ref{eq3.26}), and adding the two sums together lead to}
{
\begin{equation}\label{eq3.27}
\begin{aligned}
&\mathbb{E}\left[\sum_{i=0}^{N-1}\int_{t_{i}}^{t_{i+1}}\left|Z_{t}-\overline{\widehat{Z}}_{t_{i}}\right|^2\mathrm{d}t\right] + \mathbb{E}\left[\sum_{i=0}^{N-1}\int_{t_{i}}^{t_{i+1}}\left|\Gamma_{t}-\overline{\widehat{\Gamma}}_{t_{i}}\right|^2\mathrm{d}t \right]\\
\quad &\leq 2 C|\pi|+\varepsilon^{Z}(\pi)+\varepsilon^{\Gamma}(\pi)+\left(2d+2\int_{E}\gamma^{2}(e)\lambda(\mathrm{d}e)\right)\mathbb{E}|g(\mathcal{X}_{T})-g(X_{T})|^{2}\\
&\quad +\left(2d+2\int_{E}\gamma^{2}(e)\lambda(\mathrm{d}e)\right)\sum_{i=0}^{N-1}\left( \mathbb{E}\left|Y_{t_{i}}-\widehat{\mathcal{U}}_{i}(X_{t_{i}})\right|^2-\mathbb{E}\left|\mathbb{E}_{i}\left[Y_{t_{i+1}}-\widehat{\mathcal{U}}_{i+1}(X_{t_{i+1}})\right]\right|^2\right).
\end{aligned}
\end{equation}}
Moreover, by combining (\ref{eq3.15}) and (\ref{eq3.9}), we have
\begin{eqnarray*}
\begin{aligned}
&{\left(2d+2\int_{E}\gamma^{2}(e)\lambda(\mathrm{d}e)\right)}\left( \mathbb{E}\left|Y_{t_{i}}-\widehat{\mathcal{U}}_{i}(X_{t_{i}})\right|^2-\mathbb{E}\left|\mathbb{E}_{i}\left[Y_{t_{i+1}}-\widehat{\mathcal{U}}_{i+1}(X_{t_{i+1}})\right]\right|^2\right)\\
\qquad \leq&  \left(\frac{1+r |\pi|}{1-|\pi|}-1\right){\left(2d+2\int_{E}\gamma^{2}(e)\lambda(\mathrm{d}e)\right)} \mathbb{E}\left| \mathbb{E}_{i}\left[Y_{t_{i+1}}-\widehat{\mathcal{U}}_{i+1}(X_{t+1})\right]\right|^2\\
& +\frac{{(10d+10\int_{E}\gamma^{2}(e)\lambda(\mathrm{d}e))}f_{L}^2}{r}\frac{1+r |\pi|}{1-|\pi|}\left \{C|\pi|^2+2\mathbb{E}\left[\int_{t_{i}}^{t_{i+1}}\left|Y_{t}-Y_{t_{i}}\right|^2\mathrm{d}t\right]\right. \\
&+\left. 2|\pi|\mathbb{E}\left|Y_{t}-\widehat{\mathcal{V}}_{t_{i}}\right|^2+\mathbb{E}\left[\int_{t_{i}}^{t_{i+1}}\left|Z_{t}-\overline{\widehat{Z}}_{t_{i}}\right|^2\mathrm{d}t\right]
+\mathbb{E}\left[\int_{t_{i}}^{t_{i+1}}\left|\Gamma_{t}-\overline{\widehat{\Gamma}}_{t_{i}}\right|^2\mathrm{d}t\right]\right\}\\
&+\frac{{\left(2d+2\int_{E}\gamma^{2}(e)\lambda(\mathrm{d}e)\right)}}{(1-|\pi|)|\pi|}\mathbb{E}|\widehat{\mathcal{V}}_{t_{i}}-\widehat{\mathcal{U}}_{i}(X_{t_{i}})|^2.
\end{aligned}
\end{eqnarray*}
{Taking $r=\left(30d+30\int_{E}\gamma^{2}(e)\lambda(de)\right)f_{L}^2$ for the above inequality, so that $\frac{\left(10d+10\int_{E}\gamma^{2}(e)\lambda(de)\right)f_{L}^2}{r}(1+r|\pi|)/(1-|\pi|)<1/2$ for $|\pi|$ small enough}, and by plugging into (\ref{eq3.25}) and (\ref{eq3.27}), we have, in view of $(1+r|\pi|)/(1-|\pi|)-1< O(|\pi|)$,
{
\begin{equation} \label{eq3.28}
 \begin{aligned}
&\mathbb{E}\left[\sum_{i=0}^{N-1}\int_{t_{i}}^{t_{i+1}}\left|Z_{t}-\overline{\widehat{Z}}_{t_{i}}\right|^2\mathrm{d}t\right] + \mathbb{E}\left[\sum_{i=0}^{N-1}\int_{t_{i}}^{t_{i+1}}\left|\Gamma_{t}-\overline{\widehat{\Gamma}}_{t_{i}}\right|^2\mathrm{d}t \right]\\
&\quad \leq C\bigg\{\varepsilon^{Z}(\pi)+\varepsilon^{\Gamma}(\pi)+|\pi|+\mathbb{E}|g(\mathcal{X}_{T})-g(X_{T})|^{2}  \\
&\qquad+N\sum_{i=0}^{N-1}\mathbb{E}|\widehat{\mathcal{U}}_{i}(X_{t_{i}})-\widehat{\mathcal{V}}_{t_{i}}|^{2}+|\pi|\sum_{i=0}^{N-1}\mathbb{E}|Y_{t_i}-\widehat{\mathcal{V}}_{t_{i}}|^{2}\\
&\qquad + |\pi|\sum_{i=0}^{N-1}\mathbb{E}\left|\mathbb{E}[Y_{t_{i+1}}-\widehat{\mathcal{U}}_{i+1}(X_{t_{i+1}})] \right|^{2}\bigg\}.
\end{aligned}
\end{equation}}
{By (\ref{eq3.14}) and the $L^{2}$ regularity of $Y$, $Z$ and $\Gamma$, we have
\begin{equation}\label{eq3.29}
|\pi|\sum_{i=0}^{N-1}\mathbb{E}|Y_{t_{i}}-{\mathcal{V}_{t_{i}}}|^{2}
\leq C \max_{i=0,\cdots,N} \mathbb{E}|Y_{t_{i}}-\widehat{\mathcal{U}}_{i}(X_{t_{i}})|^{2}.
\end{equation}}
{In view of the properties of mathematical expectation, we get
 \begin{equation}\label{eq3.30}
\begin{aligned}
|\pi| \sum_{i=0}^{N-1}\mathbb{E}\left|\mathbb{E}[Y_{t_{i+1}}-\widehat{\mathcal{U}}_{i+1}(X_{t_{i+1}})]  \right|^{2}&\leq |\pi|\sum_{i=1}^{N-1}\mathbb{E}\left|Y_{t_{i+1}}-\widehat{\mathcal{U}}_{i+1}(X_{t_{i+1}})\right|^{2}\\
&\leq C \max_{i=0,\cdots,N} \mathbb{E}\left|Y_{t_{i}}-\widehat{\mathcal{U}}_{i}(X_{t_i})\right|^{2}.
\end{aligned}
\end{equation}}
{ By combining (\ref{eq3.23}), (\ref{eq3.24}), (\ref{eq3.28}), (\ref{eq3.29}) and (\ref{eq3.30})}, {we have
\begin{equation}\label{eq3.31}
\begin{aligned}
&\mathbb{E}\left[\sum_{i=0}^{N-1}\int_{t_{i}}^{t_{i+1}}\left|Z_{t}-\overline{\widehat{Z}}_{t_{i}}\right|^2dt\right] + \mathbb{E}\left[\sum_{i=0}^{N-1}\int_{t_{i}}^{t_{i+1}}\left|\Gamma_{t}-\overline{\widehat{\Gamma}}_{t_{i}}\right|^2dt \right] \\
& \quad \leq C \left[\varepsilon^{Z}(\pi)+\varepsilon^{\Gamma}(\pi)+\frac{\gamma_{m}^{4}+\gamma_{m}^{6}+\gamma_{m}^{8}|\pi|^{2}}{N}
+\mathbb{E}\left|g(\mathcal{X}_{T})-g(X_{T})\right|^{2}+\sum_{i=0}^{N-1}(N\varepsilon_{i}^{\mathcal{N},v}+\varepsilon_{i}^{\mathcal{N},\Gamma})\right].
\end{aligned}
\end{equation}}

It is straightforward to obtain the following two inequalities for $Z$ and $\Gamma$
\begin{equation}\label{eq3.32}
\begin{aligned}
&\mathbb{E}\left[\int_{t_{i}}^{t_{i+1}}|Z_{t}-\widehat{{Z}}_{i}(X_{t_{i}})|^2dt \right]\\
&\qquad \leq 2\mathbb{E}\left[\int_{t_{i}}^{t_{i+1}}\left|Z_{t}-\overline{\widehat{Z}}_{t_i}\right|^2 dt\right]+{2 \Delta t_{i} \mathbb{E}\left[\left|\overline{\widehat{Z}}_{t_{i}}-\widehat{Z}_{i}(X_{t_{i}})\right|^2\right]},
\end{aligned}
\end{equation}
and
\begin{equation}\label{eq3.33}
\begin{aligned}
&\mathbb{E}\left[\int_{t_{i}}^{t_{i+1}}|\Gamma_{t}-\widehat{\mathcal{T}}_{i}(X_{t_{i}})|^2 \mathrm{d}t\right]\\
 & \qquad \leq 2\mathbb{E}\left[\int_{t_{i}}^{t_{i+1}}\left|\Gamma_{t}-\overline{\widehat{\Gamma}}_{t_{i}}\right|^2\mathrm{d}t\right]+{2 \Delta t_{i}\mathbb{E}\left[\left|\overline{\widehat{\Gamma}}_{t_{i}}-\widehat{\mathcal{T}}_{i}(X_{t_{i}})\right|^2\right]}\\
 &\qquad {\leq 2\mathbb{E}\left[\int_{t_{i}}^{t_{i+1}}\left|\Gamma_{t}-\overline{\widehat{\Gamma}}_{t_{i}}\right|^2\mathrm{d}t\right] + 2C \Delta t_{i}\mathbb{E}\left[\int_{E}\left(\overline{\widehat{U}}_{i}(X_{t_i},e)-\widehat{\mathcal{G}}_{i}(X_{t_{i}},e)\right)^2\lambda(\mathrm{d}e)\right]}
\end{aligned}
\end{equation}
{Summing from $0$ to $N-1$ about $i$ for (\ref{eq3.32}) and (\ref{eq3.33}), { by combining} (\ref{eq3.31}) and (\ref{eq3.23}), we can prove the convergence of $Z$ and $\Gamma$:}
\begin{equation*}
\begin{aligned}
&\mathbb{E}\left[\sum_{i=0}^{N-1}\int_{t_{i}}^{t_{i+1}}|Z_{t}-\widehat{\mathcal{Z}}_{i}(X_{t_{i}})|^2dt \right]+\mathbb{E}\left[\sum_{i=0}^{N-1}\int_{t_{i}}^{t_{i+1}}|\Gamma_{t}-\widehat{\mathcal{T}}_{i}(X_{t_{i}})|^2 \mathrm{d}t\right]\\
&\qquad \leq 2\mathbb{E}\left[\sum_{i=0}^{N-1}\int_{t_{i}}^{t_{i+1}}\left|Z_{t}-\overline{\widehat{Z}}_{t_i}\right|^2 dt\right]+2\Delta t_{i}\sum_{i=0}^{N-1}\left(\mathbb{E}\left[\left|\overline{\widehat{Z}}_{t_{i}}-\widehat{Z}_{i}(X_{t_{i}})\right|^2\right]\right)\\
&\qquad \quad+2\mathbb{E}\left[\sum_{i=0}^{N-1}\int_{t_{i}}^{t_{i+1}}\left|\Gamma_{t}-\overline{\widehat{\Gamma}}_{t_{i}}\right|^2\mathrm{d}t\right]+{2C\Delta t_{i}\sum_{i=0}^{N-1}\left(\mathbb{E}\left[\int_{E}\left(\overline{\widehat{U}}_{i}(X_{t_i},e)-\widehat{\mathcal{G}}_{i}(X_{t_{i}},e)\right)^2\lambda(\mathrm{d}e)\right]\right)}\\
&\qquad\leq C\left(\mathbb{E}|g(\mathcal{X}_{T})-g(X_{T})|^2+\frac{\gamma_{m}^{4}+\gamma_{m}^{6}+\gamma_{m}^{8}|\pi|^2}{N}+\varepsilon^{Z}(\pi)+\varepsilon^{\Gamma}(\pi)+{\sum_{i=0}^{N-1}\left(N\varepsilon^{\mathcal{N},v}_{i}+\varepsilon^{\mathcal{N},\Gamma}_{i}\right)}\right)\\
&\qquad \quad +C\left(\sum_{i=0}^{N-1}\varepsilon_{i}^{\mathcal{N},v}+|\pi|\sum_{i=0}^{N-1}\varepsilon_{i}^{\mathcal{N},\Gamma}+\left(\gamma_{m}^{4}+\gamma^{6}_{m}+|\pi|^2\gamma^8_{m}\right)|\pi|^2\right)\\
&\qquad \leq C\left(\mathbb{E}|g(\mathcal{X}_{T})-g(X_{T})|^2+\frac{\gamma_{m}^{4}+\gamma_{m}^{6}+\gamma_{m}^{8}|\pi|^{2}}{N}+\varepsilon^{Z}(\pi)+\varepsilon^{\Gamma}(\pi)+{\sum_{i=0}^{N-1}\left(N\varepsilon^{\mathcal{N},v}_{i}+\varepsilon^{\mathcal{N},\Gamma}_{i}\right)}\right).
\end{aligned}
\end{equation*}

{From (\ref{eq2.6}) we know that $\frac{\gamma_{m}^{4}}{N}=o\left(\frac{\gamma_{m}^{6}}{N}\right)$ and $\frac{\gamma_{m}^{8}|\pi|^{2}}{N}=o\left(\frac{\gamma_{m}^{6}}{N}\right)$ as $m,N\rightarrow \infty$, and thus $\frac{\gamma_{m}^{4}+\gamma_{m}^{6}+\gamma_{m}^{8}|\pi|^{2}}{N}=O\left(\frac{\gamma_{m}^{6}}{N}\right)$ as $m,N\rightarrow \infty$}. Recalling the convergence of $Y$ (see (\ref{eq3.24}) in \textbf{Sept 4}), we complete the proof of the theorem.
\end{proof}

{ \textbf{Remark 3.1}}  The error contributions for the DFBDP scheme in the r.h.s. of estimation (\ref{eq3.8}) consists of {five terms}. The first four terms correspond to the time discretization of BSDEJs, namely (i) the strong approximation of the terminal condition (depending on the forward scheme and the terminal condition $g(\cdot)$), and converging to zero as $|\pi|$ goes to zero when $g(\cdot)$ is Lipschitz; (ii) the strong approximation of the forward Euler scheme, and the $L^{2}-$ regularity of $Y$, which gives a convergence of order $O(|\pi|)$, equivalently $O(N^{-1})$; (iii) the $L^{2}-$ regularity of $Z$, which converges to zero as $|\pi|$ goes to zero when $g(\cdot)$ is Lipschitz; {(iv) the $L^{2}-$ regularity of $\Gamma$, which also converges to zero as $|\pi|$ goes to zero when $g(\cdot)$ is Lipschitz.} Finally, the better the neural networks are able to approximate the functions ${\widehat{v}}_{t_i}(\cdot)$ and $\overline{\widehat{\mathcal{U}}}_{t_i}(\cdot,e)$ at each time $i=0,\cdots,N-1$, the smaller is the last term in the error estimation. Moreover, the number of parameters of the employed deep neural networks grows at most polynomially in the PIDE dimension.

%%%%%%%%%%%%%%%%%%%%%%%%%%%%%%%%%%%%%%%%第4章 数值实验%%%%%%%%%%%%%%%%%%%%%%%%%%%%%%%%%%%%%%%%%%%%

\section{Numerical examples}
In this section, we will explore two numerical examples to illustrate the effectiveness of the deep learning-based DFBDP algorithm.

\subsection{One-dimensional problem}
We first consider a one-dimensional problem. Let $
  g(T,x)=\sin(X_{T})$ and
  \begin{eqnarray*}
  b(t,X_{t})=0,~~ \sigma(t,X_{t})=1,~~ \beta(t,X_{t^{-}},e)=e,~~ d=1,
  \end{eqnarray*}
such that the corresponding PIDE is
\begin{equation}\label{4.1}
\begin{cases}
\begin{aligned}
\frac{\partial u}{\partial t} + \frac{1}{2}\frac{\partial^{2} u}{\partial x^{2}}
&+ \int_{\mathbb{E}}\left(u(t,x+e) - u(t,x) - e\frac{\partial u}{\partial x}(t,x)\right)\lambda(\mathrm{d}e) \\
&+ f\left(t,x,u,\sigma^\mathsf{T}\nabla_{x}u,B[u]\right) = 0,
\end{aligned} \\
u(T,x) = \sin(x),
\end{cases}
\end{equation}
and the exact solution of the above corresponding PIDE is $u(t,x)=e^{t-1} \sin(x)$. The compensated Poisson random measure:
  \begin{eqnarray*}
    \lambda(\mathrm{d}e) :=\lambda\rho(e)\mathrm{d}e,
  \end{eqnarray*}
where $\rho(e)$ is the density function of a distribution. Then the corresponding FBSDEJs of PIDE (\ref{4.1}) is
\begin{eqnarray}\label{4.2}
\left\{
\begin{aligned}
	\mathrm{d}X_{t}=& \mathrm{d}W_{t}+\int_{E}e\tilde{\mu}(\mathrm{d}e,\mathrm{d}t),\\
    -\mathrm{d}Y_{t}=& f\left(t,X_{t},Y_{t},Z_{t},\Gamma_{t}\right) \mathrm{d}t- Z_{t}^{T}\mathrm{d}W_{t}-\int_{E}U_{t}(e)\tilde{\mu}(\mathrm{d}e,\mathrm{d}t).
\end{aligned}
\right.
\end{eqnarray}

The form of the Lévy measure plays a decisive role in non-local terms. Here, we aim to evaluate the performance of our method under different form of Lévy measures $\rho(e)$. The following four different jump distributions $\rho(e)$ are utilized:
\begin{enumerate}
    \item Normal distribution with $\mu = 0.4$, $\sigma = 0.25$:
    \begin{align*}
\rho(e) & = \frac{1}{\sqrt{2\pi}\sigma} e^{-\frac{1}{2} \left(\frac{e-\mu}{\sigma}\right)^2}, e \in \mathbb{R};
\end{align*}
    \item Uniform distribution with $\delta = 0.7$:
    \begin{align*}
\rho(e) & =
\begin{cases}
\frac{1}{2\delta}, & -\delta \leq e \leq \delta \\
0, & \text{else}
\end{cases};
\end{align*}
    \item Exponential distribution with $\lambda_0 = 3$:
    \begin{align*}
\rho(e) & =
\begin{cases}
\lambda_0 e^{-\lambda_0 e}, & e \geq 0 \\
0, & e < 0
\end{cases};
\end{align*}
    \item Bernoulli distribution with $a_1 = -0.4$, $a_2 = 0.8$, $p = 0.5$:
    \begin{align*}
\rho(e) & =
\begin{cases}
p, & e = a_1 \\
1 - p, & e = a_2
\end{cases}.
\end{align*}
\end{enumerate}

Corresponding the above Lévy measures (1),(2),(3) and (4) , the term $f\left(t,x,u,\sigma^{T}(t,x)D_{x}u,B[u]\right)$ of the corresponding PIDEs are formulated as follows
\begin{align*}
f\left(t,x,u,\sigma^\mathsf{T}\nabla_{x}u,B[u]\right) & = -\frac{u \exp(\frac{\partial u}{\partial x})}{\exp(e^{t-1}\cos(x))}+0.5u-\lambda e^{t-1}\left(\frac{\sin(\delta)}{\delta}-1\right)\sin(x), \\
f\left(t,x,u,\sigma^\mathsf{T}\nabla_{x}u,B[u]\right) & = -\frac{u \exp(\frac{\partial u}{\partial x})}{\exp(e^{t-1}\cos(x))}+0.5u-\lambda e^{t-1}\left(e^{-0.5\sigma^{2}}-1\right)\sin(x)+\lambda \mu \frac{\partial u}{\partial x}, \\
f\left(t,x,u,\sigma^\mathsf{T}\nabla_{x}u,B[u]\right) & = -\frac{u \exp(\frac{\partial u}{\partial x})}{\exp(e^{t-1}\cos(x))}+0.5u+\frac{ \lambda}{\lambda_{0}} \frac{\partial u}{\partial x} \\
& \quad -\lambda e^{t-1}\left(\frac{\lambda_{0}}{\lambda_{0}^{2}+1}\cos(x)-\frac{1}{\lambda_{0}^{2}+1}\sin(x)\right), \\
f\left(t,x,u,\sigma^\mathsf{T}\nabla_{x}u,B[u]\right) & = -\frac{u \exp(\frac{\partial u}{\partial x})}{\exp(e^{t-1}\cos(x))}+0.5u+(p_1 a_1+p_2a_2) \lambda \frac{\partial u}{\partial x} \\
& \quad -\lambda e^{t-1}[p_{1}\sin(x+a_1)+p_{2}\sin(x+a_2)-\sin(x)],
\end{align*}
respectively.

Now, let us set $\lambda=1$, $N=30$ and set $1$ hidden layers with $1+20$ dimensional. The neural network for approximating the solution has a $1$-dimensional input and output, while the network for approximating the integral kernel of non-local terms has a $2$-dimensional input and a $1$-dimensional output.

To demonstrate the effectiveness of DFBDP algorithm for the FBSDEJs with different Lévy measures, Table 1 depicts average value of $u(0,X_{0})$ and standard deviation of $u(0,X_{0})$ based on $1000$ Monte Carlo samples and $10$ independent runs. It is observed from Table 1 that we can obtain a good approximation of $u(0,X_{0})$ by using the DFBDP algorithm after enough iterations of parameters updates.

To further demonstrate the neural network’s accuracy in approximating the solution function at different time instances with different Lévy measures, we compare the true process $Y_{t_i}=u(t_i,X_{t_i})$ with the corresponding neural network approximations $\hat{\mathcal{U}}_{i}(X_{t_i})$, and compare the exact $u(t,x)$ and $D_{x}u(t,x)$ with the corresponding neural network approximations at drfferent time instances.

 Depicted in the four subplots (A,B,C,D) of Figure 1 are the exact solutions and the approximate solutions of FBSDEJs (\ref{4.2}) under different Lévy measures (1)- (4), where the blue lines represent the exact $Y_t$ and the deep red dashed lines represent its estimates. { When simulating the sample paths of the exact and the predicted $Y_t$, jumps of the Poisson process within certain subintervals $[t_i,t_{i+1}]$ lead to the sample paths exhibit jumps within certain subintervals. Thus, the green bold lines represent the sample paths of exact $Y_{t}$ exhibit jumps within certain subintervals, and the yellow bold lines represent the sample paths of predicted $Y_{t}$ exhibit jumps within certain subintervals.} From Figure 1, we can observe that the DFBDP algorithm performs well in approximating the solutions to the FBSDEJs (\ref{4.2}) under different Lévy measures.

 {Moreover, depicted in Figures 2-5 are the estimates of the solution $u(t,x)$ of the PIDE (\ref{4.1}) at $t=0,~0.33,~0.66,~0.96$ under different Lévy measures (subplots A,B,C,D). In addition, we have plotted the graph of the gradient of the solution $u(t,x)$, denoted by $D_{x}u(t,x)$ (subplots E,F,G,H).} { From these figures, we can observe that} the DFBDP algorithm performs well in approximating the solutions and their gradients to the PIDE (\ref{4.1}) under different Lévy measures.

%%%%%%%%%%%%%%%%%%%%%%%%%%%%%%L1误差表%%%%%%%%%%%%%%%%%%%%%%%%%%%%%%%%%%%%
\begin{table}[h]
\centering
\caption{Estimate of \( u(0, x_0) \) where \( d = 1 \) and \( x_0 = \pi/2 \). Average and standard deviation observed over 10 independent runs are reported. The theoretical solution is 0.3679.}
\label{tab:4.1}
\begin{tabular}{lccc}
\toprule
 \textbf{Lévy measure} & \textbf{Averaged value} & \textbf{Standard deviation} & \textbf{Relative L1 Error} \\
\midrule
\textbf{Uniform} & 0.3599 & 0.00132 & 0.02174 \\
\textbf{Normal} & 0.3580 & 0.00101 & 0.02691 \\
\textbf{Exponential} & 0.3711 & 0.00098 & 0.00870 \\
\textbf{Bernoulli} & 0.3861 & 0.00147 & 0.04950 \\
\bottomrule
\end{tabular}
\end{table}

%%%%%%%%%%%%%%%%%%%%%%%%%绘制不同Levy测度下Yt的近似图%%%%%%%%%%%%%%%%%%%%%%%%%%%
\begin{figure}[htb]
    \centering
    \begin{minipage}{0.47\linewidth}
        \centering
        \includegraphics[width=\linewidth]{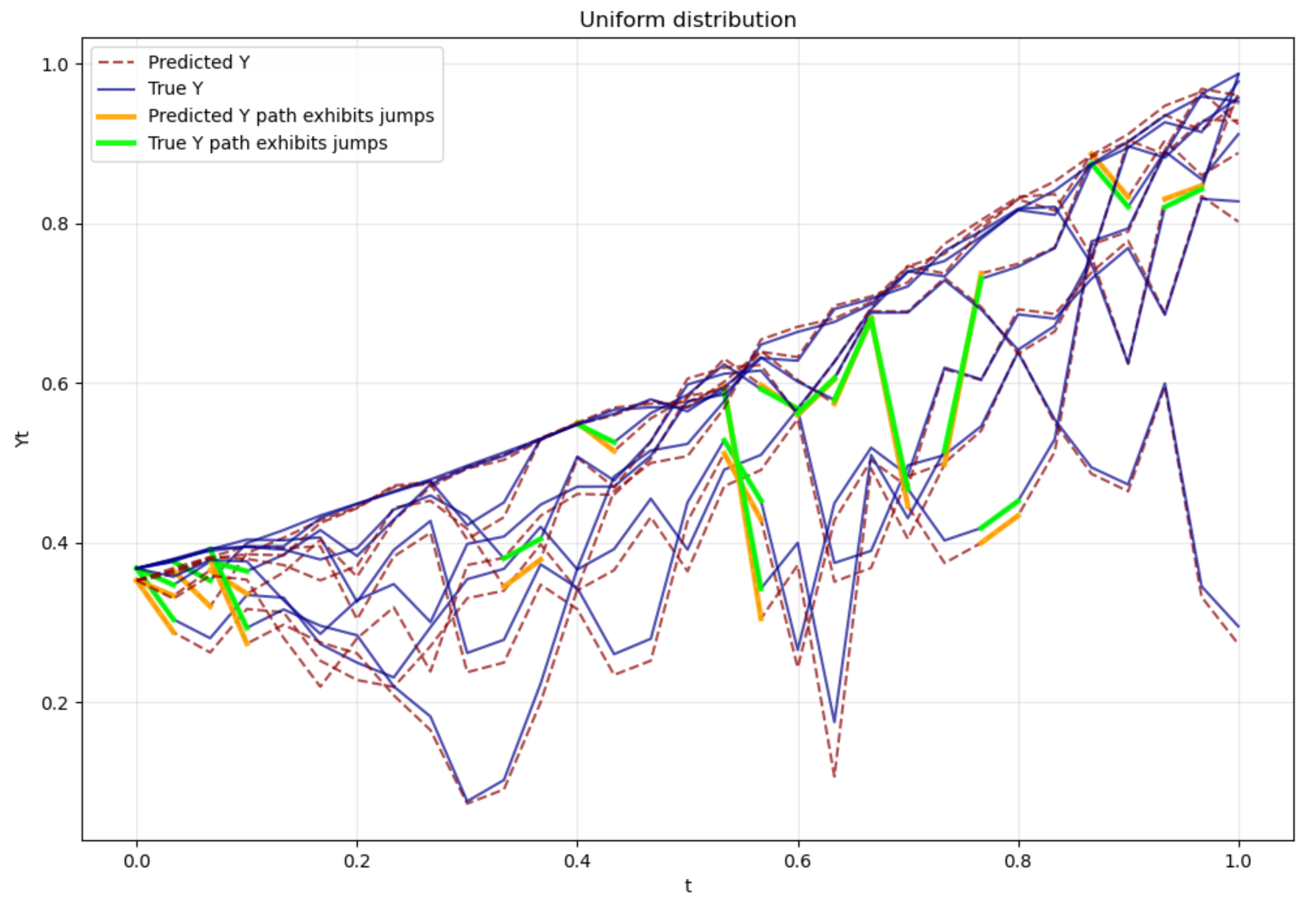}
        \subcaption{Uniform distribution}
    \end{minipage}
    \hfill
    \begin{minipage}{0.47\linewidth}
        \centering
        \includegraphics[width=\linewidth]{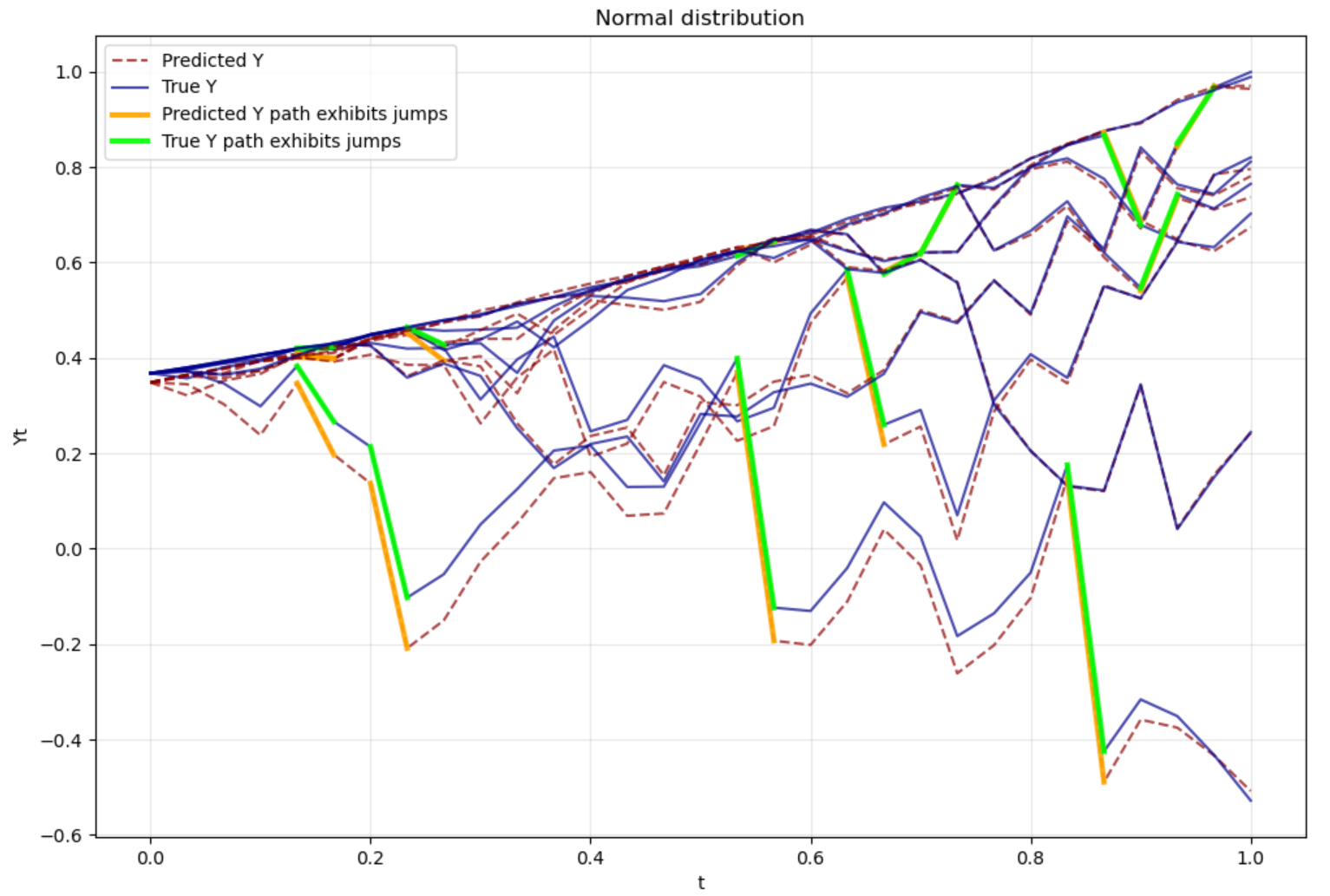}
        \subcaption{Normal distribution}
    \end{minipage}
    \begin{minipage}{0.47\linewidth}
        \centering
        \includegraphics[width=\linewidth]{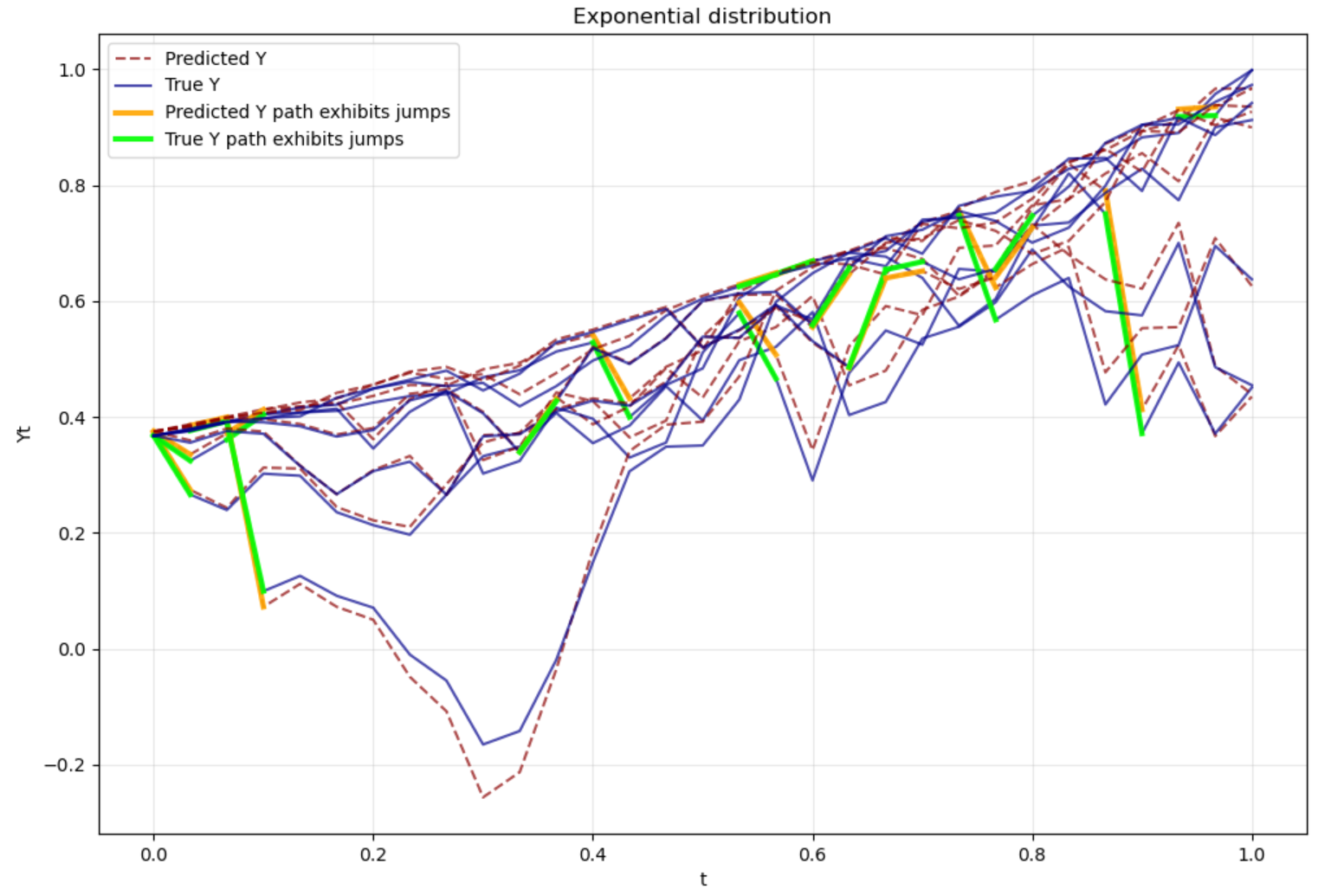}
        \subcaption{Exponential distribution}
    \end{minipage}
    \hfill
    \begin{minipage}{0.47\linewidth}
        \centering
        \includegraphics[width=\linewidth]{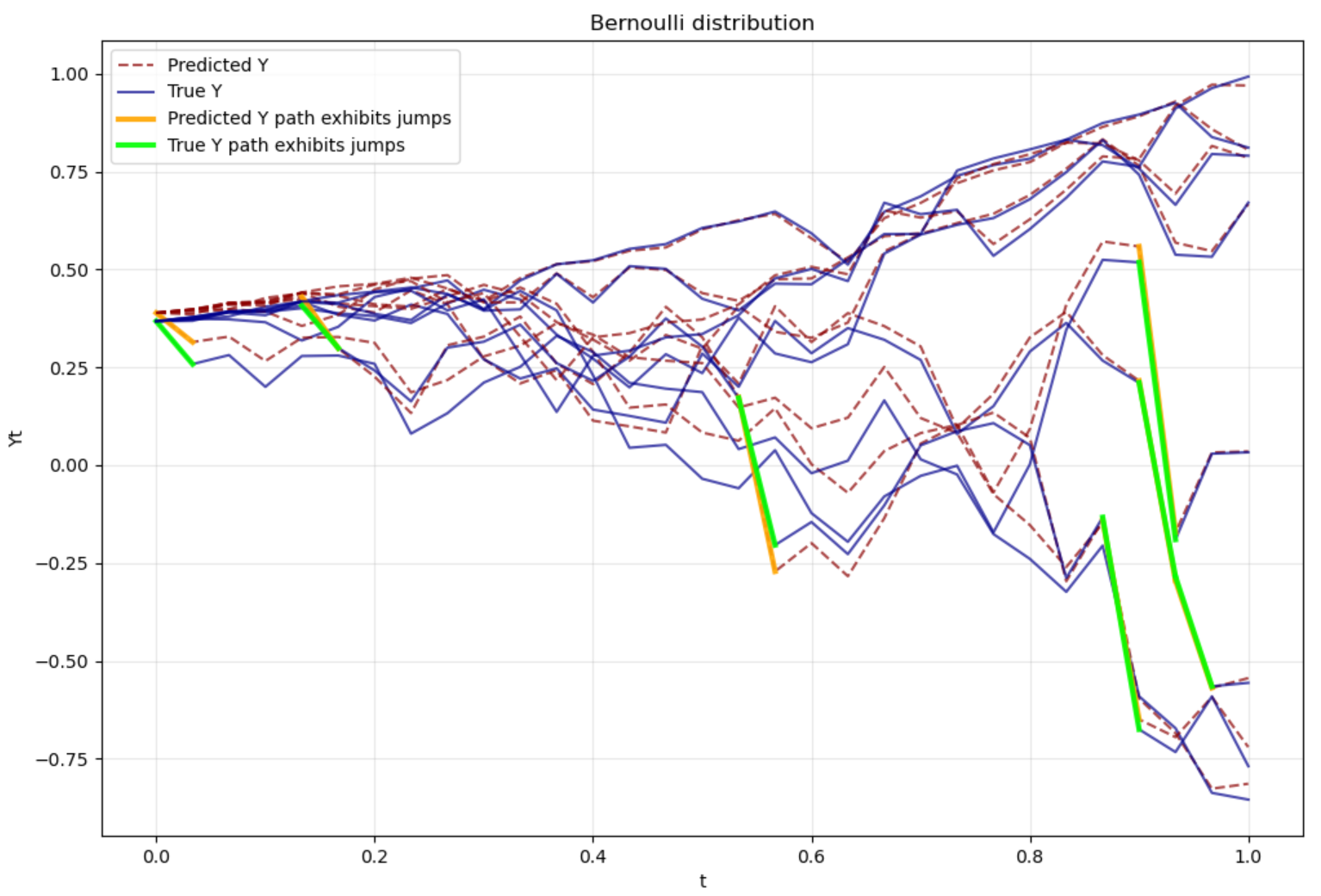}
        \subcaption{Bernoulli distribution}
    \end{minipage}
    \caption{Estimates of $Y_t$ using DFBDP algorithm under different Levy measures}
    \label{Figure4.1}
\end{figure}

%%%%%%%%%%%%%%%%%%%%%%%%%%%绘制均匀分布下u(t,x),Du(t,x)的近似图%%%%%%%%%%%%%%%%%%%%%%%%%%%%%
\begin{figure}[htb]
    \centering
    % 第一行图
    \begin{minipage}{0.23\linewidth}
        \centering
        \includegraphics[width=\linewidth]{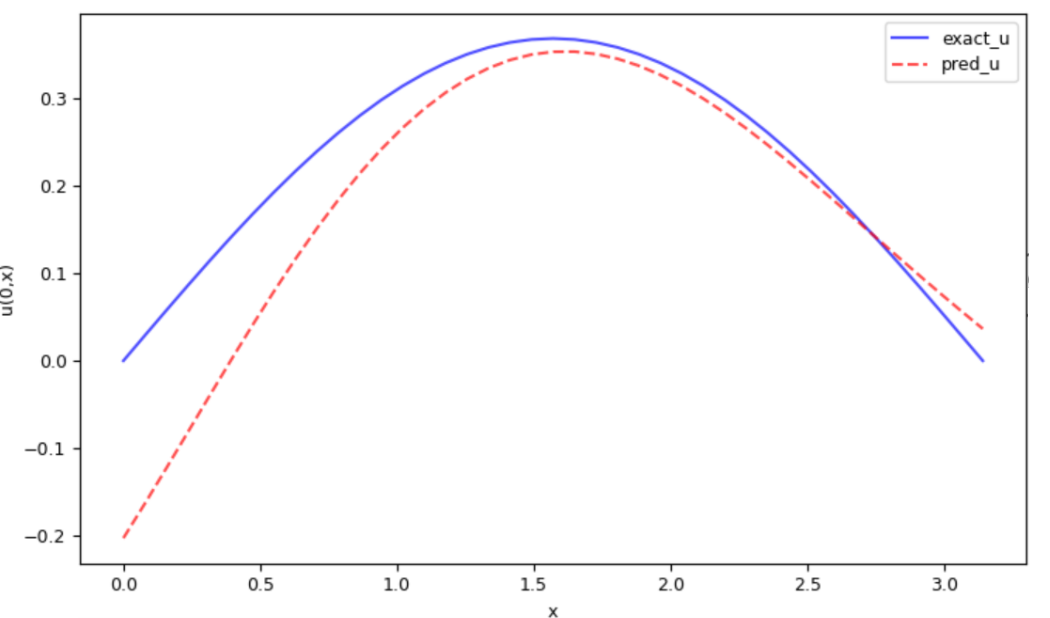}
        \subcaption{$u(0,x)$} % 使用 \footnotesize 减小字体
    \end{minipage}
    \hfill
    \begin{minipage}{0.23\linewidth}
        \centering
        \includegraphics[width=\linewidth]{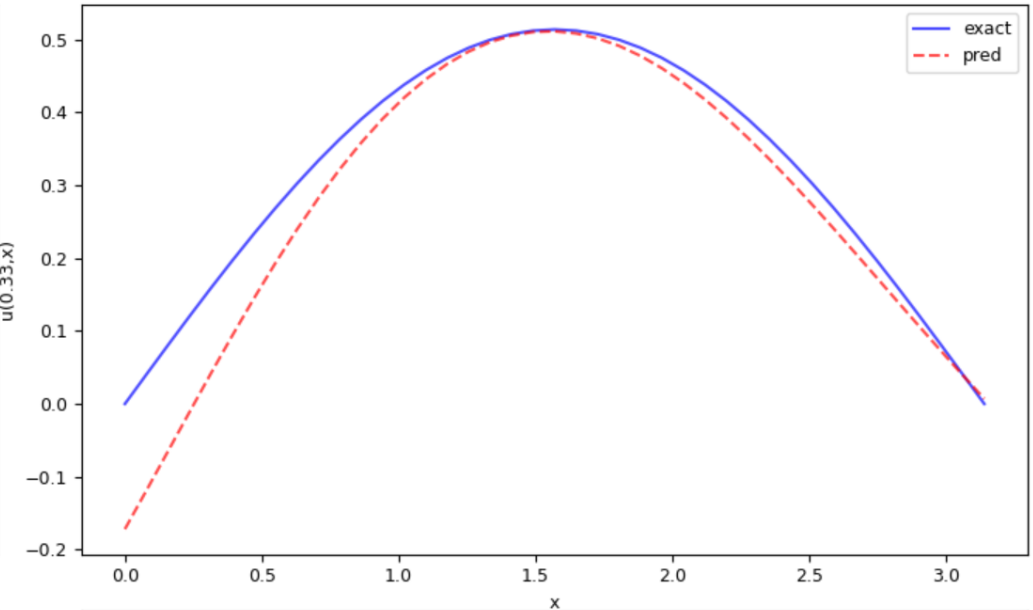}
        \subcaption{$u(0.33,x)$}
    \end{minipage}
    \hfill
    \begin{minipage}{0.23\linewidth}
        \centering
        \includegraphics[width=\linewidth]{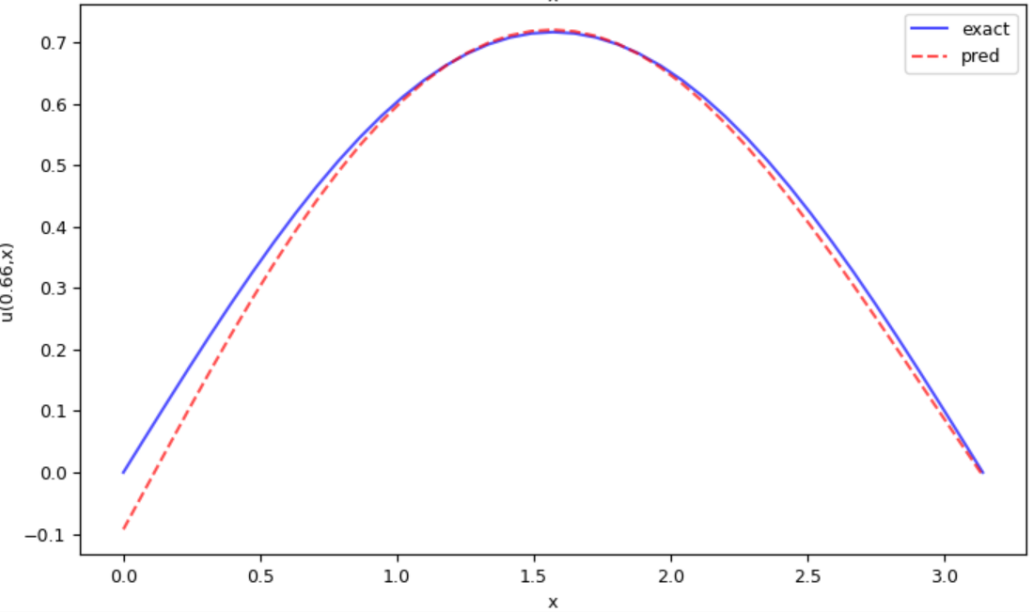}
        \subcaption{$u(0.66,x)$}
    \end{minipage}
    \hfill
    \begin{minipage}{0.23\linewidth}
        \centering
        \includegraphics[width=\linewidth]{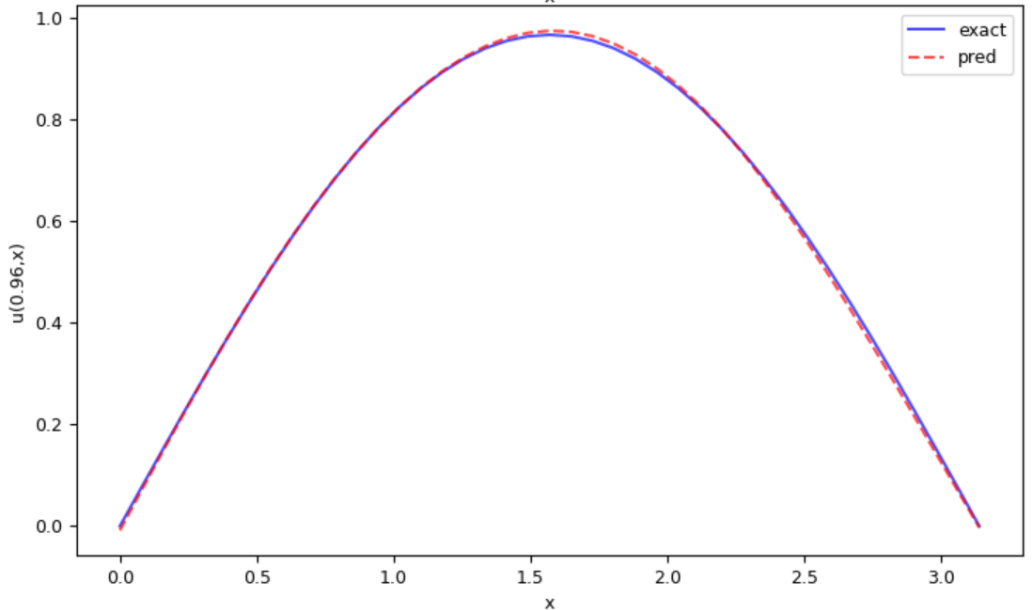}
        \subcaption{$u(0.96,x)$}
    \end{minipage}

    % 增加行间距
    \vspace{0.8cm}

    % 第二行图
    \begin{minipage}{0.23\linewidth}
        \centering
        \includegraphics[width=\linewidth]{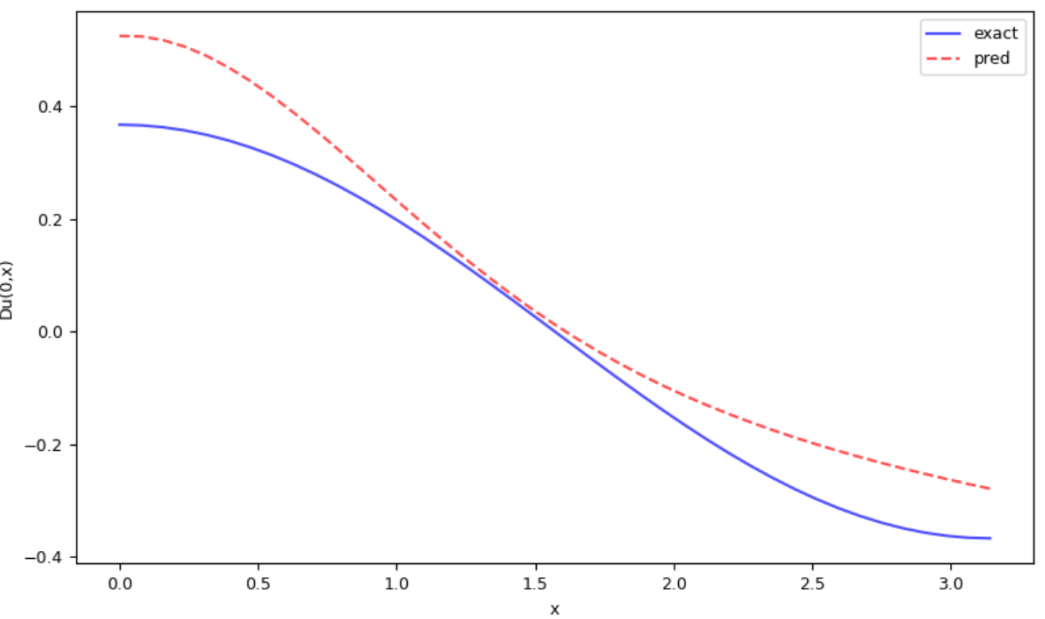}
        \subcaption{$D_{x}u(0,x)$}
    \end{minipage}
    \hfill
    \begin{minipage}{0.23\linewidth}
        \centering
        \includegraphics[width=\linewidth]{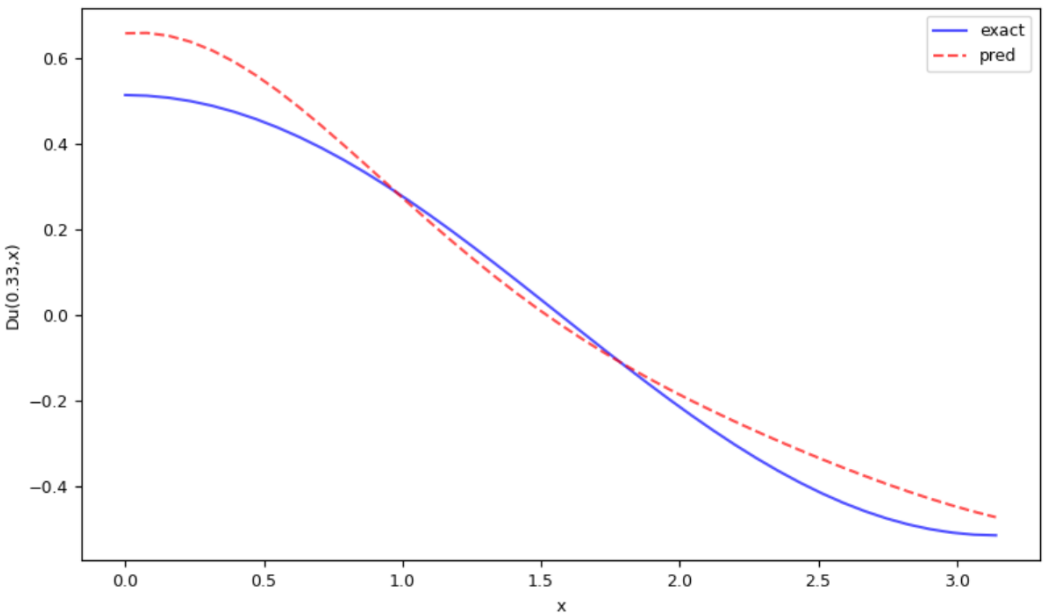}
        \subcaption{$D_{x}u(0.33,x)$}
    \end{minipage}
    \hfill
    \begin{minipage}{0.23\linewidth}
        \centering
        \includegraphics[width=\linewidth]{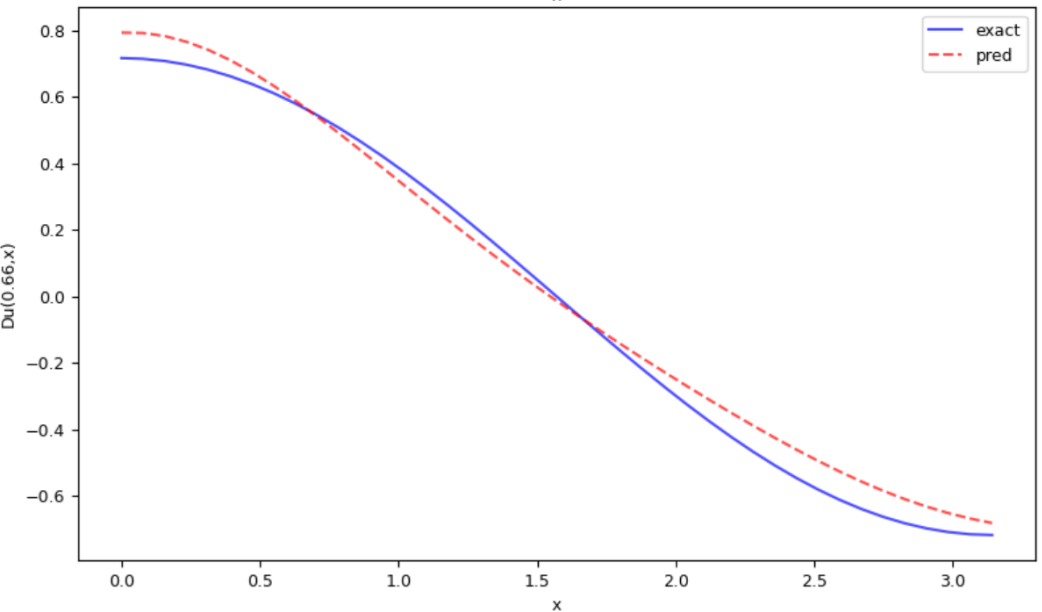}
        \subcaption{$D_{x}u(0.66,x)$}
    \end{minipage}
    \hfill
    \begin{minipage}{0.23\linewidth}
        \centering
        \includegraphics[width=\linewidth]{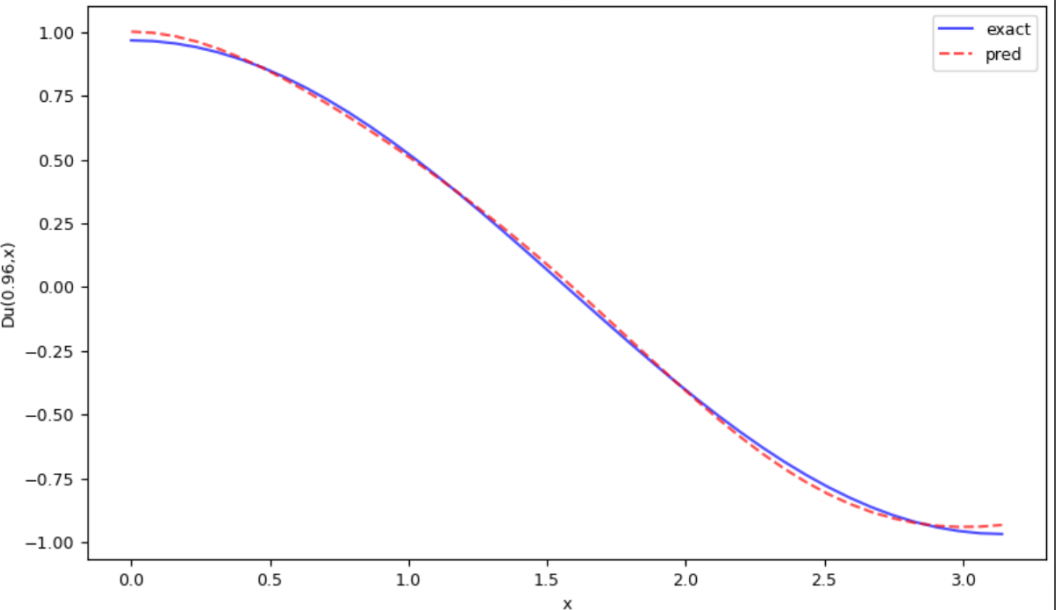}
        \subcaption{$D_{x}u(0.96,x)$}
    \end{minipage}

    \caption{Estimates of $u(t,x)$ at $t=0,0.33,0.66,0.96$ (A,B,C,D) and estimates of $D_{x}u(t,x)$ at $t=0,0.33,0.66,0.96$ (E,F,G,H) under Uniform distribution Lévy measure}
    \label{Figure4.3}
\end{figure}

%%%%%%%%%%%%%%%%%%%%%%%%%%%绘制正态分布下u(t,x),Du(t,x)的近似图%%%%%%%%%%%%%%%%%%%%%%%%%%%%%

\begin{figure}[htb]
    \centering
    \begin{minipage}{0.23\linewidth}
        \centering
        \includegraphics[width=\linewidth]{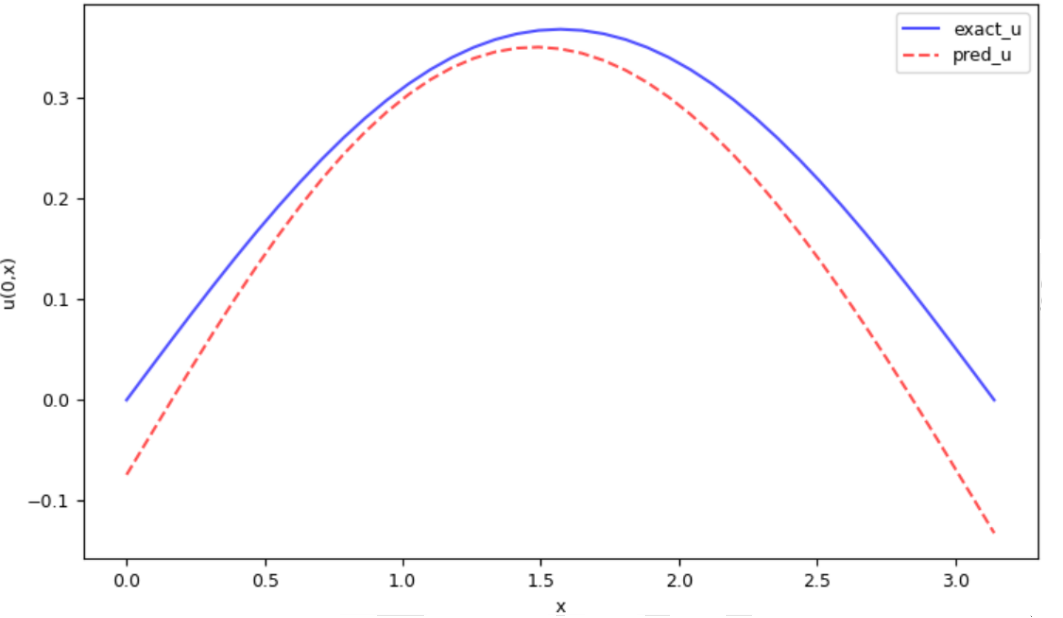}
        \subcaption{$u(0,x)$} % 添加子图标题
    \end{minipage}
    \hfill
    \begin{minipage}{0.23\linewidth}
        \centering
        \includegraphics[width=\linewidth]{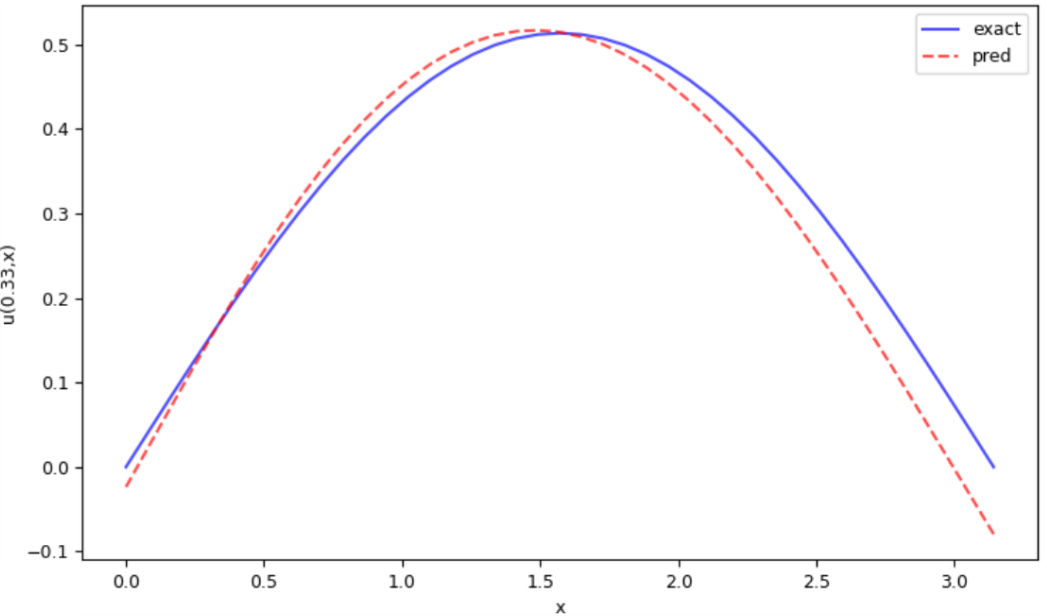}
        \subcaption{$u(0.33,x)$} % 添加子图标题
    \end{minipage}
    \begin{minipage}{0.23\linewidth}
        \centering
        \includegraphics[width=\linewidth]{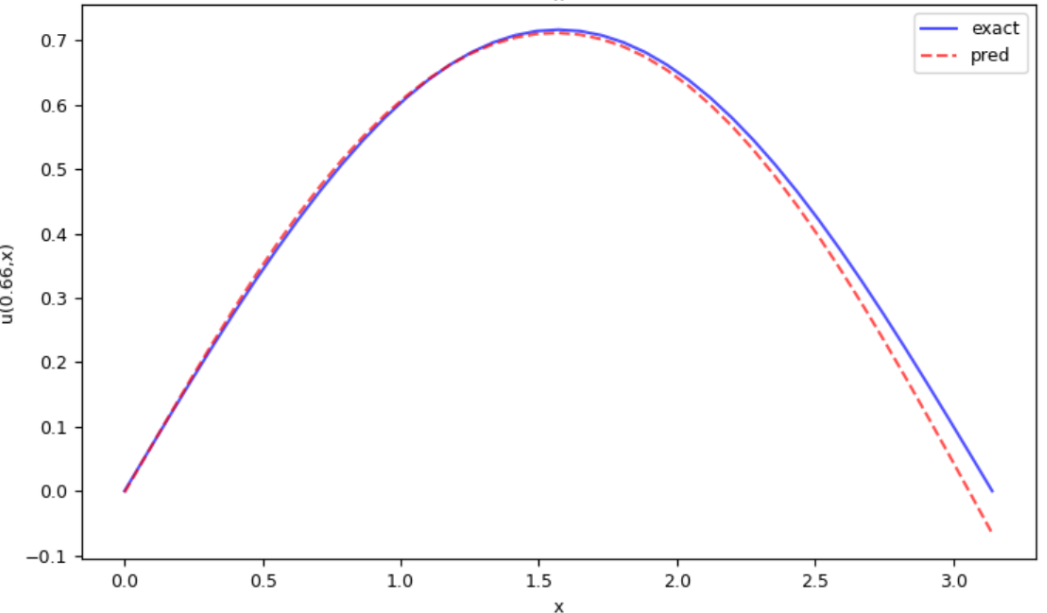}
        \subcaption{$u(0.66,x)$} % 添加子图标题
    \end{minipage}
    \hfill
    \begin{minipage}{0.23\linewidth}
        \centering
        \includegraphics[width=\linewidth]{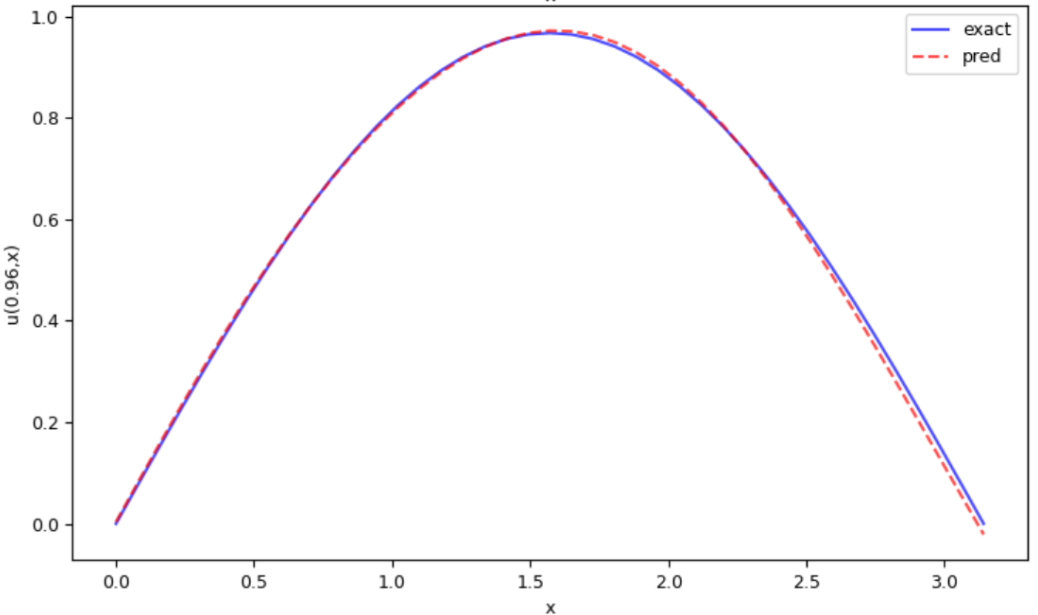}
        \subcaption{$u(0.96,x)$} % 添加子图标题
    \end{minipage}

    % 增加行间距
    \vspace{0.8cm}

    \begin{minipage}{0.23\linewidth}
        \centering
        \includegraphics[width=\linewidth]{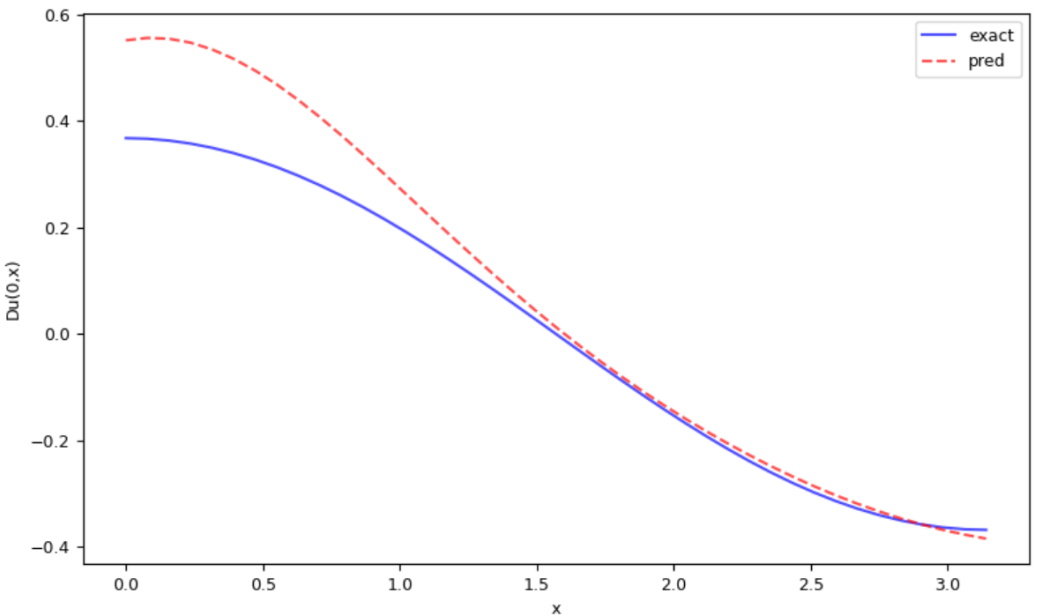}
        \subcaption{$D_{x}u(0,x)$} % 添加子图标题
    \end{minipage}
    \hfill
    \begin{minipage}{0.23\linewidth}
        \centering
        \includegraphics[width=\linewidth]{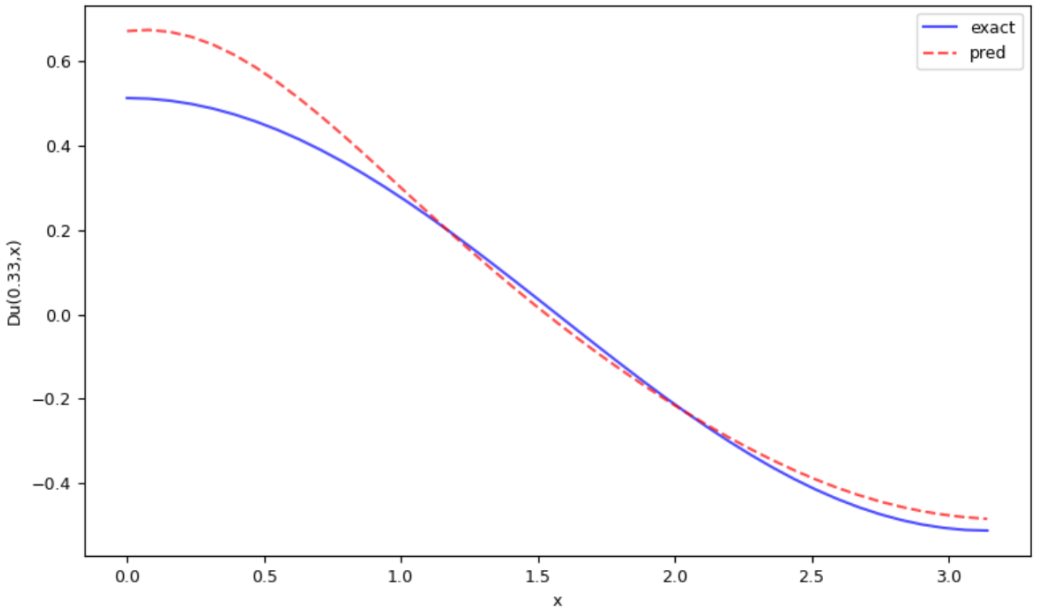}
        \subcaption{$D_{x}u(0.33,x)$} % 添加子图标题
    \end{minipage}
    \begin{minipage}{0.23\linewidth}
        \centering
        \includegraphics[width=\linewidth]{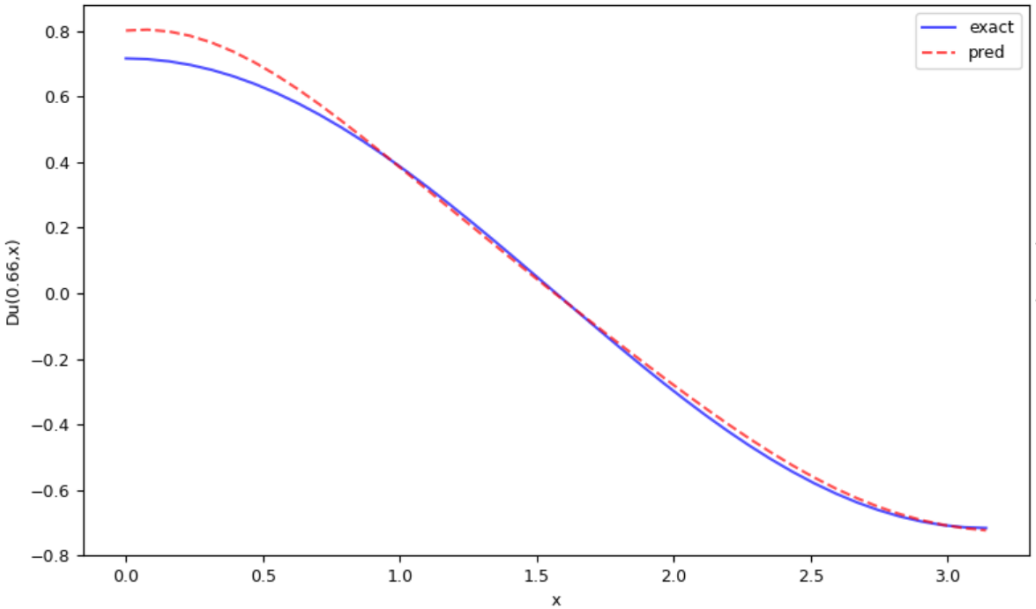}
        \subcaption{$D_{x}u(0.66,x)$} % 添加子图标题
    \end{minipage}
    \hfill
    \begin{minipage}{0.23\linewidth}
        \centering
        \includegraphics[width=\linewidth]{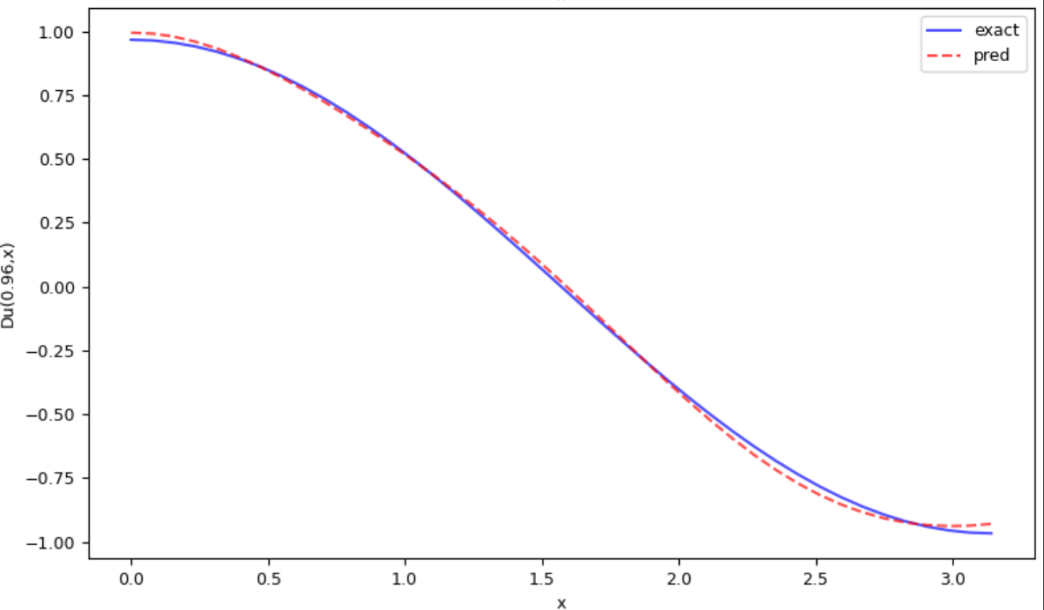}
        \subcaption{$D_{x}u(0.96,x)$} % 添加子图标题
    \end{minipage}
    \caption{Estimates of $u(t,x)$ at $t=0,0.33,0.66,0.96$ (A,B,C,D) and estimates of $D_{x}u(t,x)$ at $t=0,0.33,0.66,0.96$ (E,F,G,H) under Normal distribution Lévy measure}
    \label{Figure4.5}
\end{figure}

%%%%%%%%%%%%%%%%%%%%%%%%%%%绘制指数分布下u(t,x),Du(t,x)的近似图%%%%%%%%%%%%%%%%%%%%%%%%%%%%%

\begin{figure}[htb]
    \centering
    \begin{minipage}{0.23\linewidth}
        \centering
        \includegraphics[width=\linewidth]{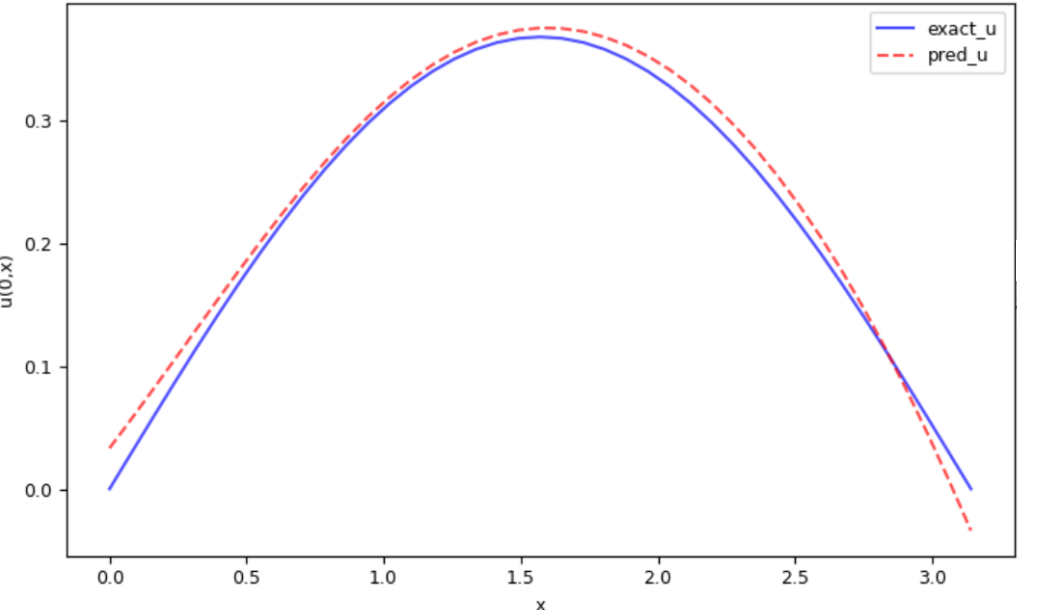}
        \subcaption{$u(0,x)$} % 添加子图标题
    \end{minipage}
    \hfill
    \begin{minipage}{0.23\linewidth}
        \centering
        \includegraphics[width=\linewidth]{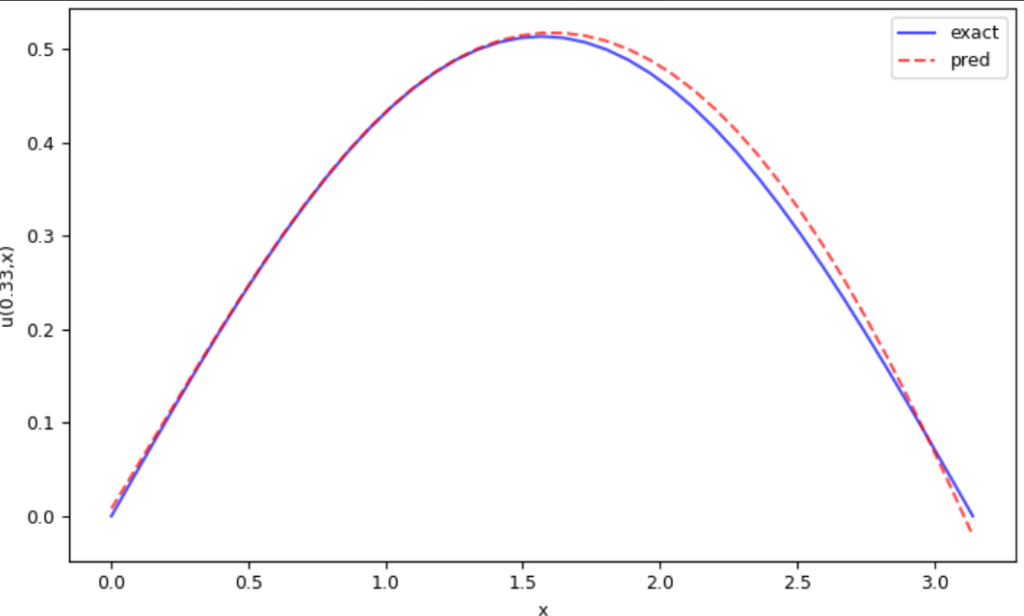}
        \subcaption{$u(0.33,x)$} % 添加子图标题
    \end{minipage}
    \begin{minipage}{0.23\linewidth}
        \centering
        \includegraphics[width=\linewidth]{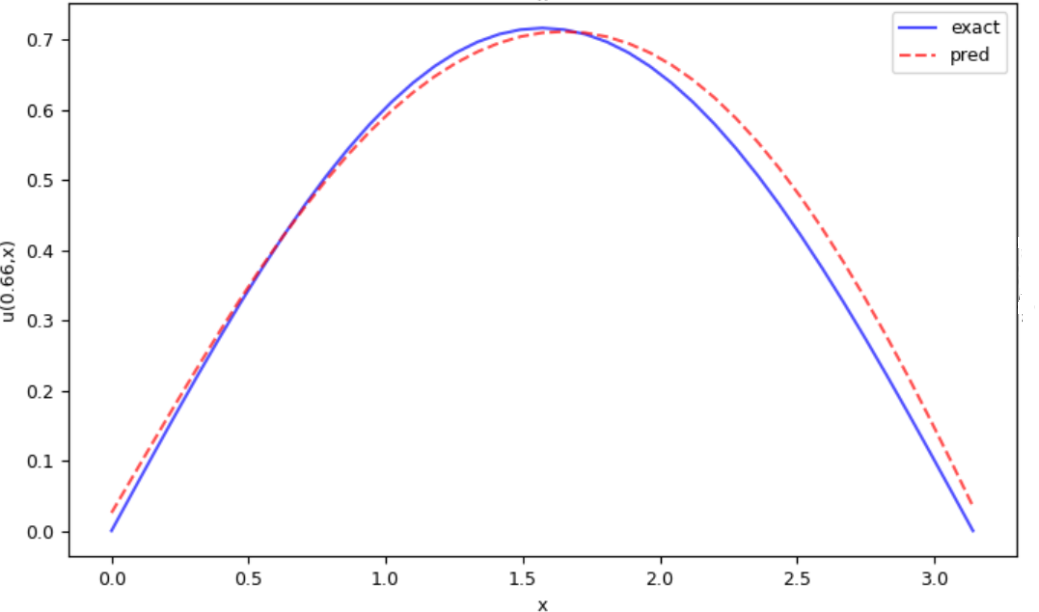}
        \subcaption{$u(0.66,x)$} % 添加子图标题
    \end{minipage}
    \hfill
    \begin{minipage}{0.23\linewidth}
        \centering
        \includegraphics[width=\linewidth]{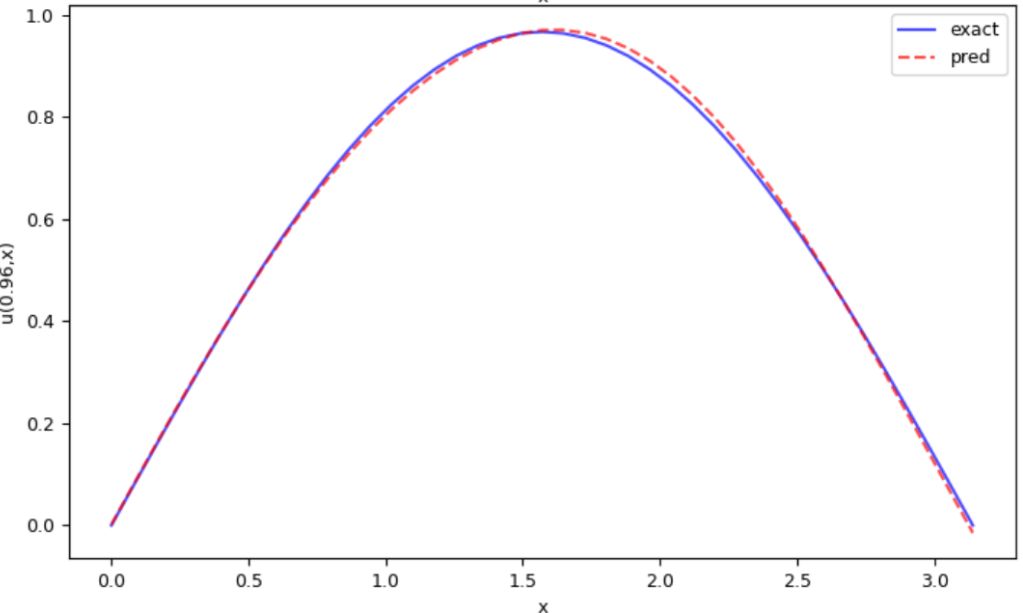}
        \subcaption{$u(0.96,x)$} % 添加子图标题
    \end{minipage}

    % 增加行间距
    \vspace{0.8cm}

    \begin{minipage}{0.23\linewidth}
        \centering
        \includegraphics[width=\linewidth]{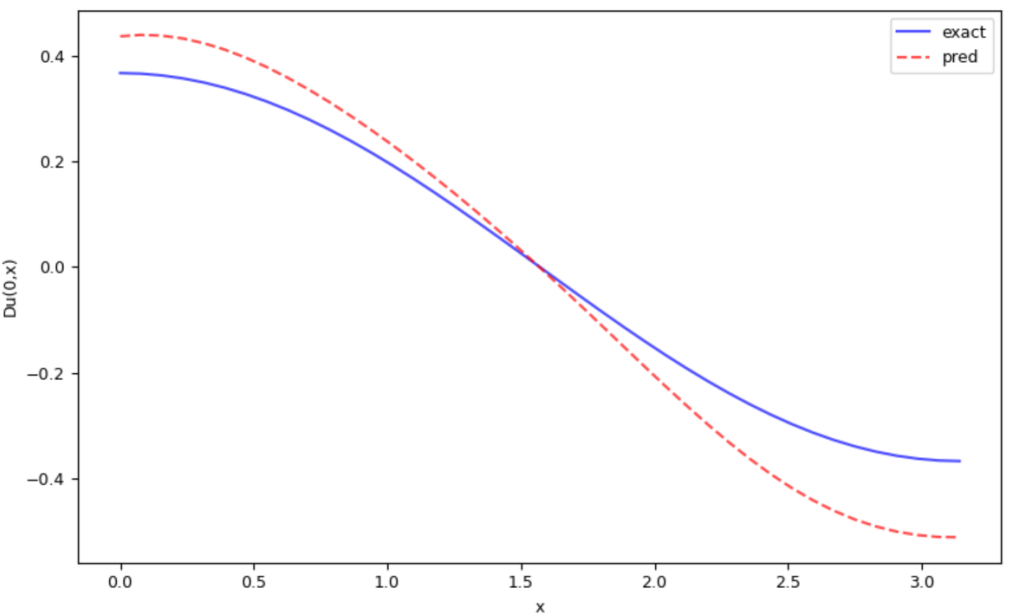}
        \subcaption{$D_{x}u(0,x)$} % 添加子图标题
    \end{minipage}
    \hfill
    \begin{minipage}{0.23\linewidth}
        \centering
        \includegraphics[width=\linewidth]{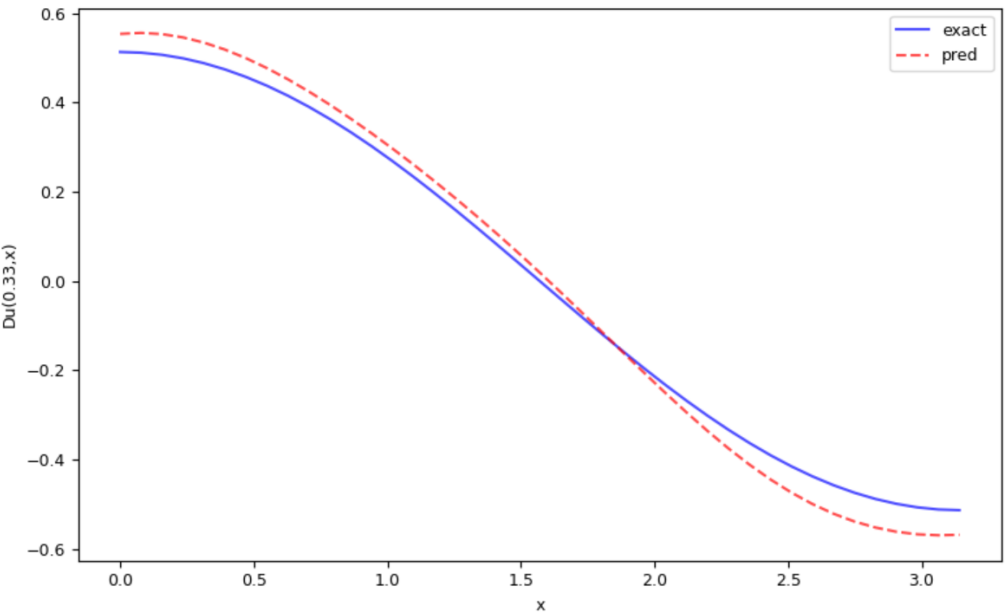}
        \subcaption{$D_{x}u(0.33,x)$} % 添加子图标题
    \end{minipage}
    \begin{minipage}{0.23\linewidth}
        \centering
        \includegraphics[width=\linewidth]{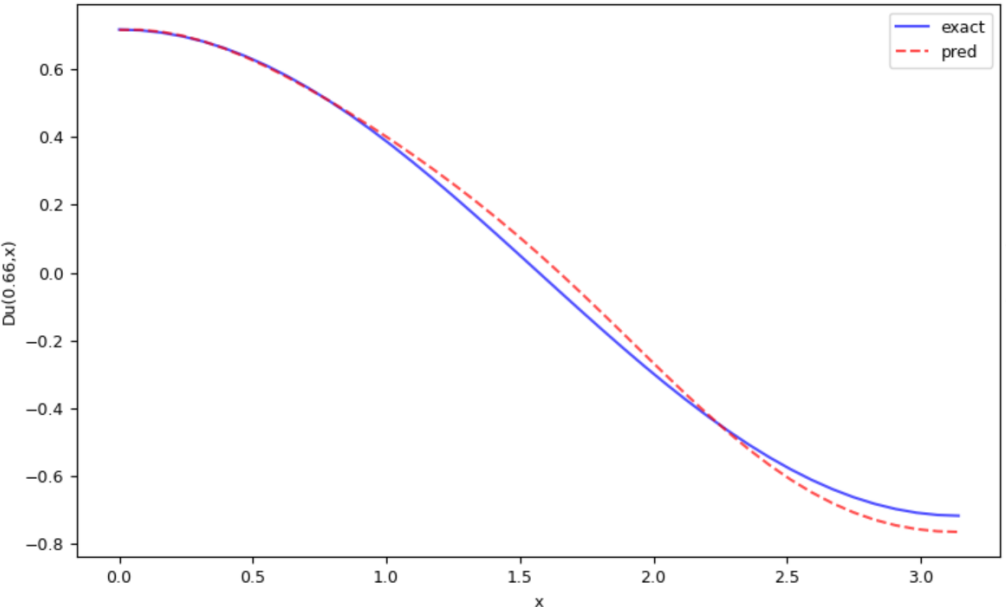}
        \subcaption{$D_{x}u(0.66,x)$} % 添加子图标题
    \end{minipage}
    \hfill
    \begin{minipage}{0.23\linewidth}
        \centering
        \includegraphics[width=\linewidth]{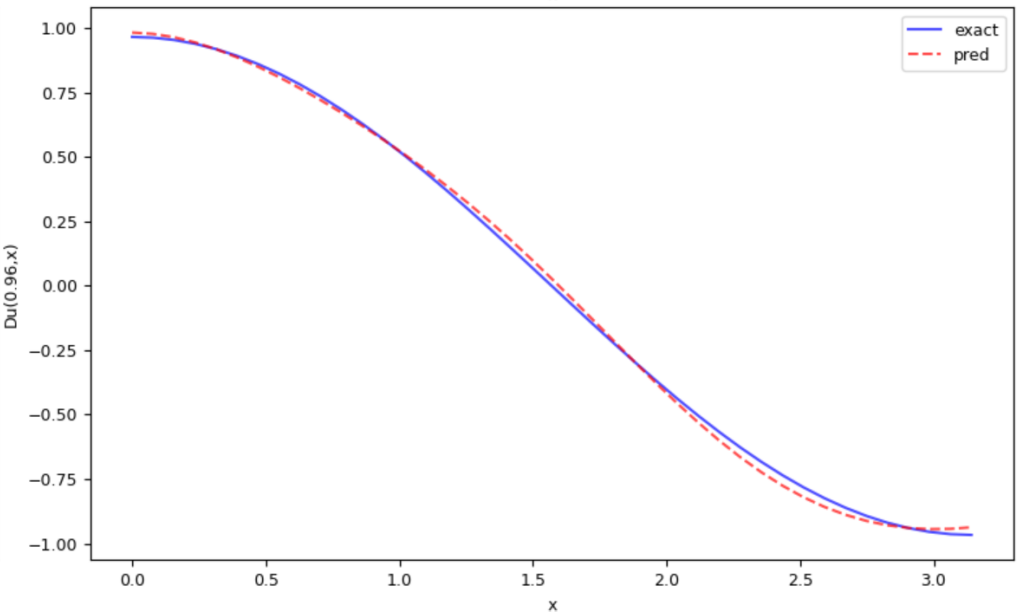}
        \subcaption{$D_{x}u(0.96,x)$} % 添加子图标题
    \end{minipage}
    \caption{Estimates of $u(t,x)$ at $t=0,0.33,0.66,0.96$ (A,B,C,D) and estimates of $D_{x}u(t,x)$ at $t=0,0.33,0.66,0.96$ (E,F,G,H) under Exponential distribution Lévy measure}
    \label{Figure4.7}
\end{figure}

%%%%%%%%%%%%%%%%%%%%%%%%%%%绘制伯努利分布下u(t,x),Du(t,x)的近似图%%%%%%%%%%%%%%%%%%%%%%%%%%%%%

\begin{figure}[htb]
    \centering
    \begin{minipage}{0.23\linewidth}
        \centering
        \includegraphics[width=\linewidth]{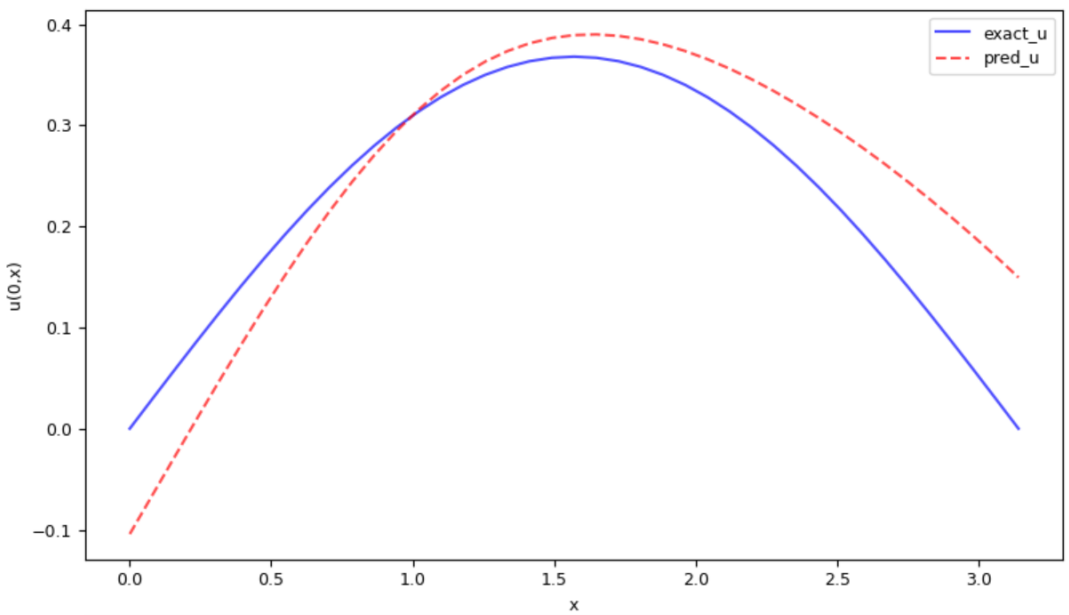}
        \subcaption{$u(0,x)$} % 添加子图标题
    \end{minipage}
    \hfill
    \begin{minipage}{0.23\linewidth}
        \centering
        \includegraphics[width=\linewidth]{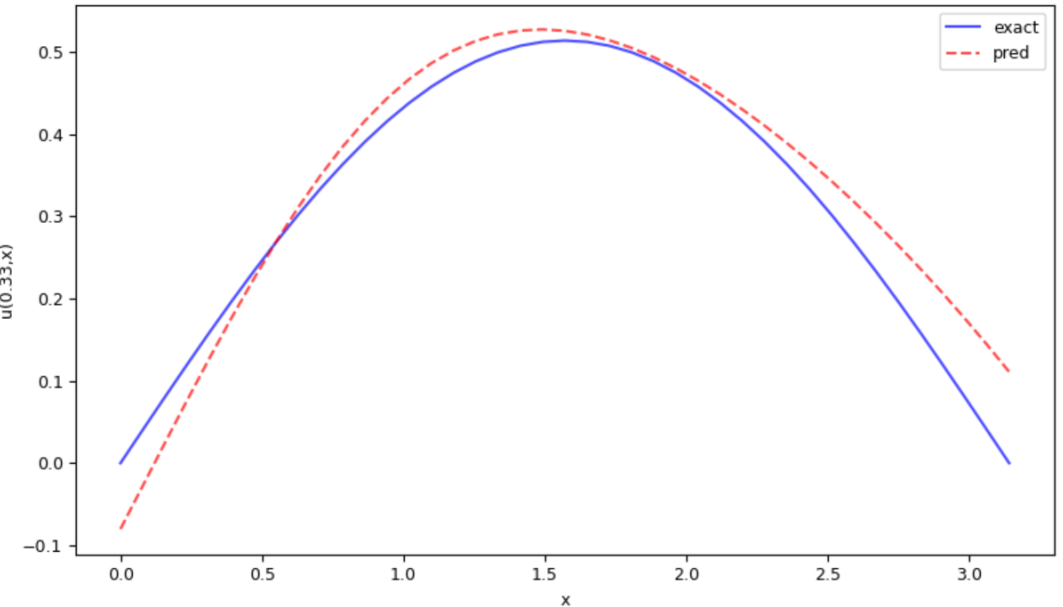}
        \subcaption{$u(0.33,x)$} % 添加子图标题
    \end{minipage}
    \begin{minipage}{0.23\linewidth}
        \centering
        \includegraphics[width=\linewidth]{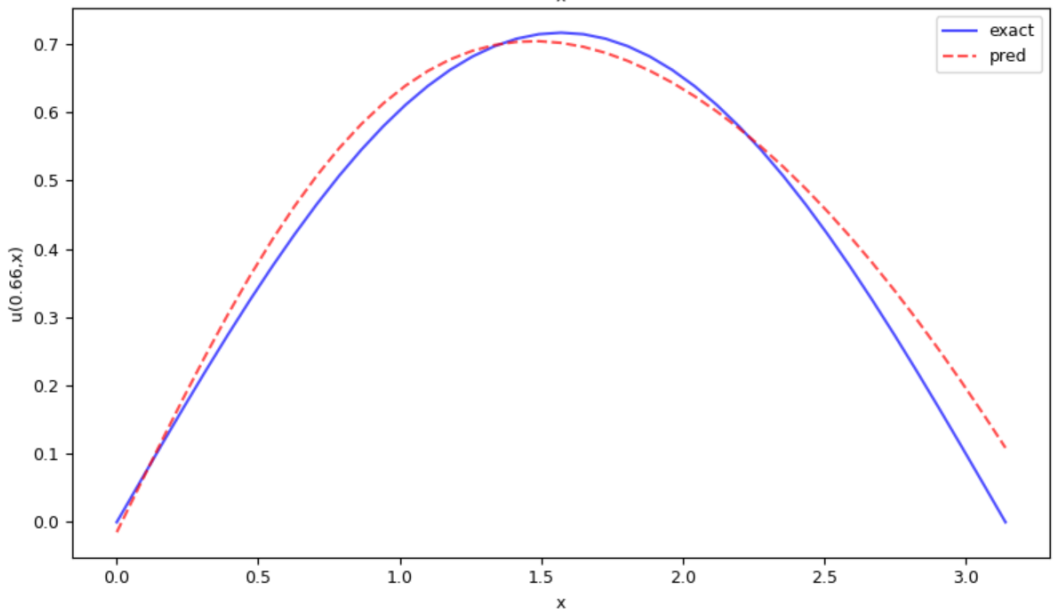}
        \subcaption{$u(0.66,x)$} % 添加子图标题
    \end{minipage}
    \hfill
    \begin{minipage}{0.23\linewidth}
        \centering
        \includegraphics[width=\linewidth]{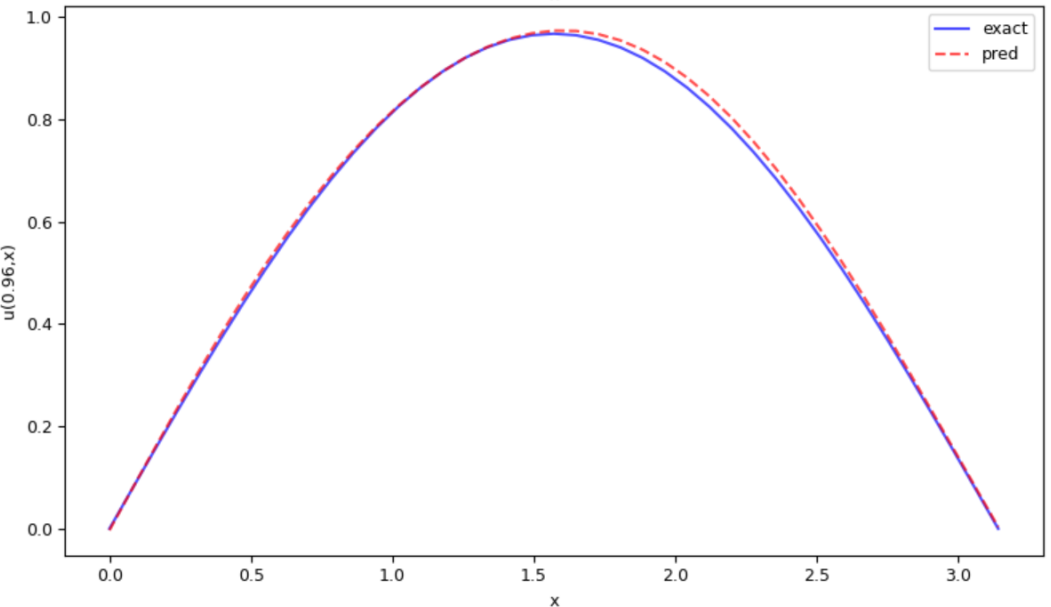}
        \subcaption{$u(0.96,x)$} % 添加子图标题
    \end{minipage}

    % 增加行间距
    \vspace{0.8cm}

    \begin{minipage}{0.23\linewidth}
        \centering
        \includegraphics[width=\linewidth]{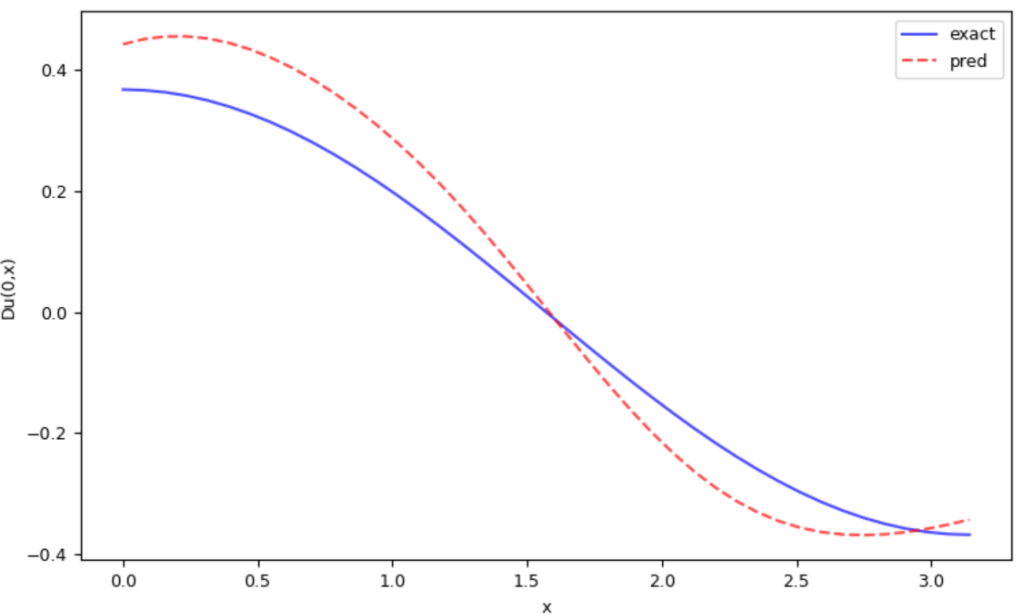}
        \subcaption{$D_{x}u(0,x)$} % 添加子图标题
    \end{minipage}
    \hfill
    \begin{minipage}{0.23\linewidth}
        \centering
        \includegraphics[width=\linewidth]{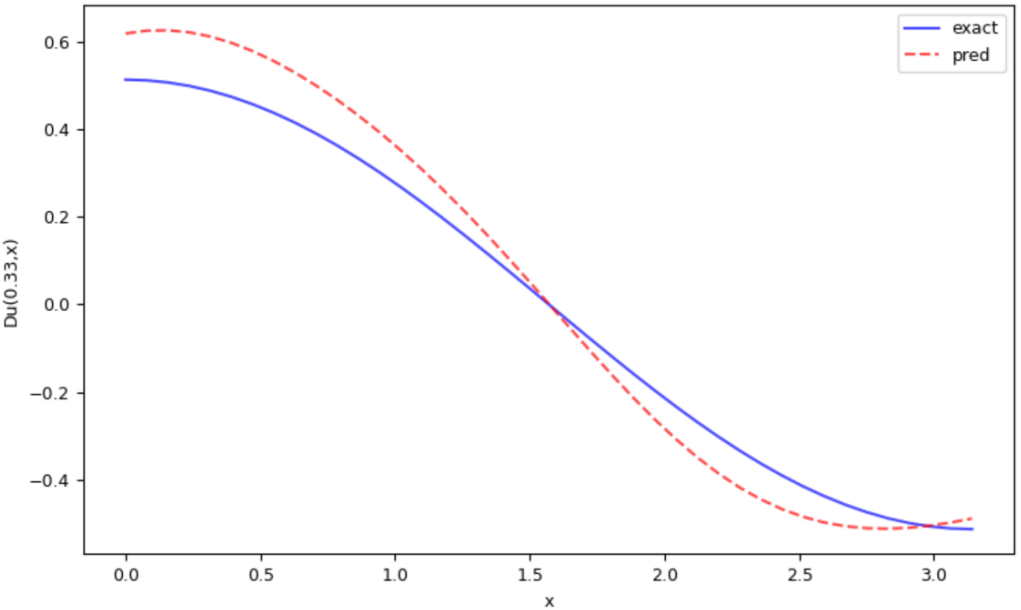}
        \subcaption{$D_{x}u(0.33,x)$} % 添加子图标题
    \end{minipage}
    \begin{minipage}{0.23\linewidth}
        \centering
        \includegraphics[width=\linewidth]{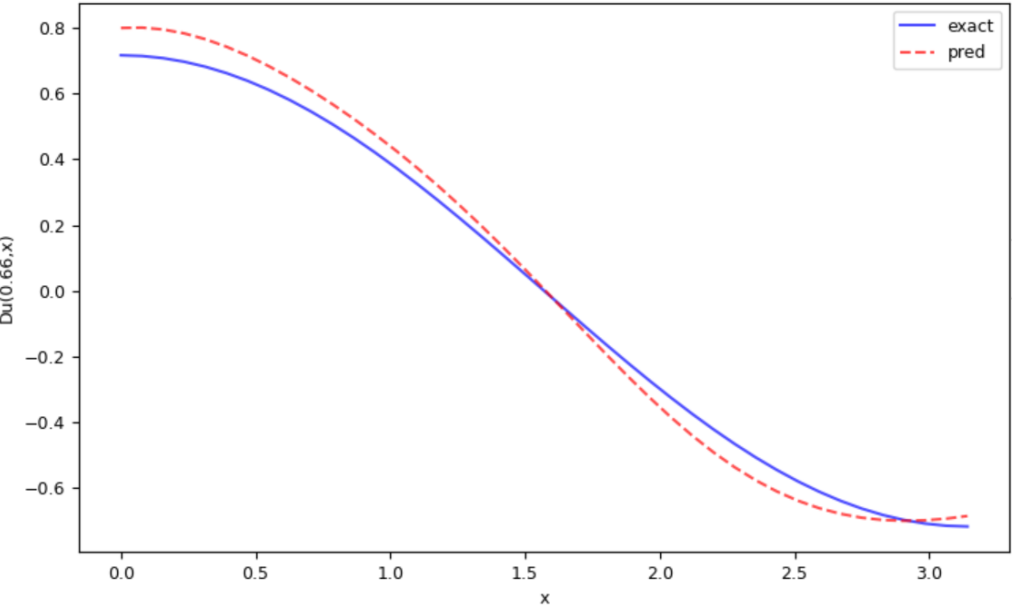}
        \subcaption{$D_{x}u(0.66,x)$} % 添加子图标题
    \end{minipage}
    \hfill
    \begin{minipage}{0.23\linewidth}
        \centering
        \includegraphics[width=\linewidth]{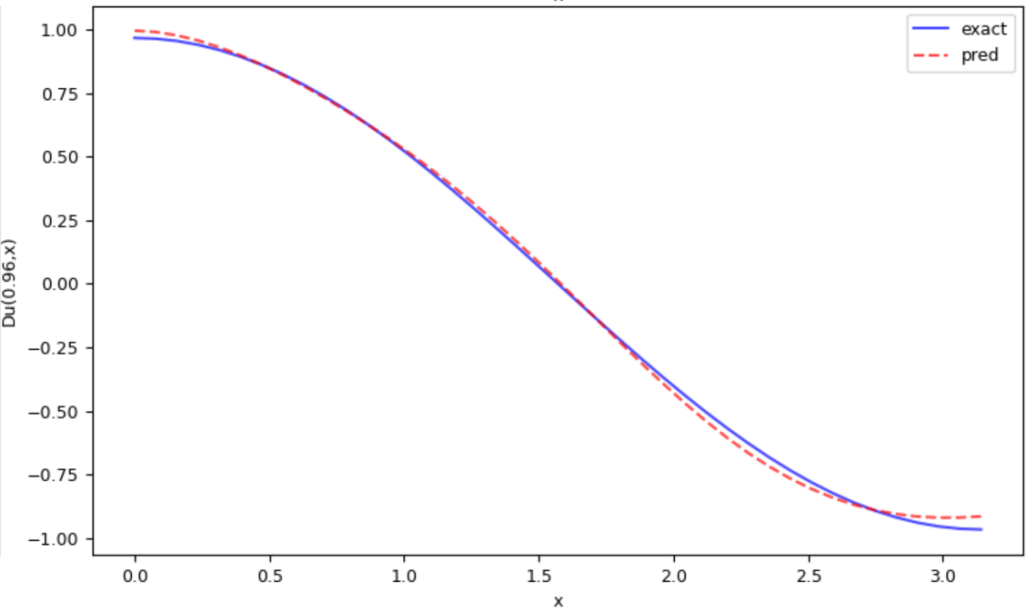}
        \subcaption{$D_{x}u(0.96,x)$} % 添加子图标题
    \end{minipage}
    \caption{Estimates of $u(t,x)$ at $t=0,0.33,0.66,0.96$ (A,B,C,D) and estimates of $D_{x}u(t,x)$ at $t=0,0.33,0.66,0.96$ (E,F,G,H) under Bernoulli distribution Lévy measure}
    \label{Figure4.9}
\end{figure}

%%%%%%%%%%%%%%%%%%%%%%%%%%%%%%%%%%%%验证高维情形%%%%%%%%%%%%%%%%%%%%%%%%%%%%%%%%%%%%

\subsection{High-dimensional problems}  In the second example, we consider a high-dimensional partial integro-differential equations (PIDEs) problem, {based on Equations (3.9) in \cite{Lu2024}}:
\begin{equation}\label{4.3}
\begin{cases}
\begin{aligned}
&\frac{\partial u}{\partial t}(t,x) + \frac{\epsilon}{2}x \cdot \nabla u(t,x) + \frac{1}{2}\mathrm{Tr}(\theta^2\mathrm{H}(u))-\lambda\mu^2 - \theta^2 \\
&\quad + \int_{E} \left( u(t,x+e) - u(t,x) - e \cdot \nabla u(t,x) \right) \lambda(\mathrm{d}e) = 0 ,
\end{aligned}\\
\begin{aligned}
u(T,x) = \frac{1}{d}\|x\|^2.\end{aligned}
\end{cases}
\end{equation}
The corresponding FBSDEJ is
\begin{equation}\label{4.4}
\begin{cases}
\begin{aligned}
\mathrm{d}X_{t}=\theta \mathrm{d}W_{t}+\int_{E}e \widetilde{\mu}(\mathrm{d}t,\mathrm{d}e), \end{aligned}
\\ \begin{aligned}
    -\mathrm{d}Y_{t}=-(\lambda \mu^{2}+\theta^{2})\mathrm{d}t-Z_{t}^{T}\mathrm{d}W_{t}-\int_{E}U_{t}(e)\widetilde{\mu}(\mathrm{d}t,\mathrm{d}e), \end{aligned}
\end{cases}
\end{equation}
with $\theta = 0.3$, $x \in \mathbb{R}^d$ and Poisson intensity $\lambda = 0.3$. Different from [75], to highlight the effectiveness of our method in handling high-dimensional problems, we employ a constant jump form $(0.1, 0.1, \cdots, 0.1)$, which means that the jump size at each coordinate is always $\mu=0.1$. The equation possesses an exact solution of $u(t, x) = \frac{1}{d} \|x\|^2$. The training process utilizes $M = 1000$ trajectories, and the time $[0, 1]$ is divided into $N = 60$ intervals. Furthermore, the number of neurons in the linear layers of each block is adjusted to $d + 10$, to accommodate the varying input dimensions $d$. The $L^{1}$-relative errors of $Y_0=u(0,X_0)=u(0,1,\cdots,1)$ is shown in Table 2. It is observed from Table 2 that the $L^{1}$-relative errors of \( u(0,X_0) \) obtained by the DFBDP algorithm are uniformly small for different dimensions $d$ and do not depend on the dimension $d$. This illustrates that the DFBDP algorithm can effectively solve high-dimensional semilinear PIDEs problems.

\begin{table}[h!]
\centering
\caption{The $L^{1}$-relative error of \( u(0,X_0) \) for different dimensions \(d\). }
\label{tab:4.2}
\begin{tabular}{cccccc}
\toprule
Dimension & 2 & 4 & 6 & 8 & 10 \\
\midrule
\( Y_0 \) relative error & 0.754\% & 1.251\% & 1.625\% & 2.001\% & 2.344\% \\

\midrule
Dimension & 20 & 30 & 40 & 50  \\
\midrule
\( Y_0 \) relative error & 1.503\% & 1.212\% & 2.435\% & 2.547\%  \\

\bottomrule
\end{tabular}
\end{table}

%%%%%%%%%%%%%%%%%%%%%%%%%%%验证二阶混合偏导数的系数不全为零的情形%%%%%%%%%%%%%%%%%%%%%%%%%%%%

We observe that the coefficients of the second-order mixed partial derivatives ($\frac{\partial ^{2}u}{\partial x_i \partial x_j},~~i\neq j$) are all zero in \eqref{4.3}. For high-dimensional problems, it is also necessary to consider the case where they are not all zero. Then we consider the following high-dimensional PIDEs with the coupled diffusion term {as Equations (3.10) in \cite{Lu2024}}, but we set the drift term to be identically zero:

\begin{equation}\label{4.5}
    \begin{cases}
\begin{aligned}
&\frac{\partial u}{\partial t}+\frac{1}{2}\mathrm{Tr}(\sigma\sigma^{T}\mathrm{H}(u))-\lambda \mu^{2}-\frac{2d-1}{d}\theta^{2}\\
&+\int_{E}\left(u(t,x+e)-u(t,x)-e\nabla_{x}u(t,x)\right)\lambda(\mathrm{d}e)=0,
\end{aligned}\\
\begin{aligned}
    u(T,x)=\frac{1}{d}||x||^{2},
\end{aligned}
\end{cases}
\end{equation}
where
\begin{equation*}
\mathrm{Tr}(\sigma\sigma H(u))=\theta^{2}\left(\frac{\partial ^{2}u}{\partial x_{1}^2}+2\frac{\partial ^{2}u}{\partial x_{1}\partial x_{2}}+2\frac{\partial ^{2}u}{\partial x_{2}^2}+2\frac{\partial ^{2}u}{\partial x_{2}\partial x_{3}}+2\frac{\partial ^{2}u}{\partial x_{3}^2}+\cdots+2\frac{\partial ^{2}u}{\partial x_{d-1}\partial x_{d}}+2\frac{\partial ^{2}u}{\partial x_{d}^2}\right).
\end{equation*}

The corresponding FBSDEJ reads
\begin{equation}\label{4.6}
\begin{cases}
\begin{aligned}
\mathrm{d}X_{t}=\sigma \mathrm{d}W_{t}+\int_{E}e \widetilde{\mu}(\mathrm{d}t,\mathrm{d}e), \end{aligned}
\\ \begin{aligned}
    -\mathrm{d}Y_{t}=-\left(\lambda \mu^{2}+\frac{2d-1}{d}\theta^{2}\right)\mathrm{d}t-Z_{t}^{T}\mathrm{d}W_{t}-\int_{E}U_{t}(e)\widetilde{\mu}(\mathrm{d}t,\mathrm{d}e).
\end{aligned}
\end{cases}
\end{equation}
Here, we set $\theta=0.2$ and Poisson intensity $\lambda=0.3$, and the diffusion term is given by
\begin{equation*}
 \sigma = \theta
\begin{pmatrix}
1 & 0 & 0 & 0 & \cdots & 0 \\
1 & 1 & 0 & 0 & \cdots & 0 \\
0 & 1 & 1 & 0 & \cdots & 0 \\
0 & 0 & 1 & 1 & \cdots & 0 \\
\vdots & \vdots & \vdots & \vdots & \ddots & \vdots \\
0 & 0 & 0 & \cdots & 1 & 1
\end{pmatrix}.
\end{equation*}
The exact solution is $u(t,x)=\frac{1}{d}||x||^{2}$. In this numerical test, we use the same setups
as in the previous example, including the selection of the jump form, the discretization
of the time interval, the number of sample trajectories. We test the cases for dimensions $d = 5$ and $d=10$. The $L^{1}$-relative errors of $Y_0=u(0,X_0)$ is $0.009767$ and $0.023211$, which demonstrate that the DFBDP algorithm can also achieve the ideal results for semilinear high-dimensional PIDEs under the cases where the coefficients of the second-order mixed partial derivatives are not all zero.

%%%%%%%%%%%%%%%%%%%%%%%%%%%%%%%%%%%%%%总结%%%%%%%%%%%%%%%%%%%%%%%%%%%%%%%%%%%%%
\section{Conclusions}
In this work, we propose a new deep learning algorithm for solving high-dimensional { semilinear parabolic partial integro-differential equations (PIDEs)} and corresponding forward-backward stochastic differential equations with jumps (FBSDEJs). This novel algorithm can be viewed as an extension and generalization of the DBDP2 scheme and a dynamic programming version of the forwardbackward algorithm proposed recently for high-dimensional {semilinear PIDEs}, respectively. { Different from the DBDP2 scheme for semilinear PDEs, our algorithm approximate simultaneously the solution and the integral kernel by independent neural networks, while the gradient of the solution is approximated by numerical differential techniques.} The related error estimates for the integral kernel approximation play key roles in deriving error estimates for the novel algorithm. {Given the universal approximation capability of neural networks, we analyze the convergence and consistency of the DFBDP scheme and provide error estimates.}

Numerical experiments confirm our theoretical results and verify the effectiveness of the proposed algorithm. Moreover, this novel algorithm consistently delivers good results for PIDEs and corresponding FBSDEJs with different Lévy measures.

\section*{Code availability} https://github.com/yezaijun/DFBDP.git

%%%%%%%%%%%%%%%%%%%%%%%%%%%%%%%%%%%%%参考文献%%%%%%%%%%%%%%%%%%%%%%%%%%%%%%%%%%%%%%%

\end{document}